\documentclass[12pt, twoside]{article}
\usepackage{amsmath,amsthm,amssymb}
\usepackage{times}
\usepackage{enumerate}
\usepackage{epsfig,graphics,amscd}
\usepackage{mathtools}

\def\titlerunning#1{\gdef\titrun{#1}}
\makeatletter
\def\author#1{\gdef\autrun{\def\and{\unskip, }#1}\gdef\@author{#1}}
\def\address#1{{\def\and{\\\hspace*{18pt}}\renewcommand{\thefootnote}{}%
\footnote {#1}}%
\markboth{\autrun}{\titrun}}
\makeatother
\def\email#1{e-mail: #1}
\def\subjclass#1{{\renewcommand{\thefootnote}{}%
\footnote{\emph{Mathematics Subject Classification (2010):} #1}}}

\newtheorem{theorem}[subsection]{Theorem}
\newtheorem{proposition}[subsection]{Proposition}
\newtheorem{lemma}[subsection]{Lemma}

\theoremstyle{definition}
\newtheorem{definition}[subsection]{Definition}

\theoremstyle{remark}
\newtheorem{remark}[subsection]{Remark}
\newtheorem{example}{Example}

\theoremstyle{definition}

\numberwithin{equation}{section}

\renewcommand{\leq}{\leqslant}
\renewcommand{\geq}{\geqslant}
\newsavebox{\proofbox}
\savebox{\proofbox}{\begin{picture}(7,7)  \put(0,0){\framebox(7,7){}}\end{picture}}

\newcommand\E{\mathbb{E}}
\newcommand\Z{\mathbb{Z}}
\newcommand\R{\mathbb{R}}
\newcommand\T{\mathbf{T}}
\newcommand\C{\mathbb{C}}
\newcommand\N{\mathbf{N}}

\newcommand\D{\mathbb{D}}
\newcommand\G{\mathbf{G}}
\newcommand\HH{\mathbf{H}}

\newcommand\B{\mathbf{B}}
\newcommand\SL{\operatorname{SL}}
\newcommand\PSL{\operatorname{PSL}}
\newcommand\SO{\operatorname{SO}}
\newcommand\Suz{\operatorname{Suz}}

\newcommand\End{\operatorname{End}}

\newcommand\GL{\operatorname{GL}}
\newcommand\PGL{\operatorname{PGL}}
\newcommand\PSU{\operatorname{PSU}}
\newcommand\SU{\operatorname{SU}}

\newcommand\Mat{\operatorname{Mat}}

\newcommand\rk{\operatorname{rk}}

\newcommand\Supp{\operatorname{Supp}}

\newcommand\tr{\operatorname{tr}}
\newcommand\Id{\operatorname{Id}}
\newcommand\ch{\operatorname{char}}
\newcommand\Aut{\operatorname{Aut}}

\newcommand\Stab{\operatorname{Stab}}

\newcommand\Sp{\operatorname{Sp}}

\newcommand\w{{\operatorname{w}}}
\renewcommand\P{\mathbb{P}}
\newcommand\F{\mathbb{F}}
\newcommand\Q{\mathbb{Q}}
\newcommand\U{\mathbf{U}}
\newcommand\X{\mathbf{X}}
\renewcommand\SS{\mathbf{S}}

\renewcommand{\k}{\mathbf{k}}
\newcommand\eps{\varepsilon}
\newcommand\id{\operatorname{id}}
\def\endproof{\hfill{\usebox{\proofbox}}\vspace{11pt}}

\begin{document}


\baselineskip=17pt


\titlerunning{Expansion in finite simple groups}

\title{Expansion in finite simple groups of Lie type}

\author{Emmanuel Breuillard, Ben Green, Robert Guralnick, \and Terence Tao}

\date{}

\maketitle

\address{
E. Breuillard: Laboratoire de Math\'ematiques B\^atiment 425, Universit\'e Paris Sud 11 91405 Orsay FRANCE
address1; \email{emmanuel.breuillard@math.u-psud.fr};
B. Green: Mathematical Institute, Radcliffe Observatory Quarter, Woodstock Road, Oxford OX2 6GG,
England; \email{ben.green@maths.ox.ac.uk};
R. Guralnick: Department of Mathematics, University of Southern California, 3620 Vermont Avenue, Los Angeles, California 90089--2532;
\email{guralnick@usc.edu} \and
T. Tao: Department of Mathematics, UCLA, 405 Hilgard Ave, Los Angeles CA 90095, USA; \email{tao@math.ucla.edu}
}



\begin{abstract}
We show that random Cayley graphs of finite simple (or semisimple) groups of Lie type of fixed rank are expanders. The proofs are based on the Bourgain-Gamburd method and on the main result of our companion paper \cite{bggt2}.

\end{abstract}

\section{Introduction and statement of results}\label{sec1}

 The aim of this paper is to show that random pairs of elements in finite simple (or semisimple) groups of Lie type are expanding generators if the rank of the group is fixed. The precise definition of a finite simple group of Lie type is deferred to Definition \ref{dfs}, but let us remark that these are the infinite families of simple groups appearing in the classification of finite simple groups (CFSG) other than the alternating groups. They include the so-called classical or Chevalley groups as well as the families of ``twisted'' groups, such as  the Steinberg groups and the Suzuki-Ree groups.
A typical example of a classical group is the projective special linear group $A_r(q) \coloneqq  \PSL_{r+1}(\F_q)$ over a finite field $\F_q$, which is the quotient of $\SL_{r+1}(\F_q)$ by its centre $Z(\SL_{r+1}(\F_q))$, whilst an easily described example of a twisted group is the family\footnote{This family of twisted groups is one of the families of Steinberg groups.  Note that in some literature this group is denoted $A_r(\tilde q)$ instead of $A_r(\tilde q^2)$; see Remark \ref{diff}.} of projective special unitary groups ${}^2A_{r}(\tilde q^2) \coloneqq  \PSU_{r+1}(\F_{\tilde q^2})$, which are the quotients of the special unitary groups
\begin{equation}\label{surq}
 \SU_{r+1}(\F_{\tilde q^2}) \coloneqq  \{g \in \SL_{r+1}(\F_{\tilde q^2}), {}^Tg^{\sigma} g = 1\}
 \end{equation}
by their centres $Z(\SU_{r+1}(\F_{\tilde q^2}))$. Here, $\tilde q$ is a prime or a power of a prime, ${}^Tg$ is the transposed matrix, $1$ is the identity matrix, and $g^{\sigma}$ is the image of $g$ under the Frobenius map $\sigma : \F_{\tilde q^2} \rightarrow \F_{\tilde q^2}$ defined by $x \mapsto x^{\tilde q}$.   As stated earlier, we will be interested in the regime where the rank $r$ of these groups remains fixed, while the field size $q= \tilde q^2$ is allowed to go to infinity.

Let us recall what the notion of \emph{expansion} means in the context of generators for finite groups.  It will be convenient to use a spectral notion of expansion.

\begin{definition}[Spectral expansion]\label{specex}
Suppose that $\eps > 0$, that $G$ is a finite group and that $x_1,\dots, x_k \in G$. Let $\mu$ be the probability measure
$$ \mu \coloneqq  \frac{1}{2k} \sum_{i=1}^k \delta_{x_i} + \delta_{x_i^{-1}}$$
on $G$, where (by abuse of notation) we view $\delta_x$ as a function on $G$ that equals\footnote{As noted in the notation section to follow, we will use $|A|$ to denote the cardinality of a finite set $A$.} $|G|$ at $x$ and zero elsewhere.  Consider the convolution operator $T: f \mapsto f * \mu$ on the Hilbert space $L^2(G)$ of functions $f \colon G \to \C$ with norm
$$ \|f\|_{L^2(G)} \coloneqq  (\E_{x \in G} |f(x)|^2)^{1/2} = \left(\frac{1}{|G|} \sum_{x \in G} |f(x)|^2\right)^{1/2}$$
where convolution $f*\mu$ is defined by the formula
$$ f*\mu(x) \coloneqq  \E_{y \in G} f(y) \mu(y^{-1} x).$$
We say that $\{x_1,\dots,x_k\}$ is \emph{$\eps$-expanding} if one has
$$ \| Tf \|_{L^2(G)} \leq (1-\eps) \|f\|_{L^2(G)}$$
for all functions $f \colon G \to \C$ of mean zero.  Equivalently, all eigenvalues of the self-adjoint operator $T$, other than the trivial eigenvalue of $1$ coming from the constant function, lie in the interval $[-1+\eps,1-\eps]$.
\end{definition}

It is well known (see e.g. \cite[Section 2]{hoory-linial-wigderson}, or \cite[Prop. 4.2.5]{lubotzky-book}) that an $\eps$-expanding set $\{x_1,\ldots,x_k\}$ is also \emph{combinatorially expanding} in the sense that
\[ |(A x_1 \cup A x_1^{-1} \cup A x_2 \cup A x_2^{-1} \cup \dots \cup A x_k \cup Ax^{-1}_k) \setminus A| \geq \eps' |A|\]
for every set $A \subseteq G$ with $|A| \leq |G|/2$, and some $\eps'>0$ depending only on $\eps,k>0$; in particular this implies that the $x_1,\ldots,x_k$ generate $G$ (otherwise one could take $A$ to be the group generated by $x_1,\ldots,x_k$ to obtain a counterexample to combinatorial expansion).  Actually, this implication can be reversed as long as the Cayley graph generated by $x_1,\ldots,x_k$ is not bipartite, or equivalently that there does not exist an index $2$ subgroup $H$ of $G$ which is disjoint from the $x_1,\ldots,x_k$; we record this argument in Appendix \ref{equiv-sec}.  An $\eps$-expanding set is also \emph{rapidly mixing} in the sense that
\begin{equation}\label{mung}
 \| \mu^{(n)} - 1 \|_{L^\infty(G)} \leq |G|^{-10}
\end{equation}
(say) for all $n \geq C_1 \log |G|$, where $C_1$ depends only on $\eps,k$.  Conversely, if one has rapid mixing \eqref{mung} for some $n \leq C \log |G|$, then one has $\eps$-expansion for some $\eps>0$ depending only on $C,k$, as can be easily deduced from a computation of the trace of $T^{2n}$ (or equivalently, the Frobenius norm of $T^n$) and the spectral theorem.  If one has a family of finite groups $G$ and $\eps$-expanding sets $\{x_1,\ldots,x_k\} \subset G$, with $\eps, k$ uniform in the family, then the associated Cayley graphs $\operatorname{Cay}(G, \{x_1,\ldots,x_k\})$ form an expander family.  We will however not focus on the applications to expander graph constructions here, again referring the reader to \cite{hoory-linial-wigderson} and \cite{lubotzky-book} for more discussion.

Our main theorem is as follows.

\begin{theorem}[Random pairs of elements are expanding]\label{mainthm}
Suppose that $G$ is a finite simple group of Lie type and that $a,b \in G$ are selected uniformly at random. Then with probability at least $1 - C |G|^{-\delta}$, $\{a,b\}$ is $\eps$-expanding for some $C, \eps, \delta > 0$ depending only on the rank of $G$.
\end{theorem}

In section \ref{semisimple}, we also extend the above result to almost direct products of quasisimple groups of Lie type; see Theorem \ref{mainthm2} there.

There has been a considerable amount of prior work on expansion in finite simple groups. We offer a brief and incomplete summary now:

\begin{enumerate}
\item Using Kazhdan's property $(T)$, Margulis \cite{margulis} gave the first explicit construction of expander graphs.  In particular, he constructed explicit expanding sets of generators for $\PSL_d(\F_p)$ for any fixed $d \geq 3$, by projection from a fixed set of generators of $\SL_d(\Z)$.  Sharp analogous results for $d=2$ were later obtained by Margulis \cite{margulis-explicit} and by Lubotzky, Phillips, and Sarnak \cite{lps} using known cases of the Ramanujan-Petersson conjectures.
\item In a breakthrough paper, Bourgain and Gamburd  \cite{bourgain-gamburd} proved Theorem \ref{mainthm} in the case $G = \PSL_2(\F_p)$. A key ingredient of their proof was Helfgott's product theorem in this group \cite{helfgott-sl2}. By combining subsequent work of theirs with generalisations of Helfgott's work by Pyber-Szab\'o \cite{pyber-szabo} and Breuillard-Green-Tao \cite{bgt} one may show the existence of some expanding pairs of generators in $\SL_r(\F_p)$ and indeed in $\G(\F_p)$ for any almost simple algebraic group $\G$. Here $p$ is prime. See the paper by Varj\'u and Salehi-Golsefidy \cite{varju-salehi} for more on this aspect.
\item Kassabov, Lubotzky and Nikolov  \cite{kassabov-nikolov-lubotzky} showed that every finite simple group, with the possible exception of the Suzuki groups $\Suz(q)$, admits an $\eps$-expanding set of generators $x_1,\dots, x_k$, with $k$ and $\eps$ independent of the group (and in particular, uniform even in the rank of the group $G$).
\item The first, second and fourth authors showed in \cite{suzuki} that the previous claim also holds for the Suzuki groups with $k = 2$, and in fact that Theorem \ref{mainthm} holds in this case and for $G = \PSL_2(\F_q)$.
\item Gamburd and the first named author showed in \cite{breuillard-gamburd} that there is $\eps>0$ such that every generating pair of $\PSL_2(\F_p)$ is an $\eps$-expanding pair, whenever $p$ stays outside a set of primes $\mathcal{P}_0$ of density $0$.
\item In the converse direction, it was proved by Lubotzky and Weiss \cite[Cor. 4.4]{lubotzky-weiss} that for any fixed prime $p$ there is a $3$-element generating set of $\SL_n(\F_p)$ such that the family of resulting Cayley graphs is not an expanding family when $n$ tends to $+\infty$. Another proof of this fact, due to Y. Luz, was also given in \cite{lubotzky-weiss} and recently Somlai \cite{somlai} generalized that argument to show that every sequence of finite simple groups of Lie type with rank going to infinity admits a sequence of Cayley graphs arising from at most $10$ generators which is \emph{not} $\eps$-expanding for any uniform $\eps>0$.
\end{enumerate}

The proofs of the works of Kassabov, Lubotzky and Nikolov \cite{kassabov-nikolov-lubotzky,lubotzky} used very different arguments coming from a rather impressive range of mathematical areas. An important aspect of their proof was to make use of the existence of various copies of $\SL_2(\F_q)$ in higher rank finite simple groups. This feature required the use of more than $2$ generators to produce an expanding set and does not seem to be suited to the treatment of random sets of generators. Note however that they were able to treat all finite simple groups (except for $\Suz(q)$) uniformly, while our result falls short of saying anything when the rank goes to infinity.

The method used in the present paper follows the Bourgain-Gamburd strategy first introduced in \cite{bourgain-gamburd}, as did the paper \cite{suzuki} on Suzuki groups by three of the authors, and will be outlined in the next section.  To verify the various steps of the Bourgain-Gamburd argument, we will need a number of existing results in the literature, such as the quasirandomness properties and product theorems for finite simple groups of Lie type, as well as the existence of strongly dense free subgroups that was (mostly) established in a previous paper \cite{bggt2} of the authors.

In most of our argument, the finite simple groups of Lie type can be treated in a unified manner, albeit with some additional technical complications when dealing with twisted groups rather than classical groups.  However, there are two exceptional cases which need special attention at various stages of the argument, which we briefly mention here.

The first exceptional case is when $G = \Sp_4(\F_q) = C_4(q)$ is the symplectic group of order $4$ over a field $\F_q$ of characteristic $3$.  This case was omitted from the results in \cite{bggt2} for a technical reason having to do with an absence of a suitable algebraic subgroup of $G$ to which a certain induction hypothesis from \cite{bggt2} could be applied.  In Appendix \ref{sp4-app} we present an alternate argument that can substitute for the arguments in \cite{bggt2} in that case.

The other exceptional case occurs when $G = {}^3D_4(q)$ is a triality group.  This is a twisted group that contains subgroups associated to fields of index $2$, and for technical reasons it turns out that such fields are too ``large'' (and the Schwartz-Zippel type bounds for twisted groups too weak) for our main argument to work in this case.  In Section \ref{d4-sec} we give the modifications to the main argument necessary to address this case.\vspace{11pt}

\emph{Remark.} We note that none of our work has anything to say at the present time about the alternating groups $\operatorname{Alt}_n$ (or the closely related symmetric groups $\operatorname{Sym}_n$). Although it was shown by Kassabov \cite{kassabov} that there are uniform $\eps>0, k>2$ such that every $\operatorname{Alt}_n$ has an $\eps$-expanding $k$-tuple, the existence of a pair of $\eps$-expanding generators for $\operatorname{Alt}_n$ (with $\eps$ not depending on $n$) remains an open problem, as does the question of whether a random pair or even a random $k$-tuple of elements in $\operatorname{Alt}_n$ has this property. Note that it has been known for some time (see \cite{dixon}) that a random pair of elements generates $\operatorname{Alt}_n$ with probability going to $1$ as $n\to \infty$.  Recent progress on understanding the diameter of Cayley graphs on such groups may be found in \cite{bghhss,helfgott-seress}.\vspace{11pt}

\emph{Remark.} Our arguments actually show the following generalisation of Theorems \ref{mainthm} and \ref{mainthm2}: if $G$ is an almost direct product of finite simple groups of Lie type, $a,b \in G$ are selected uniformly at random, and $w_1,w_2 \in F_2$ are non-commuting words of length at most $|S|^\delta$, then with probability at least $1 - C |S|^{-\delta}$, $\{w_1(a,b),w_2(a,b)\}$ is $\eps$-expanding, for some $C, \eps, \delta > 0$ depending only on the rank of $G$, where $S$ is the smallest simple factor of $G$.  See Remark \ref{word} for details.  Note that this is a non-trivial extension of the original theorem because the map $(a,b) \mapsto (w_1(a,b), w_2(a,b))$ does not need to resemble a bijection; for instance, if $G$ is a matrix group and $w_1,w_2$ are conjugate non-commuting words, then $w_1(a,b)$ and $w_2(a,b)$ necessarily have the same trace.  It is in fact conjectured that \emph{all} pairs of generators of $G$ should expand at a uniform rate, but this is not known in general, although in \cite{breuillard-gamburd}, Gamburd and the first author established this conjecture for $\SL_2(\F_p)$ for all $p$ in a density one set of primes.  From the above result and the union bound, we can at least show (after adjusting $\delta$ slightly) that with probability at least $1-C |G|^{-\delta}$, it is the case that for \emph{all} pairs of non-commuting words $w_1, w_2 \in F_2$ of length at most $\delta \log |G|$, the pair $\{w_1(a,b),w_2(a,b)\}$ is $\eps$-expanding.\vspace{11pt}

\textsc{Notation.} We use the asymptotic notation $O(X)$ to denote any quantity whose magnitude is bounded by $CX$ for some absolute constant $C$.  If we need $C$ to depend on parameters, we indicate this by subscripts, e.g. $O_k(X)$ is a quantity bounded in magnitude by $C_k(X)$ for some constant $C$ depending only on $k$.  We write $Y \ll X$ for $Y = O(X)$, $Y \ll_k X$ for $Y = O_k(X)$, etc.


We use $|E|$ to denote the cardinality of a finite set $E$.  If $E$ is finite and non-empty and $f \colon E \to \R$ is a function, we write
$$ \E_{x \in E} f(x) \coloneqq  \frac{1}{|E|} \sum_{x \in E} f(x)$$
and if $P(x)$ is a property of elements $x$ of $E$, we write
$$ \P_{x \in E} P(x) \coloneqq  \frac{1}{|E|} |\{ x \in E: P(x) \; \mbox{holds} \}|.$$

Suppose that $V$ is an affine variety defined over an algebraically closed field $\k$, thus $V$ is a subset of $\k^n$ for some $n$ that is cut out by some polynomials defined over $\k$.  If $\F$ is a subfield of $\k$, we use $V(\F)$ to denote the $\F$-points of $V$, thus $V(\F) = V \cap \F^n$.  In particular, $V = V(\k)$.  We will generally use boldface symbols such as $\G$ to denote algebraic groups defined over $\k$, while using plain symbols such as $G$ to denote finite groups.

\vspace{11pt}

\section{An outline of the argument}\label{outlinesection}

In this section we give an overview of the proof of Theorem \ref{mainthm}.

We first perform a convenient reduction.  As we will recall in Definition \ref{dfs}, all the finite simple groups $G$ of Lie type can be expressed in the form
$$ G = \tilde G / Z(\tilde G)$$
for some slightly larger group $\tilde G$ which lies inside a linear algebraic ``mother'' group $\G$, with the centre $Z(\tilde G)$ being of bounded cardinality: $|Z(\tilde G)| = O(1)$.  For instance, if $G = \PSL_{r+1}(\F_q)$, one can take $\tilde G = \SL_{r+1}(\F_q)$, in which case $\G$ is the algebraic group $\SL_{r+1}$ and the centre has order $\gcd(q,r+1)$.  For each finite simple group $G$, there are a finite number of possibilities (up to isomorphism) for the bounded cover $\tilde G$ and the mother group $\G$; in most cases, the exact choice of $\tilde G$ and $\G$ will not be too important, but in the warmup case of the projective special linear group $\PSL_{r+1}(\F_q)$ in Section \ref{class} and in the special case of the triality group ${}^3 D_4(q)$ in Section \ref{d4-sec} it will be convenient for computational purposes to work with a particular such choice.  It is easy to see that to prove Theorem \ref{mainthm} for the group $G$, it suffices to do so for the bounded cover $\tilde G$. Indeed if $\{a,b\}$ is an $\eps$-expanding pair in $\tilde G$, then its projection $\{\overline{a},\overline{b}\}$ to $G$ is also $\eps$-expanding, because every eigenvalue of the averaging operator $T_{\{\overline{a},\overline{b}\}}$ of Definition \ref{specex} is also an eigenvalue of $T_{\{a,b\}}$. Henceforth we will work with $\tilde G$, as this allows us to easily use the algebraic geometry structure of the mother group $\G$.

As mentioned in the introduction, we will establish expansion for $\tilde G$ via the ``Bourgain-Gamburd machine'', which we formalise in Section \ref{sec2}.
Roughly speaking, this machine gives sufficient conditions for rapid mixing of the iterated convolutions $\mu^{(n)}$ in a finite group $\tilde G$ associated to a bounded set of generators (which, in our case, are $\{a^{\pm 1}, b^{\pm 1}\}$ for some randomly chosen $a, b \in \tilde G$).  By standard arguments, this rapid mixing then implies expansion of the set of generators.

To obtain this mixing, one needs to establish three ingredients, which we state informally as follows:
\begin{itemize}
\item[(i)] (Non-concentration)  Most words of moderate length generated by a random pair of generators will not be concentrated in a proper subgroup.
\item[(ii)] (Product theorem) If a medium-sized set $A$ is not contained in a proper subgroup, then the product set $A \cdot A \cdot A$ is significantly larger than $A$.
\item[(iii)] (Quasirandomness) $G$ has no non-trivial low-dimensional representations, or equivalently that convolution of broadly supported probability measures on $G$ are rapidly mixing.
\end{itemize}
For a more precise version of these three hypotheses, see Proposition \ref{machine}.  Roughly speaking, the non-concentration hypothesis (i) is needed to ensure that $\mu^{(n)}$ expands for small $n$ (less than $C_0 \log |G|$ for some constant $C_0$), the product theorem (ii) is needed to show that $\mu^{(n)}$ continues to expand for medium $n$ (between $C_0 \log |G|$ and $C_1 \log |G|$ for some larger constant $C_1$), and the quasirandomness hypothesis (iii) is needed to show that $\mu^{(n)}$ rapidly approaches the uniform distribution for large $n$ (between $C_1 \log |G|$ and $C_2 \log |G|$ for some even larger constant $C_2$).  See Section \ref{sec2} for further discussion.

The quasirandomness property (iii) is an immediate consequence of the existing literature \cite{landazuri-seitz, seitz-zal} on representations of finite simple groups of Lie type; see Section \ref{sec1a}.  The product theorem for general finite simple groups (ii) of Lie type was recently established by Pyber and Szab\'o \cite{pyber-szabo} (building upon the earlier work of Helfgott \cite{helfgott-sl2,helfgott-sl3} that treated the cases $\SL_2(\F_p), \SL_3(\F_p)$); in Section \ref{sec1a} we give an alternate derivation of this theorem using the closely related result obtained by three of the authors in \cite{bgt}.  The main remaining difficulty is then to establish the non-concentration estimate, which prevents too many of the words generated by a random pair of elements from being trapped inside a proper subgroup of $G$.  More precisely, we will need to establish the following key proposition:

\begin{proposition}[Non-concentration]\label{nonc} Suppose that $G$ is a finite simple group of Lie type.  We allow implied constants to depend on the rank $r = \rk(G)$ of $G$.  Let $\tilde G$ be the bounded cover of $G$ from Definition \ref{dfs}.  Then there exists a positive even integer $n = O(\log |\tilde G|)$ and an exponent $\gamma>0$ depending only on the rank $r$ such that
 \begin{equation}\label{probaupperbound}\P_{a,b \in \tilde G}( \P_{w \in W_{n,2}}(w(a,b) \in H) \leq |\tilde G|^{-\gamma} \hbox{ for all } H < \tilde G ) = 1 - O(|\tilde G|^{-\gamma}),
 \end{equation}
where $a,b$ are drawn uniformly at random from $\tilde G$, $w$ is drawn uniformly at random from the space $W_{n,2}$ of all formal words \textup{(}not necessarily reduced\textup{)} on two generators of length exactly $n$, and $H$ ranges over all proper subgroups of $\tilde G$.
\end{proposition}

\begin{remark}  If $n \leq n'$ and $H$ is a subgroup of $G$, we have the inequality
$$ \P_{w \in W_{n',2}}(w(a,b) \in H) \leq \sup_{g \in G} \P_{w \in W_{n,2}}(w(a,b) \in gH) $$
(as can be seen by factorising a word in $W_{n',2}$ as a word in $W_{n,2}$ and another word, whose value one then conditions over), and similarly one has
$$ \sup_{g \in G} \P_{w \in W_{n,2}}(w(a,b) \in gH)^2 \leq \P_{w \in W_{2n,2}}(w(a,b) \in H)$$
since if $w(a,b), w'(a,b) \in gH$ then $w' w^{-1}(a,b) \in H$.  From this we see that the existence of some $n=O(\log |\tilde G|)$ satisfying the conclusion of the proposition implies that essentially the same bound holds for all larger $n$.
\end{remark}

In Section \ref{sec1a} we will show why Proposition \ref{nonc} implies Theorem \ref{mainthm}.

It remains to establish the proposition.  Informally, we need to show that given a random pair of generators $a$ and $b$, then the words $w(a,b)$ arising from those generators usually do not concentrate in a proper subgroup $H$ of $\tilde G$.  Of course, we may restrict attention to the maximal proper subgroups of $\tilde G$.

The first step in doing so is a classification \cite{asch, landazuri-seitz, larsen-pink, liebeck-seitz, steinberg-yale} of the maximal proper subgroups $H$ of the bounded cover $\tilde G$ of a finite simple group $G$ of Lie type, which among other things asserts that such subgroups either live in a proper Zariski-closed subgroup of the mother group $\G$ (with bounds on the algebraic complexity of this closed subgroup), or else live in a conjugate of a subgroup of the form $\G(\F')$, where $\F'$ is a proper subfield of $\F$; see Lemma \ref{max} for a precise statement.  We refer to these two cases as the \emph{structural case} and the \emph{subfield case} respectively.

We first discuss the structural case.  In some cases, such as the Suzuki groups $\Suz(q) = {}^2 B_2(q)$ or the rank one special linear groups $\PSL_2(\F_q) = A_1(q)$, this case is relatively easy to handle, because  all proper algebraic subgroups are solvable in those cases; this fact was exploited in the papers \cite{bourgain-gamburd}, \cite{suzuki} establishing Theorem \ref{mainthm} for such groups.  In the setting of the present paper, where we consider the general higher rank case, we have to address a new difficulty due to the presence of large \emph{semisimple} proper algebraic subgroups of $\G$, which could in principle trap most words of moderate length generated by a generic pair of generators.

To eliminate this possibility, we will use the following result, which (with one exception) we established in a separate paper \cite{bggt2}.  Call a subgroup $\Gamma$ of an algebraic group $\G$ \emph{strongly dense} if every pair $x, y$ of non-commuting elements of $\Gamma$ generate a Zariski-dense subgroup of $\G$.  Informally, this means that very few pairs of elements in $\Gamma$ can be simultaneously trapped inside the same proper algebraic subgroup of $G$.   We then have:

\begin{theorem}[Existence of strongly dense subgroups]\label{sec3-key-prop}
Suppose that $\G(k)$ is a semisimple algebraic group\footnote{In this paper, semisimple algebraic groups are always understood to be connected.} over an uncountable algebraically closed field $k$. Then there exists a free non-abelian subgroup $\Gamma$ of $\G(k)$ on two generators which is strongly dense.
\end{theorem}

\begin{proof}  If $\G$ is the algebraic group $\Sp_4$ and $k$ has characteristic $3$, we establish this result in Appendix \ref{sp4-app}.  All other cases of this theorem were established as the main result of \cite{bggt2}.
\end{proof}

\emph{Remark.} Note that while Theorem \ref{sec3-key-prop} is stated for free groups on two generators, it implies the same for free groups with $m$ generators for any $m \geq 2$, since a free group on two generators contains a free group on $m$ generators (and in fact contains a countably generated free subgroup). The same applies to Theorem \ref{mainthm}, that is for every fixed $k \geq 2$ a random $k$-tuple will be $\eps$-expanding with high probability.\vspace{11pt}

By combining the above proposition with a variant of the Schwartz-Zippel lemma (see Proposition \ref{sufficiently-zdense-prop}) and an algebraic quantification of the property of generating a sufficiently Zariski-dense subgroup (see Proposition \ref{crit}), we will be able to show that for any pair of words $w, w'$ of length $n$, and for most $a, b$, $w(a,b)$ and $w'(a,b)$ will not be trapped inside the same proper algebraic subgroup of $G$, which will be sufficient to establish non-concentration in the structural case; see Section \ref{sec1a}.

Now we discuss the subfield case.  The simplest case to consider is the Chevalley group case when $\tilde G = \G(\F_q)$ is a matrix group over some finite field $\F_q$.  The starting point is then the observation that if a matrix $g$ lies in a conjugate of a subfield group $\G(\F_{q'})$, where $\F_{q'}$ is a subfield of $\F_q$, then the coefficients $\gamma_1(g),\ldots,\gamma_d(g)$ of the characteristic polynomial of $g$ all lie in $\F_{q'}$.  The idea is then to show that $\gamma_i(w(a,b))$ does not concentrate in $\F_{q'}$ for each $i$.  For most values $x_i$ of $\F_{q'}$, one can use the Schwartz-Zippel lemma for $G$ to show (roughly speaking) that $\gamma_i(w(a,b))$ only takes the value $x_i$ with probability $O(q^{-1})$; summing over all values of $x_i$, one ends up with a total concentration probability of $O(q'/q)$.  But as $\F_{q'}$ is a proper subfield of $\F_q$, one has $q' = O(q^{1/2})$, and as such the contribution of this case is acceptable.  (There is a ``degenerate'' case when $x_i = \gamma_i(1)$ which has to be treated separately, by a variant of the above argument; see Section \ref{sec1a} for details.)

In the non-Chevalley cases, $\tilde G$ is (the derived group of) the fixed points $\G(\F_q)^\sigma$ of $\G(\F_q)$ under some automorphism $\sigma$ of order $d \in \{2,3\}$.  It turns out that $\tilde G$ is ``sufficiently Zariski-dense'' in $\G$ in the sense of obeying a variant of the Schwartz-Zippel lemma; see Proposition \ref{sufficiently-zdense-prop} for a precise statement.  In principle, one can then run the same argument that was used in the classical case.  However, the presence of the automorphism $\sigma$ defining $G$ turns out to cut down the probability bound in the Schwartz-Zippel lemma from $O(q^{-1})$ to $O(q^{-1/d})$, leading to a final bound of $O(q' / q^{-1/d} )$ rather than $O(q'/q)$.  This is not a difficulty when $d=2$, as it turns out that the only relevant subfields $\F_{q'}$ in those cases have size at most $q^{1/3}$ (after excluding those cases that can also be viewed as part of the structural case); but it causes a significant problem when $G$ is a triality group ${}^3D_4(q)$, which is the unique case for which $d=3$.  To deal with this case we will apply an \emph{ad hoc} argument in Section \ref{d4-sec}, based on passing from ${}^3D_4(q)$ to a more tractable subgroup $\SL_2(\F_{q}) \circ \SL_2(\F_{\tilde q})$ in which non-concentration is easier to establish.

The rest of the paper is organised as follows.   In Section \ref{sec2} we present the abstract ``Bourgain-Gamburd machine'', which reduces the task of verifying expansion to that of verifying quasirandomness, a product theorem, and non-concentration.  Then, in Section \ref{class}, we prove these facts in the model case of the projective special linear group $G = A_r(q) = \PSL_{r+1}(\F_q)$, which is technically simpler than the general case and allows for some more explicit computations.  In Section \ref{sec1a}, we define formally the concept of a finite simple group of Lie type, and extend the arguments in Section \ref{class} to this class of groups to give a full proof of Theorem \ref{mainthm}, contingent on a certain Schwartz-Zippel type lemma which we then prove in Section \ref{non-conc-sec-untwist}, together with a separate treatment of the triality group case $G = {}^3 D_4(q)$ which requires a modification to one part of the argument. In the last section, Section \ref{semisimple}, we extend our arguments to cover the case when the group $G$ is no longer simple, but an almost direct product of simple groups. This requires adapting the product theorem to this setting and proving an analogous non-concentration estimate. 

\section{The Bourgain-Gamburd expansion machine}\label{sec2}

Bourgain and Gamburd, in their groundbreaking paper \cite{bourgain-gamburd}, supplied a new paradigm for proving that sets of generators expand, applying it to show that any set of matrices in $\SL_2(\Z)$ generating a Zariski-dense subgroup descends to give an expanding set of generators in $\SL_2(\F_p)$, and also the special case $G = \SL_2(\F_p)$ of Theorem \ref{mainthm}.

This ``Bourgain-Gamburd machine'' was also critical in \cite{suzuki}, the paper on expansion in Suzuki groups by the first, second and fourth authors.

In this section we give a version of this machine, suitable for use for finite simple groups, which will be adequate for our purposes, with the proofs deferred to Appendix \ref{bg-app}.  In that appendix we will also remark on slightly more general contexts in which one might hope to apply the machine at the end of the section, but readers looking for the most general setting in which the method is valid should consult the work of Varj\'u \cite{varju}.

Suppose that $G = (G,\cdot)$ is a finite group, and let $S = \{x_1,\dots,x_k\}$ be a symmetric set of generators for $G$. In this paper we will usually be taking $S = \{a, a^{-1}, b , b^{-1}\}$ for a random pair $a,b \in G$, and $G$ will be a finite simple group of Lie type but the discussion in this section will apply to more general types of generators $S$ and more general finite groups $G$.

Write
\[ \mu = \mu_S \coloneqq  \frac{1}{k}(\delta_{x_1} + \dots + \delta_{x_k})\]
for the uniform probability measure on the set $S$, where $\delta_x$ is the Dirac mass at $x$.  We abuse notation very slightly and identify the space of probability measures on the discrete space $G$ with the space of  functions $\mu \colon G \to \R^+$ with mean $\E_{x \in G} \mu(x) = 1$.  In particular, we identify the uniform measure with the constant function $1$, and the Dirac mass $\delta_x$ with the function that equals $|G|$ at $x$ and vanishes elsewhere.

We write
\[ \mu^{(n)} \coloneqq  \mu \ast \dots \ast \mu\]
for the $n$-fold convolution power of $\mu$ with itself, where the convolution $\mu_1 \ast \mu_2$ of two functions $\mu_1, \mu_2 \colon G \to \R^+$ is given by the formula
\begin{equation}\label{convdef}
\mu_1 \ast \mu_2(g) \coloneqq  \E_{x \in G} \mu_1(g x^{-1}) \mu_2(x).
\end{equation}
One may think of $\mu^{(n)}(x)$ as describing the normalised probability that a random walk of length $n$ starting at the identity in $G$ and with generators from $S$ hits the point $x$.  In particular, if $S  = \{a,b,a^{-1},b^{-1}\}$, and $H \subset G$, then
\begin{equation}\label{munh}
\mu^{(n)}(H) = \P_{w \in W_{n,2}}( w(a,b) \in H ),
\end{equation}
where $W_{n,2}$ is the space of all formal words (not necessarily reduced) on two generators of length exactly $n$.

Here is an instance of the Bourgain-Gamburd machine that will suffice for our paper (and for \cite{suzuki}).  Define a \emph{$K$-approximate subgroup} of a group $G$ to be a finite symmetric subset $A$ of $G$ containing the identity such that the product set $A \cdot A \coloneqq  \{a \cdot b: a,b \in A\}$ can be covered by at most $K$ left-translates of $A$.

\begin{proposition}[Bourgain-Gamburd machine]\label{machine}
Suppose that $G$ is a finite group, that $S \subseteq G$ is a symmetric set of $k$ generators, and that there are constants $0 < \kappa < 1 <  \Lambda$ with the following properties.
\begin{enumerate}
\item \textup{(Quasirandomness)}. The smallest dimension of a nontrivial representation $\rho \colon G \to \GL_d(\C)$ of $G$ is at least $|G|^{\kappa}$;
\item \textup{(Product theorem)}. For all $\delta > 0$ there is some $\delta' = \delta'(\delta)>0$ such that the following is true. If $A \subseteq G$ is a $|G|^{\delta'}$-approximate subgroup  with $|G|^{\delta} \leq |A| \leq |G|^{1 - \delta}$ then $A$ generates a proper subgroup of $G$;
\item \textup{(Non-concentration estimate)}. There is some even number $n \leq  \Lambda\log |G|$ such that
\[ \sup_{H < G}\mu^{(n)}(H) < |G|^{-\kappa},\] where the supremum is over all proper subgroups $H < G$.
\end{enumerate}
Then $S$ is $\eps$-expanding for some $\eps > 0$ depending only on $k,\kappa, \Lambda$, and the function $\delta'(\cdot )$ \textup{(}and this constant $\eps$ is in principle computable in terms of these constants\textup{)}.
\end{proposition}

We prove this proposition in Appendix \ref{bg-app}, by a variant of the techniques in \cite{bourgain-gamburd}, using a version of the Balog-Szemer\'edi-Gowers lemma which we give in Appendix \ref{bsg-app-sec}.  As mentioned in the previous section, the hypothesis (iii), the non-concentration estimate, represents the bulk of the new work in this paper, in the context when $S$ is generated by two random elements. This condition was also difficult to verify in earlier works such as \cite{bourgain-gamburd2,bourgain-gamburd3}, where deep results from algebraic geometry and random matrix products were required. The interesting feature of (iii) is that it is actually \emph{necessary} in order to verify expansion, as it is a consequence of the rapid mixing property \eqref{mung}. This is in contrast to (i) and (ii) which, although they certainly ``pull in the direction of'' expansion, are by no means strictly speaking necessary in order to establish it.   We also remark that (iii) is the only condition of the three that actually involves the set $S$.

In view of Proposition \ref{machine} (and \eqref{munh}), as well as the observation that Theorem \ref{mainthm} for a finite simple group $G$ will follow from the same theorem for the bounded cover $\tilde G$ of $G$ from Definition \ref{dfs}, we see that Theorem \ref{mainthm} will follow from Proposition \ref{nonc} and the following additional propositions.

\begin{proposition}[Quasirandomness]\label{quasi}  Let $G$ be a finite simple group of Lie type, and let $\tilde G$ be the bounded cover of $G$ coming from Definition \ref{dfs}.  Then every non-trivial irreducible representation $\rho \colon \tilde G \to \GL_d(\C)$ of $\tilde G$ has dimension $d$ at least $|G|^{\beta}$, where $\beta > 0$ depends only on the rank of $G$.
\end{proposition}

\begin{proposition}[Product theorem]\label{prod}  Let  $G$ be a finite simple group of Lie type, and let $\tilde G$ be the bounded cover of $G$ coming from Definition \ref{dfs}.  For all $\delta > 0$ there is some $\delta' = \delta'(\delta) > 0$ depending only on $\delta$ and the rank of $G$ such that the following is true. If $A \subseteq \tilde G$ is a $|\tilde G|^{\delta'}$-approximate subgroup  with $|\tilde G|^{\delta} \leq |A| \leq |\tilde G|^{1 - \delta}$ then $A$ generates a proper subgroup of $\tilde G$.
\end{proposition}

These two propositions will follow easily from known results in the literature on finite simple groups of Lie type, as we will discuss shortly.

\section{A model case}\label{class}

In this section, we establish Theorem \ref{mainthm} in the model case of the projective special linear group
$$ G = A_r(q) = \PSL_{r+1}(\F_q)$$
over some finite field $\F_q$ and some rank $r \geq 1$.  This case is significantly simpler than the general case, but will serve to illustrate the main ideas of the argument.  In particular, many of the arguments here will eventually be superceded by more general variants in later sections.

Henceforth we allow all implied constants to depend on the rank $r$ of $G$, thus $r=O(1)$.  We may assume that the finite field $\F_q$ is sufficiently large depending on $r$, as the claim is trivial otherwise.

It will be convenient to lift up from the finite simple group $G$ to the linear algebraic group
$$\tilde G \coloneqq  \G(\F_q) \subset \G(\k) \subset \GL_{m}(\k) \subset \Mat_m(\k)$$
over $\F_q$, where $m \coloneqq  r+1$, $\G\coloneqq  \SL_m$, $\k$ is an uncountable algebraically closed field containing $\F_q$, and $\Mat_m$ is the ring of $m \times m$ matrices.  Note that $G$ is the quotient $G = \tilde G/Z(\tilde G)$ of $\tilde G$ by its centre $Z(\tilde G)$, which has order $O(1)$ and as such will play a negligible role in the analysis that follows.  The group $\G = \SL_{m}$ is an example of an \emph{absolutely almost simple} algebraic group, in the sense that $\G$ has no non-trivial proper connected normal subgroups.

As remarked earlier, to prove Theorem \ref{mainthm} for $G = \PSL_{r+1}(\F_q)$ it will suffice to do so for $\tilde G = \SL_m(\F_q)$, so we shall henceforth work with the special linear group $\tilde G$ instead of $G$.

Now we review the structure of the special linear group $\G = \SL_{r+1}$, with an eye towards future generalisation to other finite simple groups of Lie type.  We recall the \emph{Bruhat decomposition}
$$ \G = \B W \B$$
where the \emph{Borel subgroup} $\B$ of $\G$ is the space of upper-triangular $m \times m$ matrices of determinant one, and the \emph{Weyl group} $W$ is the group of permutation matrices.  We factorise
$$ \B = \U \T = \T \U$$
where the \emph{unipotent group} $\U$ is the subgroup of $\B$ consisting of upper-triangular matrices with ones on the diagonal, and the \emph{maximal torus} is the group of diagonal matrices of determinant one.  Since $W$ normalises $\T$, we thus have
$$ \G = \U \T W \U.$$
In fact we have the more precise decomposition
$$ \G = \bigsqcup_{\w \in W} \U \T \w \U_\w^-$$
that decomposes $\G$ as the disjoint union of $\U \T \w \U_\w^-$, where if $\w$ is the permutation matrix associated to a permutation $\pi: \{1,\ldots,m\} \to \{1,\ldots,m\}$ (thus $\w$ has an entry $1$ at $(i, \pi(i))$ for $i=1,\ldots,m$, and zero elsewhere), then $U_\w^-$ is the subgroup of $U$ consisting of matrices in $U$ whose $(i,j)$ entries vanish whenever $\pi(i) < \pi(j)$.  Furthermore, the products in the above decomposition are all distinct, thus each $g \in \G$ has a unique representation of the form $g = u_1 h \w u$ with $u_1 \in \U$, $h \in \T$, $\w \in W$, and $u \in \U_w^-$; see \cite[Corollary 8.4.4]{carter} for a proof of this result (which is in fact valid in any Chevalley group).  This decomposition (which is essentially a form of Gaussian elimination) can be specialised to the field $\F_q$, thus
$$ \tilde G = \bigsqcup_{\w \in W} \U(\F_q) \T(\F_q) \w \U_\w^-(\F_q)$$
and thus
$$ |\tilde G| = \sum_{\w \in W} |\U(\F_q)| |\T(\F_q)| |\U_\w^-(\F_q)|.$$
The group $\U(\F_q)$ is the group of upper-triangular matrices with ones on the diagonal and coefficients in $\F_q$, and thus has cardinality $q^{m(m-1)/2}$.  When $\w$ is the \emph{long word} $\w_0$ that equals one on the anti-diagonal (and thus corresponds to the permutation $i \mapsto m+i-1$), $\U_\w^-(\F_q)$ is equal to $\U(\F_q)$ and thus also has cardinality $q^{m(m-1)/2}$; in all other cases it has cardinality $q^{d_\w}$ for some $d_\w < m(m-1)/2$.  From this we see that
$$ |\tilde G| = (1+O(1/q)) |\U(\F_q)| |\T(\F_q)| |\U(\F_q)|,$$
so that the ``large Bruhat cell'' $\B(\F_q) \w_0 \B(\F_q) = \U(\F_q) \T(\F_q) \w_0 \U(\F_q)$ occupies almost all of $\tilde G$:
\begin{equation}\label{gbwb}
 |\tilde G| = (1+O(1/q)) |\B(\F_q) \w_0 \B(\F_q)|.
 \end{equation}
Note that a similar argument shows that the large Bruhat cell $\B \w_0 \B$ has larger dimension than all other Bruhat cells $\B \w \B$, and so $\B \w_0 \B$ is Zariski-dense in $\G$.

Among other things, this gives the very crude bound
\begin{equation}\label{ftig}
 |\tilde G| \ll q^{O(1)}
 \end{equation}
so that any gain of the form $O(q^{-c})$ for some $c>0$ is also of the form $O( |\tilde G|^{-c'} )$ for some $c'>0$ depending on $c$ and $r$.
Indeed, one has the more precise exact formula
$$ |\tilde G| = \frac{1}{q-1} \prod_{j=0}^{m-1} (q^m - q^j),$$
although we will not need such precision in our arguments.

As discussed in the last section, by Proposition \ref{machine}, it suffices to verify the quasirandomness property (Proposition \ref{quasi}), the product theorem (Proposition \ref{prod}), and the non-concentration property (Proposition \ref{nonc}) for the projective special linear group $G$.

We begin with the quasirandomness.  It is a result of Landazuri and Seitz \cite{landazuri-seitz} that all non-trivial irreducible projective representations of $\tilde G/Z(\tilde G) = \PSL_{m}(\F_q)$ have dimension\footnote{Actually, as the precise value of $c$ is not important for applications, it suffices to establish the $m=2$ case (as $\SL_m(\F_q)$ clearly contains a copy of $\SL_2(\F_q)$ for any $m \geq 2$, and is almost simple), and this follows already from \cite[Lemma 4.1]{landazuri-seitz}.} $\gg q^c$ for some absolute constant $c>0$ (in fact the more precise lower bound of $q^{m-1}-1$ is obtained for $m > 2$, and $\frac{1}{(2,q-1)} (q-1)$ for $m=2$).  This implies that all non-trivial irreducible \emph{linear} representations of $\tilde G$ also have dimension $\gg q^c$, which gives Proposition \ref{quasi} for the special linear group.

Now we turn to the product theorem for $\SL_m$.  When $m=2$ and $\F_q$ has prime order, this result is due to Helfgott \cite{helfgott-sl2}, who then later established the case when $m=3$ and $\F_q$ has prime order in \cite{helfgott-sl3}.  The case when $m=2$ and $\F_q$ is a prime power was obtained by Dinai \cite{dinai} (see also Varj\'u \cite[sec. 4.1]{varju} for another proof), and the case of general $r$ when $\F_q$ has prime order and $A$ was somewhat small was obtained in \cite{gill}.  The general case is due independently to Pyber and Szab\'o \cite{pyber-szabo} and to the first, second and fourth authors \cite{bgt}.  We state here the main result from \cite{bgt}:

\begin{theorem}\label{mainthm-preciseform} Let $M, K \geq 1$, and let $\G(\k) \subset \GL_m(\k)$ be an absolutely almost simple linear algebraic group of complexity\footnote{An algebraic set in $\k^n$ is said to be \emph{of complexity at most $M$} if it is the boolean combination of the zero sets of at most $M$ polynomials on $\k^n$, each of degree at most $M$, and one also has $n \leq M$.} at most $M$ over an algebraically closed field $\k$, and let $A$ be a $K$-approximate subgroup of $\G(\k)$.  Then at least one of the following statements hold:
\begin{itemize}
\item[(i)] \textup{(}$A$ is not sufficiently Zariski-dense\textup{)} $A$ is contained in an algebraic subgroup $\HH(\k)$ of $\G(\k)$ of complexity $O_{M}(1)$ and dimension strictly less than $\G$.
\item[(ii)] \textup{(}$A$ is small\textup{)} $|A| \ll_M K^{O_M(1)}$.
\item[(iii)] \textup{(}$A$ controlled by $\langle A \rangle$\textup{)} The group $\langle A \rangle$ generated by $A$ is finite, and has cardinality $|\langle A \rangle| \ll_M K^{O_M(1)} |A|$.
\end{itemize}
\end{theorem}

We can now prove Proposition \ref{prod} in the case $\tilde G = \SL_m(\F_q)$.  Let $\delta>0$, and let $\delta'>0$ be sufficiently small depending on $\delta$ and $r$.  We may assume that $|\F_q|$ is sufficiently large depending on $\delta,r$, as the claim is trivial otherwise (since a $K$-approximate group is automatically a group whenever $K<2$).

Let $A$ be a $|\tilde G|^{\delta'}$-approximate subgroup of $\tilde G$ with
\begin{equation}\label{gfd}
 |\tilde G|^\delta \leq |A| \leq |\tilde G|^{1-\delta}.
 \end{equation}
We apply Theorem \ref{mainthm-preciseform} in $\G(\k)$ with $K \coloneqq  |\tilde G|^{\delta'}$ and $M=O(1)$, where $\k$ is some algebraically closed field containing $\F_q$.  We conclude that one of options (i), (ii), (iii) is true.  Option (ii) is ruled out from the lower bound of $|A|$ in \eqref{gfd}, if $\delta'$ is sufficiently small (and $q$ sufficienly large).  If option (iii) holds, then from the upper bound in \eqref{gfd} we see (again for $\delta'$ sufficiently small and $q$ sufficiently large) that $|\langle A \rangle| < |\tilde G|$, and the claim follows in this case.

Finally, suppose that option (i) holds.  Then $\langle A \rangle$ is contained in the algebraic group $\HH$.  We need to recall the Schwartz-Zippel lemma \cite{schwartz}:

\begin{lemma}[Schwartz-Zippel lemma]\label{basic-schwartz-zippel}  Let $P\colon \k^d \to \k$ be a polynomial of degree at most $D$ which is not identically zero.  Then
$$ |\{ x \in \F_q^d: P(x) = 0 \}| \leq d D q^{d-1}.$$
\end{lemma}

Indeed, the $d=1$ case of this lemma follows from the fundamental theorem of algebra, and the higher $d$ cases can then be established by induction (cf. Lemma \ref{prol}(i) below).  We remark that sharper bounds can be obtained (for low values of $D$, at least) using the Lang-Weil estimates \cite{lang}, but we will not need such bounds here (particularly since we will be interested in the case when $D$ is moderately large, in which case it becomes difficult to control the error terms in the Lang-Weil estimates).

We can adapt this lemma to $\G\coloneqq \SL_m$:

\begin{lemma}[Schwartz-Zippel lemma in $\SL_m$]\label{Schwartz}\ \
\begin{itemize}
\item[(i)] If $P\colon \Mat_m(\k) \to \k$ is a polynomial of degree $D \geq 1$ that does not vanish identically on $\G(\k)$, then
$$ |\{ a \in \tilde G: P(a) = 0 \}| \ll D q^{-1} |\tilde G|.$$
\item[(ii)] Similarly, if $P\colon \Mat_m(\k) \times \Mat_m(\k) \to \k$ be a polynomial of degree $D \geq 1$, which does not vanish identically on $\G(\k) \times \G(\k)$, then
$$ |\{ (a,b) \in \tilde G \times \tilde G: P(a,b) = 0 \}| \ll D q^{-1} |\tilde G|^2.$$
\end{itemize}
\end{lemma}

\begin{proof} We first prove (i).   By \eqref{gbwb} we may replace $\tilde G$ by the large Bruhat cell
$\B(\F_q) \w_0 \B(\F_q) = \U(\F_q) \T(\F_q) \w_0 \U(\F_q)$, thus it suffices to show that
\begin{align*}  |\{ (u_1,h,u) \in \U(\F_q) \times \T(\F_q) \times \U(\F_q): & P(u_1,h,u) = 0 \}|  \\
& \ll D q^{-1} |\U(\F_q)| |\T(\F_q)| |\U(\F_q)|.\end{align*}
We can parameterise an element $u$ of $\U(\F_q)$ by $q^{m(m-1)/2}$ independent coordinates in $\F_q$ by using the strictly upper triangular entries $u_{ij}, 1 \leq i < j < m$, of that element.  An element in $\T(\F_q)$ can also be parameterised by $m-1$ independent coordinates $t_1,\ldots,t_{m-1}$ in $\F_q^\times \coloneqq  \F_q \backslash \{0\}$ by identifying such a tuple of coordinates with the element
$$ \operatorname{diag}( t_1,\ldots,t_{m-1}, \frac{1}{t_1\cdot \ldots \cdot t_{m-1}})$$
in $T(\F_q)$.  In particular, $|U(\F_q)| = q^{m(m-1)/2}$ and $|T(\F_q)| = (1+O(1/q)) q^{m-1}$.  We can then view $P(u_1,h,u)$ as a polynomial
$$ Q\left( (u_{1,ij})_{1 \leq i < j \leq m}, (t_i)_{i=1}^{m-1}, (u_{ij})_{1 \leq i < j \leq m} \right) $$
of degree $O(D)$ in $\frac{m(m-1)}{2} + (m-1) + \frac{m(m-1)}{2}$ coordinates $(u_{1,ij})_{1 \leq i < j \leq m}$, $(t_i)_{i=1}^{m-1}$, $(u_{ij})_{1 \leq i < j \leq m}$ divided by a monomial in the $t_1,\ldots,t_{m-1}$ coordinates, where the $u_{1,ij}, u_{ij}$ range in $\F_q$ and the $t_i$ range in $\F_q^\times$.  Clearing denominators, it thus suffices to establish the bound
\begin{align*}
|\{
 (u_{1,ij})_{1 \leq i < j \leq m} \times (t_i)_{i=1}^{m-1} \times (u_{ij})_{1 \leq i < j \leq m} \in & \; \F_q^{m(m-1)/2} \times \\  \times (\F_q^\times)^m \times \F_q^{m(m-1)/2}: 
 Q( (u_{1,ij})_{1 \leq i < j \leq m}, & (t_i)_{i=1}^{m-1}, (u_{ij})_{1 \leq i < j \leq m} ) = 0 \}|
\\  & \ll D q^{-1} q^{m(m-1)/2} q^{m-1} q^{m(m-1)/2}.
\end{align*}
If $Q$ is non-vanishing, then this is immediate from Lemma \ref{basic-schwartz-zippel}; and when $Q$ is vanishing, then $P$ vanishes on the Zariski-dense subset $\B \w_0 \B = \U \T \w_0 \U$ of $\G$ and thus vanishes on all of $\G(\k)$, a contradiction.  This gives (i).

Now we use (i) to prove (ii).  By hypothesis, one can find $(a_0,b_0) \in \G(\k) \times \G(\k)$ such that $P(a_0,b_0) \neq 0$.  From (i), we have
$$ |\{ a \in \tilde G: P(a,b_0) = 0 \}| \ll D^{O(1)} q^{-1} |\tilde G|.$$
On the other hand, for each $a \in \tilde G$ with $P(a,b_0) \neq 0$, another application of (i) gives
$$ |\{ b \in \tilde G: P(a,b) = 0 \}| \ll D^{O(1)} q^{-1} |\tilde G|.$$
Summing over all $a$, we obtain (ii) as required.
\end{proof}

From part (i) of this lemma we see that
$$ |\langle A \rangle| \ll q^{-1} |\tilde G|.$$
If $\F_q$ is sufficiently large, we conclude that $A$ generates a proper subgroup of $\tilde G$, as required.  This concludes the proof of Proposition \ref{prod} for special linear groups (and hence for projective special linear groups).  (Part (ii) of the above lemma will be used at a later stage of the argument.)

Finally, we need to establish the non-concentration estimate, Proposition \ref{nonc}, for the projective special linear group $G$.  Again, as $\tilde G$ is a bounded cover of $G$, it suffices to establish the analogous claim for the special linear group $\tilde G = \SL_{m}(\F_q)$.

We will need the following (rough) description of the subgroups of $\tilde G$.

\begin{proposition}[Subgroups of $\SL_{m}(\F_q)$]\label{sob}  For any proper subgroup $H$ of $\tilde G$, one of the following statements hold:
\begin{itemize}
\item[(i)] \textup{(}Structural case\textup{)} $H$ lies in a proper algebraic subgroup of $\G$ of complexity $O(1)$.
\item[(ii)] \textup{(}Subfield case\textup{)} Some conjugate of $H$ lies in $\G(\F_{q'})$, where $\F_{q'}$ is a proper subfield of $\F_q$.
\end{itemize}
\end{proposition}

\begin{proof}  This is a special case of a more general statement about maximal subgroups of finite simple groups of Lie type; see Lemma \ref{max}.
\end{proof}

Let us call $H$ a \emph{structural} subgroup if the first conclusion of Proposition \ref{sob} holds, and a \emph{subfield} subgroup if the second conclusion holds.  Note that it is certainly possible for $H$ to be simultaneously structural and subfield; we will take advantage of this overlap in a subsequent part of the paper when dealing with the twisted group case.

Set $n \coloneqq  2 \lfloor c_0 \log |\tilde G| \rfloor$ for some sufficiently small $c_0>0$.  By Proposition \ref{sob}, to prove Proposition \ref{nonc} for the projective special linear group, it suffices to establish the claims
\begin{align}\nonumber
 \P_{a,b \in \tilde G}( \P_{w \in W_{n,2}}(w(a,b) \in H) \leq |\tilde G|^{-\gamma} & \hbox{ for all structural } H < \tilde G ) \\ & = 1 - O(|\tilde G|^{-\gamma}),\label{structural-case}
\end{align}
and
\begin{align}\nonumber
 \P_{a,b \in \tilde G}( \P_{w \in W_{n,2}}(w(a,b) \in H) \leq |\tilde G|^{-\gamma} & \hbox{ for all subfield } H < \tilde G ) \\ & = 1 - O(|\tilde G|^{-\gamma})\label{subfield-case}
\end{align}
for some sufficiently small $\gamma>0$.\vspace{11pt}

\textsc{The structural case}. We now establish \eqref{structural-case}, following the arguments of Bourgain and Gamburd \cite{bourgain-gamburd}.  We rewrite this estimate as
$$  \P_{a,b \in \tilde G}( \P_{w \in W_{n,2}}(w(a,b) \in H) > |\tilde G|^{-\gamma} \hbox{ for some structural } H < \tilde G ) \ll |\tilde G|^{-\gamma}.$$
Note that if $a,b$ is such that
$$ \P_{w \in W_{n,2}}(w(a,b) \in H) > |\tilde G|^{-\gamma}$$
for some structural $H < \tilde G$, then we have
$$ \P_{w,w' \in W_{n,2}}(w(a,b), w'(a,b) \in H) > |\tilde G|^{-2\gamma}.$$
Thus, by Markov's inequality, it suffices to show that
\begin{equation}\label{wawa}
 \P_{a,b \in \tilde G; w,w' \in W_{n,2}}( w(a,b), w'(a,b) \in H \hbox{ for some structural } H < \tilde G ) \ll |\tilde G|^{-3\gamma}.
\end{equation}

Let $e_1,e_2$ be generators of a free group $F_2$.
Let us first dispose of the contribution when $w(e_1,e_2), w'(e_1,e_2)$ commute.

\begin{lemma}[Generic non-commutativity]\label{noncom}  One has
$$ \P_{w,w' \in W_{n,2}}( w(e_1,e_2) w'(e_1,e_2) = w'(e_1,e_2) w(e_1,e_2) ) \ll \exp(-cn)$$
for some absolute constant $c>0$.
\end{lemma}

\begin{proof} By the Nielsen-Schreier theorem, this case only occurs when $w(e_1,e_2)$ and $w'(e_1,e_2)$ lie in a cyclic group, which means that $w(e_1,e_2) = x^a$ and $w'(e_1,e_2) = x^b$ for some integers $a,b$ and some element $x \in F_2$.

It is a classical fact \cite{kesten} that a random walk on the free group $F_2$ will return to the identity in time $n$ with probability $O(\exp(-cn))$ for some absolute constant $c>0$.  (Indeed, as exactly half of the path has to consist of backtracking, one has a crude bound of $\binom{n}{n/2} 3^{n/2} = O( \exp(-cn) 4^n)$ for the number of paths that return to the identity at time $n$.)   In particular, we may assume that $w(e_1,e_2)$ and $w'(e_1,e_2)$ are not equal to the identity.  This forces $x$ to be a non-identity word, and $a,b$ to have magnitude at most $n$.  There are thus $O(n^2)$ choices for $a,b$, and once $a, b$ is fixed, $w(e_1,e_2)$ uniquely determines $x$ and hence $w'(e_1,e_2)$.  On the other hand, by another appeal to the above classical fact, any given value of $w'(e_1,e_2)$ is attained by at most $O(\exp(-cn))$ choices of $w'$.  The claim follows.
\end{proof}

In view of the above lemma, and of the choice of $n$, the contribution of the commuting case to \eqref{wawa} is acceptable.  Thus, by Fubini's theorem, it will suffice to show that
\begin{equation}\label{wawa-2}
 \P_{a,b \in \tilde G}( w(a,b), w'(a,b) \in H \hbox{ for some structural } H < \tilde G ) \ll |\tilde G|^{-3\gamma}.
\end{equation}
whenever $w,w'\in W_{n,2}$ are such that $w(e_1,e_2)$ and $w'(e_1,e_2)$ do not commute.

Fix $w,w'$.  If $w(a,b), w'(a,b)$ lie in the same structural subgroup $H$, then they are contained in a proper algebraic subgroup of $\G$ of complexity $O(1)$.  We now convert this claim into an algebraic constraint on $w(a,b), w'(a,b)$, with an eye towards eventually applying the Schwartz-Zippel lemma
(Lemma \ref{Schwartz}).  It would be very convenient if the set of all pairs $(x,y) \in \G(\k) \times \G(\k)$ for which $x,y$ were contained in a proper algebraic subgroup of $\G(\k)$ was a proper algebraic subset of $\G(\k) \times \G(\k)$.  Unfortunately, in positive characteristic this is not the case; for instance, if we replaced $\k$ with a locally finite field such as $\overline{\F_q}$, then every pair $x,y \in \G(\k) \times \G(\k)$ would be contained in $\G(k_{xy})$ for some finite field $k_{xy}$, yet $\G(k_{xy})$ is a finite group and thus obviously a proper algebraic subset of $\G(\k) \times \G(\k)$.

However, we can obtain a usable substitute for the above (false) claim by enforcing a bound on the complexity of the proper algebraic subsets involved.  More precisely, we have the following.

\begin{proposition}\label{crit}  Let $N \geq 1$ be an integer, and let $\k$ be an algebraically closed field.  Let $\G(\k) \subset \GL_d(\k)$ be a connected linear algebraic group of complexity $O(1)$.  Then there exists a closed algebraic subset $X_N(\k)$ of $\G(\k) \times \G(\k)$ of complexity $O_N(1)$ with the following properties:
\begin{itemize}
\item[(i)] If $x,y \in \G(\k)$ are such that $x$ and $y$ are contained in a proper algebraic subgroup of $\G(\k)$ of complexity at most $N$, then $(x,y) \in X_N(\k)$.
\item[(ii)] Conversely, if $(x,y) \in X_N(\k)$, then $x$ and $y$ are both contained in a proper algebraic subgroup of $\G(\k)$ of complexity $O_N(1)$.
\end{itemize}
\end{proposition}


\begin{proof}  Let $D$ be a large integer (depending on $N, d$) to be chosen later.  We view $\G(\k)$ as a subset of the ring $\Mat_m(\k)$ of $m \times m$ matrices, which is also a vector space over $\k$ of dimension $m^2 = O(1)$.  Let $V$ be the space of polynomials $P\colon \Mat_m(\k) \to \k$ of degree at most $D$ on $\Mat_m(k)$; then $V$ is a vector space over $\k$ of dimension $O_D(1)$, and $\G(\k)$ acts on $V$ by left-translation, thus
$$ \rho(g) P(x) \coloneqq  P( g^{-1} x)$$
for all $g \in \G(\k)$ and $P \in V$. Thus we have a homomorphism $\rho\colon \G(\k) \to \End(V)$ from $\G(\k)$ to the space $\End(V)$ of endomorphisms of $V$, which is another vector space over $\k$ of dimension $O_D(1)$.

We now define $X_N(\k)$ to be the set of all pairs $(x,y) \in \G(\k) \times \G(\k)$ such that
\begin{equation}\label{spain}
\operatorname{span}( \rho(\langle x,y\rangle) ) \neq \operatorname{span}( \rho(\G(\k)) ),
\end{equation}
where $\operatorname{span}$ denotes the linear span in $\End(V)$.

The fact that $X_N(\k)$ is a closed algebraic subset will be deferred to the end of the proof.  For now, let us verify property (i).  Suppose $x,y$ lie in a proper algebraic subgroup $H$ of $\G(\k)$ of complexity at most $N$.  We view $H$ as a subvariety of $\Mat_m(\k)$ of complexity at most $N$.  Let $I_H$ be the radical ideal of polynomials of $\Mat_m(\k)$ that vanish on $H$.
By a result\footnote{One can also obtain the degree bound on the polynomials here via an ultraproduct argument combined with the Hilbert basis theorem, as in \cite[Appendix A]{bgt}.} of Kleiman \cite[Corollary 6.11]{kleiman}, we know that $I_H$ has a generating set, say $f_1,\ldots,f_k$, of polynomials of degree $O_N(1)$. In particular, if $D$ is large enough, then $f_1,\ldots,f_k$ all lie in $V$.

Now, if $g$ lies in $\langle x,y \rangle$, then $g$ lies in $H$, and so $\rho(g)$ preserves $I_H \cap V$.  Thus,
$\operatorname{span}( \rho(\langle x,y\rangle) )$ preserves $I_H \cap V$ also.  Now suppose for contradiction that $(x,y)$ does not lie in $X_N(\k)$, then by \eqref{spain}, $\rho(g)$ preserves $I_H \cap V$ for all $g \in \G$.  In particular, $\rho(g) f_i \in I_H$ for all $i=1,\ldots,k$ and $g \in \G$.  But this implies that the $f_i$ all vanish on $\G$; since the $f_i$ generate $I_H$, this forces all polynomials that vanish on $H$, to vanish on $\G$ as well.  But $H$ is a proper subvariety of $\G$, giving the desired contradiction.

Next, we verify property (ii).  Suppose that $(x,y)$ lies in $X_N(\k)$.  By \eqref{spain} and duality, there is thus a linear functional $\phi\colon \End(V) \to \k$ which vanishes on $\rho(\langle x,y \rangle)$ but which does not vanish identically on $\rho(\G(\k))$.  Thus, the group $\langle x, y \rangle$ is contained in the set $\{ g \in \G(\k): \phi(\rho(g)) = 0 \}$, which is a proper algebraic subvariety of $\G(\k)$ of complexity $O_D(1)$.  Applying the ``escape from subvarieties'' lemma (see \cite[Lemma 3.11]{bgt}), this implies that $\langle x,y \rangle$ is contained in a proper algebraic \emph{subgroup} of $\G(k)$ of complexity $O_D(1)$.  If we select $D$ sufficiently large depending on $N$, $m$, we obtain the claim (ii).

Finally, we show that $X_N(\k)$ is a closed subvariety of complexity $O_D(1)$.  Consider the non-decreasing sequence of subspaces
$$\operatorname{span}( \rho(B_n(x,y) )$$
of $\End(V)$ for $n=0,1,2,\ldots$, where $B_n(x,y)$ denotes all words of $x$, $y$, $x^{-1}$, $y^{-1}$ of length at most $n$.  From the pigeonhole principle, we conclude that
$$\operatorname{span}( \rho(B_{n+1}(x,y) ) ) = \operatorname{span}( \rho(B_n(x,y) ) )$$
for some $n \leq \dim \End(V) = O_D(1)$.  But then $\rho(x^{\pm 1}), \rho(y^{\pm 1})$ leave\\ $\operatorname{span}( \rho(B_n(x,y) ))$ invariant, which implies in particular that
$$ \operatorname{span}( \rho(B_n(x,y) ) = \operatorname{span}( \rho(B_{\dim \End(V)}(x,y) ) ) = \operatorname{span}( \rho(\langle x,y\rangle ) ).$$
Thus, we may rewrite the condition \eqref{spain} as the condition that
$$  \operatorname{span}( \rho(B_{\dim \End(V)}(x,y) ) ) = \operatorname{span}( \rho(\G(k)) )$$
or equivalently that the elements $\rho(w(x,y))$ of $\End(V)$ for $w$ a word of length at most $\End(V)$ does not have full rank in $\operatorname{span}( \rho(B_{\dim \End(V)}(x,y) ) )$.  This is clearly an algebraic constraint on $x,y$ and establishes that $X_N(\k)$ is a closed subvariety of complexity $O_D(1)$ as required.\end{proof}

We apply the above proposition with $\G \coloneqq  \SL_{m}$ and $N = O(1)$ sufficiently large depending on the rank $r$.  By the preceding discussion, we know that if $w(a,b), w'(a,b)$ lie in a common structural subgroup, then the pair $(w(a,b),w'(a,b))$ lies in $X_N$, and thus $(a,b)$ lies in the set
$$ \Sigma_{w,w'}(\k) \coloneqq  \{ (a,b) \in \G(\k) \times \G(\k): (w(a,b),w'(a,b)) \in X_N \}$$
As $w,w'$ are words of length at most $n$, this is a closed subvariety of $\G(\k) \times \G(\k)$ of complexity $O(n)$.

We now make the crucial observation (using Theorem \ref{sec3-key-prop}) that $\Sigma_{w,w'}(\k)$ is a \emph{proper} subvariety of $\G(\k) \times \G(\k)$ when $w(e_1,e_2),w'(e_1,e_2)$ do not commute.  Indeed, by Theorem \ref{sec3-key-prop}, the hypothesis that $\k$ is uncountable, and the non-commutativity of $w(e_1,e_2)$ and $w'(e_1,e_2)$, we can find $(a,b) \in \G(\k) \times \G(\k)$ such that $w(a,b), w'(a,b)$ generate a Zariski-dense subgroup of $\G(\k)$.  By Proposition \ref{crit}(ii), this implies that $(a,b)$ lies outside of $\Sigma_{w,w'}(\k)$, and so $\Sigma_{w,w'}(\k)$ is a proper subvariety of $\G(\k) \times \G(\k)$ as required.

Applying the Schwartz-Zippel lemma (Lemma \ref{Schwartz}), we have
$$ |\Sigma_{w,w'}(\k) \cap \tilde G^2| \ll n^{O(1)} q^{-1} |\tilde G|^2.$$
By the choice of $n$ (recall $n \coloneqq  2 \lfloor c_0 \log |\tilde G| \rfloor$), we thus have (for $c_0$, $\gamma$ small enough) that
$$ |\Sigma_{w,w'}(\k) \cap \tilde G^2| \ll |\tilde G|^{2-3\gamma},$$
and \eqref{wawa-2} follows.  This concludes the proof of \eqref{structural-case}.\vspace{11pt}

\textsc{The subfield case.} It remains to establish \eqref{subfield-case}.  The starting point is the observation that if $g \in \tilde G$ is conjugate to an element of $\G(\F_{q'})$ for some subfield $\F_{q'}$ of $\F_q$, then the coefficients $\gamma_1(g),\ldots,\gamma_{m}(g)$ of the characteristic polynomial of the matrix $g$
$$\det(X \Id_m - g)=X^m +\gamma_i(g)X^{m-1} + \ldots \gamma_1(g)X + \gamma_m(g)$$
will lie in $\F_{q'}$.  Note that the number of proper subfields of $\F_q$ is at most $O(\log q ) = O( \log |\tilde G| )$.  Thus, by the union bound, it suffices to show that
$$
\P_{a,b \in \tilde G}( \P_{w \in W_{n,2}}(\gamma_i(w(a,b)) \in \F_{q'} \hbox{ for all } 1 \leq i \leq m ) > |\tilde G|^{-\gamma} ) \ll |\tilde G|^{-2\gamma}$$
(say) for each proper subfield $\F_{q'}$ of $\F_q$.  By Markov's inequality, it suffices to show that
$$ \P_{a,b \in \tilde G; w \in W_{n,2}} (\gamma_i(w(a,b)) \in \F_{q'} \hbox{ for all } 1 \leq i \leq m ) \ll |\tilde G|^{-3\gamma}.$$

Fix $\F_{q'}$, and let $e_1,e_2$ be the generators of a free group $F_2$.  As observed previously in the proof of Lemma \ref{noncom}, $w(e_1,e_2)$ will be the identity with probability $O(\exp(-cn))$ for some $c>0$.  By the choice of $n$, it thus suffices to show that
\begin{equation}\label{pabgf}
 \P_{a,b \in \tilde G} (\gamma_i(w(a,b)) \in \F_{q'} \hbox{ for all } 1 \leq i \leq m ) \ll |\tilde G|^{-3\gamma}
\end{equation}
whenever $w \in W_{n,2}$ is such that $w(e_1,e_2) \neq 1$.

Let us first consider the probability
$$ \P_{a,b \in \tilde G} (\gamma_i(w(a,b)) = x_i ) $$
for some fixed $1 \leq i \leq m$ and $x_i \in \F_{q'}$ with $x_i \neq \gamma_i(1)$.  Observe that
$$ \Sigma_{w,i,x_i} \coloneqq  \{ (a,b) \in \G \times \G: \gamma_i(w(a,b)) = x_i \}$$
is an algebraic variety of complexity $O(n)$.  Since $\gamma_i(w(1,1)) = \gamma_i(1) \neq x_i$, this variety is a proper subvariety of $\G \times \G$.  Applying the Schwartz-Zippel lemma (Lemma \ref{Schwartz}), we conclude that
$$ |\Sigma_{w,i,x_i}(\F_q)| \ll n^{O(1)} q^{-1} |\tilde G|^2,$$
and thus
$$ \P_{a,b \in \tilde G} (\gamma_i(w(a,b)) = x_i ) \ll n^{O(1)} q^{-1}.$$
Summing over all $x_i \in \F_{q'} \backslash \{\gamma_i(1)\}$, we can bound the left-hand side of \eqref{pabgf} by
$$ \P_{a,b \in \tilde G} (\gamma_i(w(a,b)) = \gamma_i(1) \hbox{ for all } 1 \leq i \leq m ) + O( n^{O(1)} q' q^{-1} ).$$
As $\F_{q'}$ is a proper subfield of $\F_q$, we have $q' \leq q^{1/2}$.  As such, the error term $O( n^{O(1)} q' q^{-1} )$ is $O(|\tilde G|^{-3\gamma})$ if $\gamma, c_0$ are small enough.  It thus suffices to show that
$$ \P_{a,b \in \tilde G} (\gamma_i(w(a,b)) = \gamma_i(1) \hbox{ for all } 1 \leq i \leq m ) \ll |\tilde G|^{-3\gamma}.$$
By the Cayley-Hamilton theorem, if $\gamma_i(w(a,b)) = \gamma_i(1)$ for all $1 \leq i \leq m$, then
$$ (w(a,b)-1)^{m} = 0.$$
Observe that
$$ \Sigma'_w \coloneqq  \{ (a,b) \in \G \times \G: (w(a,b)-1)^{m} = 0 \}$$
is an algebraic variety of complexity $O(n)$.  By repeating the previous arguments, it thus suffices to establish that $\Sigma'_w$ is a proper subvariety of $\G \times \G$.  Suppose that this is not the case; then for all $a,b \in \G \times \G$, one has $(w(a,b)-1)^{m}=0$; thus $w(a,b)$ is unipotent.  However, the space of all unipotent matrices forms a proper subvariety of $\G = \SL_m$.  Furthermore, by a theorem of Borel \cite{borel-dom} and the assumption that $w$ is non-trivial, the word map $w\colon \G \times \G \to \G$ is dominant.  Thus $\Sigma'_w$ is a proper subvariety of $\G \times \G$ as required.  This concludes the proof of Theorem \ref{mainthm} in the case when $G$ is a projective special linear group $G=A_r(q)$.

\begin{remark}\label{word}  An inspection of the above arguments shows not only that $\{a,b\}$ are $\eps$-expanding with probability $1-O(|G|^{-\delta})$, but furthermore that $\{w_1(a,b), w_2(a,b)\}$ are $\eps$-expanding with probability $1-O(|G|^{-\delta})$ for any non-commuting pair of words $w_1, w_2 \in F_2$ of length at most $|G|^\delta$.  The proof is essentially the same, with the only changes required being in the non-concentration portion of the argument, when one replaces $a,b$ by $w_1(a,b)$, $w_2(a,b)$ in all of the events whose probability is being computed. For instance the probability
$$ \P_{a,b \in \tilde G}( \P_{w \in W_{n,2}}(w(a,b) \in H) \leq |\tilde G|^{-\gamma} \hbox{ for all structural } H < \tilde G ) $$
must be replaced with
\begin{align*} \P_{a,b \in \tilde G}( \P_{w \in W_{n,2}}(w(w_1(a,b),w_2(a,b)) & \in H) \leq |\tilde G|^{-\gamma} \\ & \hbox{ for all structural } H < \tilde G ).\end{align*}
But one can easily verify that the above arguments proceed with essentially no change with this substitution.  Note that by the Nelson-Schreier theorem, $w_1$ and $w_2$ again generate a free group; in particular, if $w, w'$ are non-commuting words, then $w(w_1(a,b),w_2(a,b))$ and $w'(w_1(a,b), w_2(a,b))$ are also non-commuting words (and in particular non-trivial), allowing the crucial use of Theorem \ref{sec3-key-prop} (and also Borel's domination theorem) to continue to apply in this case to ensure that the variety
\begin{align*} & \Sigma_{w,w'; w_1,w_2}(\k) \coloneqq  \\ & \{ (a,b) \in \G(\k) \times \G(\k): (w(w_1(a,b),w_2(a,b)), w'(w_1(a,b),w_2(a,b))) \in X_N \}\end{align*}
remains proper.  Also,  and when viewed as polynomials in $a,b$, the degrees of $w(w_1(a,b),w_2(a,b)), w'(w_1(a,b),w_2(a,b))$ are larger by a factor of $O(|G|^\delta)$ than the degrees of $w(a,b), w'(a,b)$, allowing the Schwartz-Zippel estimates to stay essentially the same.  We leave the details to the reader.  One can similarly adapt the argument in later sections for more general finite simple groups of Lie type (including the triality groups ${}^3 D_4(q)$, which require a separate argument based on the same general techniques); again, we leave the details to the interested reader.
\end{remark}

\section{The general case}\label{sec1a}

Having concluded the proof of Theorem \ref{mainthm} in the model case $G = A_r(q)$, we now turn to the general case of finite simple groups of Lie type with some rank $r$.  As before, we allow all implied constants in the $O()$ notation to depend on $r$.

We begin by defining more precisely what we mean by a finite simple group of Lie type.  The reader may consult \cite{carter,gls3,wilson} for a more thorough treatment of this material.  Our notation has some slight differences with that in \cite{carter} or \cite{gls3}; see Remark \ref{diff} below.

\begin{definition}[Dynkin diagram]  A \emph{Dynkin diagram} is a graph of the form $A_r$ for $r \geq 1$, $B_r$ for $r \geq 2$, $C_r$ for $r \geq 3$, $D_r$ for $r \geq 4$, $E_6$, $E_7$, $E_8$, $F_4$, or $G_2$ (see Figure \ref{dynkin}).
\end{definition}

\begin{figure}\label{dynkin}
\begin{center}
\includegraphics[scale=0.3]{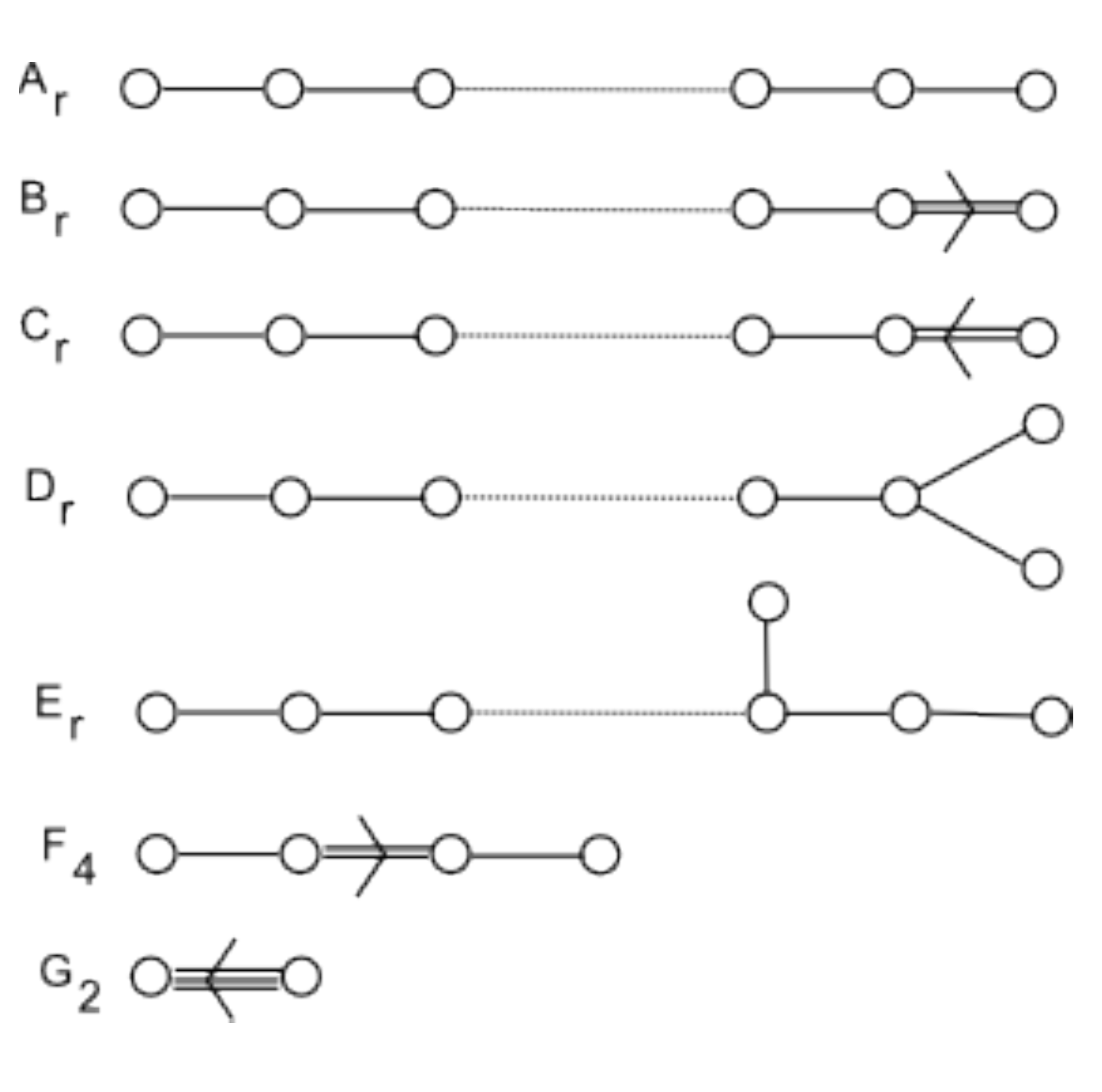}
\caption{Dynkin diagrams.  The subscript $r$ denotes the number of vertices.}
\end{center}
\end{figure}

We observe that there are only a small number of possible non-trivial graph automorphisms $\rho\colon \D \to \D$ of a Dynkin diagram $\D$.  Specifically, for the Dynkin diagrams $A_r, D_r, E_6, B_2, G_2, F_4$, there is a graph automorphism of order two, and for $D_4$ there is an additional graph automorphism of order three, and these are the only non-trivial graph automorphisms (up to conjugation, in the $D_4$ case).

\begin{definition}[Finite simple group of Lie type]\label{dfs} Let $\D$ be a Dynkin diagram, and let $\rho\colon \D \to \D$ be a graph automorphism of order $d$ (thus $d \in \{1,2,3\}$).  Let $\F_q$ be a finite field of order $q$ and characteristic $p$, and let $\k$ be an algebraically closed field\footnote{The exact choice of $\k$ is not terribly important, but it will be technically convenient to use an uncountable field $\k$ here, rather than the algebraic closure of $\F_q$, in order to easily use Theorem \ref{sec3-key-prop}.} containing $\F_q$.  Let $\G(\k)$ be a connected, absolutely almost simple algebraic group associated with the Dynkin diagram $\D$, so that $\rho\colon \G(\k) \to \G(\k)$ also acts\footnote{In the case when $\rho$ is the non-trivial automorphism of $B_2$ or $F_4$, the automorphism of $\G(\k)$ only exists in characteristic two; similarly, if $\rho$ is the non-trivial automorphism of $G_2$, the automorphism of $\G(\k)$ only exists in characteristic three; see \cite[Chapter 12]{carter}.} as an automorphism of $\G(\k)$ that fixes $\G(\F_q)$. Let $\tau\colon \F_q \to \F_q$ be a field automorphism of $\F_q$.  This field automorphism then induces a group automorphism $\tau\colon \G(\F_q) \to \G(\F_q)$ of $\G(\F_q)$ that commutes with $\rho$.  Set $\sigma \coloneqq  \tau \rho$, and suppose that $\sigma$ also is of order\footnote{If $\rho$ is trivial, this forces $\tau$ to be trivial also.  If $\rho$ is a non-trivial automorphism on $A_r$, $D_r$, or $E_6$, this forces $q = \tilde q^d$ for some integer $\tilde q$, and $\tau$ to be (up to conjugation) the Frobenius map $x \mapsto x^{\tilde q}$.  If $\rho$ is the non-trivial automorphism on $F_4$ or $G_2$, this forces $q = p \tilde q^d$ for some integer $\tilde q$, and $\tau$ to be (up to conjugation) the Frobenius map $x \mapsto x^{\tilde q}$; see \cite[14.1]{carter}.} $d$.  Let $\G(\F_q)^\sigma \coloneqq  \{ g \in \G(\F_q): \sigma(g) = g\}$ be the fixed points of $\G(\F_q)$ with respect to this automorphism.  Let $\tilde G \coloneqq  [\G(\F_q)^\sigma, \G(\F_q)^\sigma]$ be the derived group, and let $Z(\tilde G)$ be the centre of $\tilde G$.  Let $G$ be the quotient group $G \coloneqq  \tilde G/Z(\tilde G)$.  If this group is simple\footnote{It turns out that $G$ will be simple except for a finite number of exceptions, and specifically $A_1(2)$, $A_1(3)$, ${}^2 A_2(4)$, and ${}^2 B_2(2)$; see \cite[Theorem 2.2.7]{gls3}.  But our results are only non-trivial in the asymptotic regime when $q$ is large, so these exceptional cases are of no interest to us.}, we call it a \emph{finite simple group of Lie type}.  If we quotient $\tilde G$ by some subgroup of $Z(\tilde G)$, we obtain a perfect central extension of $G$, which we call a \emph{finite quasisimple group of Lie type}.  For instance $G$ and $\tilde G$ are finite quasisimple groups of Lie type.

We define the \emph{rank} $r$ of $G$ to be the rank of $\G$, that is to say the dimension of the maximal torus of $\G$.  We refer to the order $d$ of $\rho$ (or $\sigma$) as the \emph{twist order} $d$.  The group $G$ itself will be denoted ${}^d \D_r(q)$, or simply $\D_r(q)$ if the twist order $d$ is $1$.

The group $\G(\k)$ will be referred to as the \emph{mother group} of $G$, and $\tilde G$ will be the \emph{bounded cover} of $G$.
\end{definition}

It turns out that finite simple groups of Lie type can be organised into three families:

\begin{enumerate}
\item \emph{Untwisted groups} $A_r(q)$, $B_r(q)$, $C_r(q)$, $D_r(q)$, $E_6(q)$, $E_7(q)$, $E_8(q)$, $F_4(q)$, $G_2(q)$, in which $\rho$ is trivial, and $\sigma=\tau$ is the Frobenius field automorphism $x \mapsto x^q$ associated with a finite field $\F_q$ of characteristic $p$, and $\k$ will be its algebraic closure.  For instance, the projective special linear group $G = A_r(q) = \PSL_{r+1}(\F_q)$ studied in the preceding section is of this form (with $\G = \SL_{r+1}$).
\item \emph{Steinberg groups ${}^2 A_r(\tilde q^2)$, ${}^2 D_r(\tilde q^2)$, ${}^2 E_6(\tilde q^2)$, ${}^3 D_4(\tilde q^3)$}, in which $\rho$ is a non-trivial graph automorphism of $\D = A_r, D_r, E_6$ of order $d=2,3$, $q = \tilde q^d$ is a $d^{\operatorname{th}}$ power, and $\tau$ is the Frobenius automorphism $x \mapsto x^{\tilde q}$. For instance, the projective special unitary group ${}^2 A_r(\tilde q^2) = \PSU_{r+1}(\F_{\tilde q^2})$ from \eqref{surq} is of this form, with $d=2$ and $\G = \SL_{r+1}$.  We also highlight for special mention the \emph{triality groups} ${}^3 D_4(\tilde q^3)$, which are the only group with a graph automorphism of order $3$, and which will need to be treated separately in our analysis.
\item \emph{Suzuki-Ree groups} ${}^2 B_2(2^{2k+1}), {}^2 F_4(2^{2k+1}), {}^2 G_2(3^{2k+1})$, in which $\rho$ is the non-trivial automorphism of $\D = B_2, F_4, G_2$ (assuming characteristic $p=2$ in the $B_2,F_4$ cases and characteristic $p=3$ in the $G_2$ case), $q = p \theta^2$ for some $\theta = p^k$, and $\tau$ is the Frobenius map $x \mapsto x^{\theta}$.  The groups ${}^2 B_2(2^{2k+1})$ are also referred to as \emph{Suzuki groups}, while ${}^2 F_4(2^{2k+1})$ and ${}^2 G_2(3^{2k+1})$ are referred to as \emph{Ree groups}.
\end{enumerate}

We refer to the Steinberg, Suzuki, and Ree groups collectively as \emph{twisted} finite simple groups of Lie type.  The distinction between the Steinberg groups and the Suzuki-Ree groups ultimately stems from the fact that the former groups have Dynkin diagrams $\D$ from the ADE family, so that their roots all have the same length, whereas the latter groups have diagrams in which the roots have two different possible lengths.

\begin{remark}\label{diff}  We remark that our notation here is slightly different from that in \cite{carter} or \cite{gls3}.  In \cite{gls3}, the group that we would call ${}^d \D(q)$ is instead denoted ${}^d \D(q^{1/d})$ (in particular, with the convention in \cite{gls3}, the ``$q$'' parameter becomes irrational in the Suzuki-Ree cases).  Also, in \cite{gls3} the group $\tilde G$ is not taken to be the derived group of $\G(\F_q)^\sigma$, but is instead taken to be the group $O^{p'}(\G(\F_q)^\sigma)$ generated by the elements of $\G(\F_q)^\sigma$ of order a power of $p$, where $p$ is the characteristic of $\F_q$.  However, the quotient group $\G(\F_q)^\sigma / O^{p'}(\G(\F_q)^\sigma)$ acts faithfully on $O^{p'}(\G(\F_q)^\sigma)$ with an action generated by diagonal automorphisms, and is thus abelian (see \cite[Lemmas 2.5.7, 2.5.8]{gls3}), while the group $O^{p'}(\G(\F_q)^\sigma)$ is almost always perfect, with the only exceptions being $A_1(2)$, $A_1(3)$, ${}^2 A_2(4)$, ${}^2 B_2(2)$, ${}^2 G_2(3)$, and ${}^2 F_4(2)$ (see \cite[Theorem 2.2.7]{gls3}).  Thus, outside\footnote{For instance, the Tits group would be considered a finite simple group ${}^2 F_4(2)$ of Lie type under our conventions, whereas in \cite{gls3} it would be considered an index two subgroup of ${}^2 F_4(2^{1/2})$.}
 of finitely many exceptions, we have $O^{p'}(\G(\F_q)^\sigma) = \tilde G$ and so our notation coincides with that of \cite{gls3} except for replacing $q^{1/d}$ with $q$.  In \cite{carter}, the notation ${}^d \D(q)$ is used instead of ${}^d \D(q^{1/d})$ (thus matching our notation and not that of \cite{gls3}), but $\tilde G$ is instead replaced by the subgroup $G_1$ of $\G(\F_q)^\sigma$ generated by the intersection of that group with the unipotent subgroups $U,V$ of $G$ generated by the positive and negative roots; however the group $G_1$ can be shown to be identical with $O^{p'}(\G(\F_q)^\sigma)$ (see \cite[Theorem 2.3.4]{gls3}).  As we are only interested in the asymptotic regime when $q$ is large, the finite number of exceptions between the conventions here and that in \cite{carter}, \cite{gls3} will not be relevant, except for the fact that we index the finite simple groups by $q$ instead of $q^{1/d}$.
\end{remark}

Note that even after fixing the Dynkin diagram and the field $\k$, there is some flexibility in selecting the group $\G$; for
 instance, with the Dynkin diagram $A_n$, one can take $\G = \PGL_{r+1}$ or $\G = \SL_{r+1}$, leading to two slightly different bounded covers $\tilde G = \PGL_{r+1}(q)$ or $\tilde G = \SL_{r+1}(q)$ for the same simple group $G = \PSL_{r+1}(q)$.  But up to isomorphism, there are only $O(1)$ possible choices for $\G$ and hence for the bounded cover $\tilde G$; see \cite[2.2]{gls3}.  The group $\G$ can always be taken to be a linear algebraic group i.e. an algebraic subgroup of $\GL_m(\k) \subseteq \Mat_m(\k) \cong k^{m^2}$ for some $m$.  Furthermore, using the adjoint representation, we see that the complexity of $\G$ (as viewed as a subvariety of $\GL_m(\k)$) is also $O(1)$ (i.e. the complexity remains uniformly bounded in the field size $q$).  In particular, we can take $m=O(1)$.    In most cases we will be able to work with the adjoint representation, but in Section \ref{class} we used instead the tautological representation of $\SL_{r+1}(\k)$ on $\k^{r+1}$, and when dealing with some subcases of the analysis of the triality groups ${}^3 D_4(q)$ in Section \ref{d4-sec} it turns out to be convenient to similarly use a relatively low-dimensional representation (eight-dimensional, in this particular case).  The main feature one needs for the linear representation is that it be faithful, and that the algebraic group $\G$ has complexity $O(1)$, i.e. it is bounded uniformly with respect to the characteristic.

It is known that in all cases we have the inequality
$$|\G(\F_q)^\sigma/\tilde G| \times |Z(\tilde G)| \leq r+1$$
(and for Dynkin diagrams other than $A_r$, the left-hand side is in fact at most $4$) regardless of the choice of representation; see \cite[2.2]{gls3}.  In particular, we see that\footnote{Actually, if one insists on using the adjoint representation, then the centre $Z(\tilde G)$ is always trivial; see \cite[Theorem 2.2.6(c)]{gls3}.}
$$ |Z(\tilde G)| = O(1)$$
and so $\tilde G$ is indeed a bounded cover of $G$ (and is also a bounded index subgroup of $\G(k)^\sigma$).  In particular, $|G|$ and $|\tilde G|$ are comparable.  Because of this, we will in practice be able to lift our analysis up from $G$ to $\tilde G$ without difficulty.

Throughout the paper we encourage the reader to have in mind the following diagram (with all arrows here denoting inclusions except for the far left arrow, which is a quotient).
\[  \begin{CD}
\tilde G @>>> \G(\F_q)^\sigma  @>>> \G(\F_{q}) @>>> \GL_m(\F_q) @>>> \Mat_m(\F_q) \\
@VVV               @.                @VVV            @VVV            @VVV  \\
G @.                @.   \G(\k)     @>>> \GL_m(\k)   @>>> \Mat_m(\k).
\end{CD}\]
Furthermore, the three groups $G, \tilde G, \G(\F_q)^\sigma$ on the left of this diagram should be viewed as ``morally equivalent'', while the group $\tilde G$ should be viewed as a ``sufficiently Zariski-dense'' subgroup of the linear, bounded complexity algebraic mother group $\G(\k)$.  (This type of algebro-geometric viewpoint is of the same type as that used in, for example, \cite{larsen-words}.)

Since $\tilde G \subset \G(\F_{q})\subset \Mat_m(\F_{q})$ and $m=O(1)$, we have the crude upper bound
\begin{equation}\label{lang-weil}
|G| \leq |\tilde G| \ll q^{O(1)}.
\end{equation}
As such, any gain of the form $q^{-\kappa}$ in our arguments can be replaced with $|G|^{-\kappa}$ (after adjusting $\kappa>0$ slightly).

In the arguments of the previous section, a key role was played by the Schwartz-Zippel lemma (Lemma \ref{Schwartz}).  We will need a more general form of this lemma:

\begin{proposition}[Schwartz-Zippel lemma, general case]\label{sufficiently-zdense-prop}  Let $G={}^d \D(q)$ be a finite simple group of Lie type and of rank $r$, associated to the finite field $\F_q$ and with twist order $d$.  Let $\tilde G \subset \G(\k)^\sigma \subset \G(\k) \subset \Mat_m(\k)$ be as in the above discussion, with $\G$ being a linear group of complexity $O(1)$.
\begin{itemize}
\item[(i)] If $P\colon \Mat_m(\k) \to \k$ is a polynomial \textup{(}over $\k$\textup{)} of degree $D \geq 1$ that does not vanish identically on $\G(\k)$, then
\begin{equation}\label{Schwartz-1-gen}
 |\{ a \in \tilde G: P(a) = 0 \}| \ll D q^{-1/d} |\tilde G|.
\end{equation}
\item[(ii)] Similarly, if $P\colon \Mat_m(\k) \times \Mat_m(\k) \to \k$ is a polynomial \textup{(}over $\k$\textup{)} of degree $D \geq 1$, which does not vanish identically on $\G(\k) \times \G(\k)$, then
\begin{equation}\label{Schwartz-2-gen}
|\{ (a,b) \in \tilde G \times \tilde G: P(a,b) = 0 \}| \ll D q^{-1/d} |\tilde G|^2.
\end{equation}
\end{itemize}
\end{proposition}

We prove this proposition in Section \ref{non-conc-sec-untwist}.  Note that \eqref{Schwartz-1-gen} implies in particular that $\tilde G$ is ``sufficiently Zariski-dense'' in $\G$, in the sense that any polynomial over $\k$ that vanishes on $\tilde G$ but not on $\G(\k)$ must have degree $\gg q^{1/d}$.

With these preliminaries, we may now give the proof of Theorem \ref{mainthm}, except in the triality case $G = {}^3 D_4(q)$ which, as we shall see, requires separate treatment at some parts of the argument due to the high twist order $d=3$ in this case.

Henceforth we fix $G$, and let $\tilde G \subset \G(\F_q)^\sigma \subset \G(\F_{q}) \subset \G(\k) \subset \GL_m(\k)$ be as in the above discussion, with $\G$ being a linear algebraic group of complexity $O(1)$.

By Proposition \ref{machine}, we need to establish Proposition \ref{quasi}, Proposition \ref{prod}, and Proposition \ref{nonc} for general finite simple groups $G$ of Lie type.  We begin with the quasirandomness claim.

\begin{proof}[Proof of Proposition \ref{quasi}] In \cite{landazuri-seitz, seitz-zal}, it is shown that all non-trivial irreducible projective representations of $\tilde G/Z(\tilde G)$ have
dimension at least $|\tilde G|^{\beta}$ for some $\beta > 0$ depending only on the rank, which implies the analogous claim for irreducible linear representations of $\tilde G$. As noted in Section \ref{class}, we do not need the full strength of these results; it can be shown (see e.g. \cite[Theorem 4.1]{lubotzky}) that with the exception of the Suzuki group case $G = {}^2 B_2(2^{2k+1})$, the groups $\tilde G$ all contain a copy of either $\SL_2(\F_{\tilde q})$ or $\PSL_2(\F_{\tilde q})$ for some $\tilde q \gg q^c$, and furthermore that this copy and its conjugates generate all of $\tilde G$.  Thus one can reduce to the case of either Suzuki groups or $\PSL_2(\F_q)$, both of which can be treated by \cite[Lemma 4.1]{landazuri-seitz}.
\end{proof}

Now we turn to the product estimate, Proposition \ref{prod}.  In this generality the result is due to Pyber and Szab\'o \cite{pyber-szabo}. However it may also be deduced from the main result of \cite{bgt}, which was stated in the preceding section as Theorem \ref{mainthm-preciseform}.  Indeed, the proof is almost identical to the proof in the $A_r(q)$ case:

\begin{proof}[Proof of Proposition \ref{prod}]
Let $\delta>0$, and let $\delta'>0$ be sufficiently small depending on $\delta$ and $r$.  As in the preceding section, we may assume that $|\tilde G|$ (and thus $q$) is sufficiently large depending on $\delta,r$.

Let $A$ be a $|\tilde G|^{\delta'}$-approximate subgroup of $\tilde G$ with
\begin{equation}\label{gfd-again}
 |\tilde G|^\delta \leq |A| \leq |\tilde G|^{1-\delta}.
 \end{equation}
We apply Theorem \ref{mainthm-preciseform} in $\G(\k)$ with $K \coloneqq  |\tilde G|^{\delta'}$ and $M=O(1)$.  We conclude that one of the options (i), (ii), (iii) is true.  Exactly as in the preceding section, Option (ii) is ruled out from the lower bound of $|A|$ in \eqref{gfd}, and we are done if Option (iii) holds, so we may assume that Option (i) holds.  Applying \eqref{Schwartz-1-gen}, we conclude that
$$ |\langle A \rangle| \ll q^{-1/d} |\tilde G|,$$
and thus (if $q$ is large enough) $A$ does not generate $\tilde G$. The claim follows.
\end{proof}

In view of Proposition \ref{machine}, to prove Theorem \ref{mainthm} it thus suffices to establish Proposition \ref{nonc}.  To do this, we mimic the arguments from the previous section.  As before, we can pass from $G$ to the bounded cover $\tilde G$, and take $\k$ to be an uncountable algebraically closed field containing $\F_{q}$, and set $n \coloneqq  2 \lfloor c_0 \log |\tilde G| \rfloor$ for some small constant $c_0>0$.  We begin by generalising Proposition \ref{sob}.

\begin{lemma}[Maximal subgroups of simple algebraic groups]\label{max}  Suppose that $H < \tilde G$ is a proper subgroup. Then one of the following statements hold.
\begin{enumerate}
\item \textup{(Structural case)} $H$ lies in a proper algebraic subgroup of $\G$ of complexity $O(1)$.
\item \textup{(Subfield case)} Some conjugate of $H$ is contained in $\G(\F_{q'})$ for some proper subfield $\F_{q'}$ of $\F_q$ \textup{(}thus $q' = q^{1/j}$ for some $j>1$\textup{)}.
\end{enumerate}
\end{lemma}

\begin{proof}
Much more detailed results than this can be established using CFSG.
In particular, by the main results of \cite{asch} for classical groups and
Liebeck-Seitz \cite{liebeck-seitz} for exceptional groups, all maximal subgroups are known aside
from almost simple subgroups.   If the almost simple groups are finite groups of Lie
type, then by the representation theory of such groups in the classical case \cite{steinberg-yale}
or by \cite{liebeck-seitz}, they will fall into one of the two cases above.
 If the almost simple groups are not of Lie
type in the same characteristic as $\tilde G$, then there is a bound on their
order (e.g. by \cite{landazuri-seitz}).   The result also follows\footnote{In the case when $q = p^a$ for some bounded $a$, one could also use the results of Nori \cite{nori}.} by \cite[Theorem 0.5]{larsen-pink}, which is independent of CFSG.


\end{proof}


As in the preceding section, we classify the proper subgroups $H$ of $\tilde G$ into structural and subfield subgroups using the above lemma.  Then our task is to establish the analogues of \eqref{structural-case} and \eqref{subfield-case}, namely

\begin{align}\nonumber
 \P_{a,b \in \tilde G}( \P_{w \in W_{n,2}}(w(a,b) \in H) \leq |\tilde G|^{-\gamma} &\hbox{ for all structural } H < \tilde G ) \\  &= 1 - O(|\tilde G|^{-\gamma}),\label{structural-case-2}
\end{align}
and
\begin{align}\nonumber
 \P_{a,b \in \tilde G}( \P_{w \in W_{n,2}}(w(a,b) \in H) \leq |\tilde G|^{-\gamma} &\hbox{ for all subfield } H < \tilde G ) \\ &= 1 - O(|\tilde G|^{-\gamma}).\label{subfield-case-2}
\end{align}

We begin with the proof of \eqref{structural-case-2}.  Arguing exactly as in the preceding section, it will suffice to obtain the analogue of \eqref{wawa-2}, namely
\begin{equation}\label{wawa-3}
 \P_{a,b \in \tilde G}( w(a,b), w'(a,b) \in H \hbox{ for some structural } H < \tilde G ) \ll |\tilde G|^{-3\gamma}.
\end{equation}

As before, the next step is to invoke Proposition \ref{crit} for some sufficiently large $N=O(1)$.   This gives a subvariety $X_N(\k)$ of $\G(\k) \times \G(\k)$ of complexity $O(1)$, such that $(x,y) \in X_N$ whenever $x, y$ both lie in the same structural subgroup of $\tilde G$, and conversely $x,y$ lie in a proper algebraic subgroup of $\G(\k)$ of complexity $O(1)$ whenever $(x,y) \in X_N(\k)$.  Thus it will suffice to show that
$$ |\Sigma_{w,w'} \cap (\tilde G \times \tilde G)| \ll |\tilde G|^{2-3\gamma},$$
where
$$ \Sigma_{w,w'} \coloneqq  \{ (a,b) \in \G \times \G: (w(a,b),w'(a,b)) \in X_N \}.$$
By using Theorem \ref{sec3-key-prop} as in the preceding section, we know that $\Sigma_{w,w'}(\k)$ is a \emph{proper} subvariety of $\G(\k) \times \G(\k)$.  As $w,w'$ have length $O(n)$, we see that the complexity of this variety is also $O(n)$. Applying \eqref{Schwartz-2-gen}, we conclude that
$$ |\Sigma_{w,w'}(\k) \cap (\tilde G \times \tilde G)| \ll n^{O(1)} q^{-1/d} |\tilde G|^2$$
and the claim follows from \eqref{lang-weil} and the logarithmic size $n = O(\log |\tilde G|)$ of $n$.

It remains to establish the subfield case \eqref{subfield-case-2}.
As in the previous section, it suffices to show that
\begin{equation}\label{abg}
 \P_{a,b \in \tilde G} (\gamma_i(w(a,b)) \in \F_{q'} \hbox{ for all } 1 \leq i \leq m ) \ll |\tilde G|^{-3\gamma}
\end{equation}
for all proper subfields $\F_{q'}$ of $\F_q$ and words $w \in W_{n,2}$ with $w(e_1,e_2) \neq 1$.  Using the Schwartz-Zippel type estimate \eqref{Schwartz-2-gen} as in the previous section, we see that
$$ \P_{a,b \in \tilde G} (\gamma_i(w(a,b)) = x_i ) \ll n^{O(1)} q^{-1/d},$$
whenever $1 \leq i \leq m$ and $x_i \in \F_{(q')^d} \backslash \gamma_i(1)$.  Also, since $\G$ cannot consist entirely of unipotent elements, the argument from the preceding section also gives
$$ \P_{a,b \in \tilde G}((w(a,b)-1)^m = 0) \ll n^{O(1)} q^{-1/d}$$
and so we can bound the left-hand side of \eqref{abg} by $O( n^{O(1)} q' q^{-1/d} )$.

We now split into several cases, depending on the value of the twist order $d$.  We first consider the easiest case, namely the untwisted case when $d$ is equal to $1$.  Since $q' \leq q^{1/2}$, the claim then follows from \eqref{lang-weil} by choosing $c_0$ and $\gamma$ small enough.

Now suppose that $G$ is twisted, but is not a triality group ${}^3 D_4(q)$, so that $d$ is equal to $2$.  Then the above arguments give the claim \eqref{subfield-case-2} in this case so long as we restrict attention to subfield subgroups $H$ associated to subfields $\F_{q'}$ of index three or greater in $\F_q$, so that $q' \leq q^{1/3}$.  This leaves the subfield subgroups associated to a subfield $\F_{q^{1/2}}$ of index $2$.  Fortunately, in those cases, the subfield subgroups turn out to also be structural subgroups, and thus can be treated by \eqref{structural-case-2}:

\begin{lemma}\label{subfield-to-struct}
If $G = {}^2 \D(q)$ is a twisted group that is not a triality group, and $\F_{q'}$ is a subfield of $\F_q$ of index $2$, then $\tilde G \cap \G(\F_{q'})$ is contained in a proper subvariety of $\G$ of complexity $O(1)$.  In particular, every subfield subgroup of $\tilde G$ associated to $\F'$ is also a structural subgroup.
\end{lemma}

\begin{proof}  As $q = (\tilde q)^2$ is a perfect square, $G$ cannot be a Suzuki-Ree group, thus the only remaining possibilities are Steinberg groups with $d=2$ and $\D = A_r$, $D_r$, or $E_6$.  In these cases, the field automorphism $\tau\colon \F_q \to \F_q$ is the Frobenius map $x \mapsto x^{\tilde q}$ that fixes $\F_{\tilde q}$.  In particular, $\G(\F_{\tilde q})$ is fixed by $\tau$; since $\tilde G$ is fixed by $\rho \tau$, we conclude that $\tilde G \cap \G(\F_{\tilde q})$ is fixed by the graph automorphism $\rho$, thus we have $\tilde G \cap \G(\F_{\tilde q}) \subset \G(\k)^\rho$, where $\G(\k)^\rho$ is the subvariety of $\G(\k)$ fixed by $\rho$.  Since we are in the simply laced case $\D = A_r, D_r, E_6$, the roots all have the same length, and the action on $\rho$ can be defined on $\G(\k)$ for any field $\k$ as an algebraic map of complexity $O(1)$; see \cite[12.2]{carter}.
As $\rho$ is non-trivial, we conclude that $\G(\k)^\rho$ is a proper subvariety of $\G(\k)$ of complexity $O(1)$, and the claim follows.
\end{proof}

This concludes the proof of \eqref{subfield-case-2} in the case when $G$ is twisted but not a triality group.  The triality group case $G = {}^3 D_4(q)$ does not seem to be fully treatable by the above arguments, and we will present this case separately using an \emph{ad hoc} argument in Section \ref{d4-sec}.

\section{Schwartz-Zippel estimates}\label{non-conc-sec-untwist}

In this section we establish the Schwartz-Zippel bounds in Proposition \ref{sufficiently-zdense-prop}. We already proved instances of this proposition in the $A_r(q)$ case (see Lemma \ref{Schwartz}) and the Suzuki case (see \cite[Lemma 4.2.]{suzuki}). We prove here the general case. Once again a suitable parametrization of the big Bruhat cell $\G$ will be key to the proof.
  
Let $d \in \{1,2,3\}$ and $\k$ an algebraically closed field of characteristic $p$. As above $q$ denotes a power of $p$. 

\begin{definition}[Schwartz-Zippel property]  Let $V$ be an affine variety over $\k$ of complexity $O(1)$, and let $A$ be a finite subset of $V$.  We say that $(A,V)$ has the \emph{Schwartz-Zippel property} (w.r.t $q$ and with constant $c>0$) if one has
$$ |\{ a \in A: P(a) = 0 \}| \leq c D q^{-1/d} |A|$$
whenever $P$ is a polynomial on $V$ of degree $D$ that does not vanish identically on $V$.
\end{definition}

The constant $c>0$ will be allowed to depend on the complexity of $V$, but not on $q$. Our task is thus to show that $(\tilde G, \G)$ and $(\tilde G \times \tilde G, \G \times \G)$ have the Schwartz-Zippel property with respect to all $q$ and for some fixed constant $c>0$ independent of $q$.  To this end, we will rely frequently on the following simple facts that will allow us to reduce the task of verifying the Schwartz-Zippel property for a complicated pair $(A,V)$ to simpler pairs $(A',V')$. In the lemma below we fix $q$, keeping in mind that the constants $c>0$ from the above definition are not allowed to depend on $q$.

\begin{lemma}[Basic facts about the Schwartz-Zippel property]\label{prol} \ \
\begin{itemize}
\item[(i)]  If $(A_1,V_1)$ and $(A_2,V_2)$ have the Schwartz-Zippel property, then so does $(A_1 \times A_2, V_1 \times V_2)$ \textup{(}with slightly worse constants $c$\textup{)}.
\item[(ii)]  Let $Q = Q_1/Q_2\colon V \to W$ be a rational map between affine varieties with dense image, where $Q_1,Q_2$ are polynomials of degree $O(1)$ with $Q_2$ never vanishing on $V$, and let $A$ be a finite subset of $V$.  Suppose that all the preimages $\{ a \in A: Q(a)=b\}$ for $b \in Q(A)$ have the same cardinality.  Then if $(A, V)$ has the Schwartz-Zippel property, $(Q(A),W)$ does also.
\item[(iii)]  Suppose that $V$ is a Zariski-dense subvariety of $W$, $B$ is a finite subset of $W$, and $A$ is a subset of $B \cap V$ with $|B \backslash A| \ll q^{-1/d} |B|$.  Then $(A,V)$ has the Schwartz-Zippel property if and only if $(B,W)$ has the Schwartz-Zippel property.
\end{itemize}
\end{lemma}

\begin{proof}  The claim (i) follows by repeating the derivation of \eqref{Schwartz-2-gen} from \eqref{Schwartz-1-gen} in the proof of Lemma \ref{Schwartz}.  Indeed, if $P$ is a polynomial of degree at most $D$ that does not vanish on $V_1 \times V_2$, then we have $v_1 \in V_1, v_2 \in V_2$ for which $P(v_1,v_2) \neq 0$.  As $(A_1,V_1)$ has the Schwartz-Zippel property, we see that
$$ |\{ a_1 \in A_1: P(a_1,v_2) = 0 \}| \ll D q^{-1/d} |V_1|,$$
while for any $a_1 \in A_1$ with $P(a_1,v_2) \neq 0$, we see from the Schwartz-Zippel property of $(A_2,V_2)$ that
$$ |\{ a_2 \in A_2: P(a_1,a_2) = 0 \}| \ll D q^{-1/d} |V_2|.$$
Summing over $a_1 \in A_1$, we obtain the claim.

To prove (ii), let $P$ be a polynomial on $W$ of degree at most $D$ that does not vanish identically on $W$, and hence on the dense subset $Q(V)$.  Then $P \circ Q$ takes the form $R/Q_2^l$ for some polynomial $R$ on $V$ of degree $O(D)$ that does not vanish identically on $V$, and some natural number $l$.  If $(A,V)$ has the Schwartz-Zippel property, we conclude that
$$ | \{ a \in A: P(Q(a)) = 0 \}| = | \{ a \in A: R(a) = 0 \}| \ll D q^{-1/d} |A|.$$
Since all the preimages of $Q(A)$ in $A$ have the same cardinality, we conclude that
$$ | \{ b \in Q(A): P(b) = 0 \}| \ll D q^{-1/d} |Q(A)|$$
as required.

To prove (iii), suppose first that $(A,V)$ has the Schwartz-Zippel property, and $P$ is a polynomial of degree $D$ on $W$ that does not vanish identically on $W$.  Then $P$ does not vanish identically on $V$ either, as $V$ is Zariski-dense.  We conclude that
$$ | \{ a \in A: P(a) = 0 \}| \ll D q^{-1/d} |A|.$$
Since $|B \backslash A| \ll q^{-1/d} |B|$, we conclude that
$$ | \{ b \in B: P(b) = 0 \}| \ll D q^{-1/d} |B|$$
as required.  The converse implication is established similarly.
\end{proof}

By Lemma \ref{prol}(i) we see that to prove Proposition \ref{sufficiently-zdense-prop}, we only need to show that $(\tilde G, \G)$ has the Schwartz-Zippel property.

The claim is trivial when the field order $q$ is bounded, so we may assume that $q$ is sufficiently large (which will allow us to avoid some degenerate cases when $q$ is small).

As in the case of the special linear groups $\tilde G = \SL_m(\F_q)$ considered in Section \ref{class}, the main strategy here is to exploit the Bruhat decomposition to parameterise (the large Bruhat cell of) $\tilde G$ as rational combinations of a bounded number of coordinates in ``flat'' domains such as $\F_q$ or $\F_q^\times$, for which the (ordinary) Schwartz-Zippel lemma may be easily applied.  As it turns out, the argument from Section \ref{class} may be adapted without difficulty for the untwisted groups $G = \D(q)$, and also works with only a small amount of modification for the Steinberg groups ${}^2 A_r(\tilde q^2)$, ${}^2 D_r(\tilde q^2)$, ${}^2 E_6(\tilde q^2)$, ${}^3 D_4(\tilde q^3)$ and Suzuki-Ree groups ${}^2 B_2(2^{2k+1}), {}^2 F_4(2^{2k+1}), {}^2 G_2(3^{2k+1})$, the main difference in the latter cases being that the coordinates either essentially take values in $\F_{\tilde q}$ rather than $\F_q$, or involve polynomials that are ``twisted'' by the Frobenius map $x \mapsto x^{\tilde q}$ (in the Steinberg cases) or $x \mapsto x^\theta$ (in the Suzuki-Ree cases).  Fortunately, these twisted polynomials can still be handled\footnote{In principle, we could invoke the results of Hrushovski \cite{hrush-twist} here as was done in \cite{gamburdHSSV}; the model-theoretic arguments in that paper do not seem to readily give bounds that are linear (or at least polynomial) in the degree $D$, which is essential for our application.  In any event, the twisted polynomials we are reduced to studying are simple enough that they can be easily controlled by hand.} by the basic Schwartz-Zippel estimate in Lemma \ref{basic-schwartz-zippel}, at the cost of reducing the gain of $O(1/q)$  to $O(1/\tilde q) = O(q^{-1/d})$ or $O(\theta/q) = O(q^{-1/d})$.  This type of strategy was already used for the Suzuki groups ${}^2 B_2(2^{2k+1})$ in \cite{suzuki}, and it turns out that the other Suzuki-Ree groups ${}^2 F_4(2^{2k+1}), {}^2 G_2(3^{2k+1})$ can be handled in a similar fashion.   In all of these cases we will rely heavily on the parameterisations of $\tilde G$ given in the text of Carter \cite{carter}, together with many uses of Lemma \ref{prol} to reduce to working with ``one-parameter'' subgroups of $\tilde G$ or $\G$.

It is possible to treat all three cases (untwisted, Steinberg, and Suzuki-Ree) of Proposition \ref{sufficiently-zdense-prop} in a unified manner, but for pedagogical purposes we shall treat these cases in increasing order of difficulty.\vspace{11pt}

\textsc{The untwisted case.} We begin with the untwisted case $d=1$, where we may basically adapt the arguments for the special linear group from Section \ref{class}.  Here we have $\tilde G = \G(\F_q)$, with $\G$ a Chevalley group.  As such (see e.g. \cite[Chapters 4-8]{carter}), we can find algebraic subgroups $\B, \N, \T, \U$ of $\G$, in which the \emph{maximal torus} $\T$ is abelian, $\N$ contains $\T$ as a finite index normal subgroup, $\U$ is a group of unipotent matrices that are normalised by $\T$, and $\B = \T\U = \U\T$ is the \emph{Borel subgroup}.  The finite group $W \coloneqq  \N/\T$ is known as the \emph{Weyl group} of $\G$, and is of order $O(1)$.  We then have the decomposition
$$ \G = \bigsqcup_{\w \in W} \U \T n_\w \U_\w^-$$
where for each $\w \in W$, $n_\w$ is an (arbitrarily chosen) representative of $\w$ in $\N$, and $\U_\w^-$ is a certain algebraic subgroup of $\U$ (normalised by $\T$) which we will discuss in more detail later, with every $g \in \G$ having a unique decomposition
$$ g = u_1 h n_\w u$$
with $u_1 \in \U$, $h \in \T$, $\w \in W$, $u \in \U_\w^-$; see \cite[Corollary 8.4.4]{carter}.  (Note that in \cite{carter}, the maximal torus is denoted $H$ instead of $\T$.)  This decomposition descends to the field $\F_q$, so that
$$ \G(\F_q) = \bigsqcup_{\w \in W} \U(\F_q) \T(\F_q) n_\w \U_\w^-(\F_q)$$

The unipotent group $\U$ can be decomposed further.  There is a finite totally ordered set $\Phi^+$ of cardinality $O(1)$ (the set of \emph{positive roots}) associated to $\G$, and for each element $\alpha$ in $\Phi^+$, there is an injective algebraic homomorphism $x_\alpha\colon \k \to \G(\k)$ (where $\k$ is viewed as an additive group) of complexity $O(1)$ which is also defined over $\F_q$.  The exact construction of $\Phi^+$ and the $x_\alpha$ will not be important to us, but see \cite[5.1]{carter} for details.  The image of $x_\alpha$ is thus a one-dimensional algebraic subgroup of $\G$ which we will denote $\X_\alpha$.  The group $\U$ can then be parameterised as
\begin{equation}\label{factor}
 \U = \prod_{\alpha \in \Phi^+} \X_\alpha
\end{equation}
where the product is taken in increasing order, and furthermore each element $u$ of $\U$ has a unique representation in the form
$$ u = \prod_{\alpha \in \Phi^+} \X_\alpha(u_\alpha)$$
with $u_\alpha \in \k$; see \cite[Theorem 5.3.3]{carter}.  This factorisation descends to $\F_q$, thus if $u \in \U(\F_q)$ then the coordinates $u_\alpha$ lie in $\F_q$, and conversely.  In particular, we see that $|\U(\F_q)| = q^{|\Phi^+|}$.

For each word $\w \in W$, the group $\U_\w^-$ mentioned earlier can be factorised as
$$ \U_\w^- \coloneqq  \prod_{\alpha \in \Psi_\w} \X_\alpha$$
for some subset $\Psi_\w$ of $\Phi^+$; see \cite[8.4]{carter}.  In particular, $|\U_\w^-(\F_q)| = q^{|\Psi_\w|}$.  There is a unique element $\w_0$ of $W$, which we call the \emph{long word} with the property that $\Psi_\w = \Phi^+$; see \cite[Proposition 2.2.6]{carter}.  In particular, we have $|\U_\w^-(\F_q)| = O(q^{-1} |\U(\F_q)|)$ for all $\w \neq \w_0$, and so the large Bruhat cell $\B(\F_q) n_{\w_0} \B(\F_q) = \U(\F_q) \T(\F_q) n_{\w_0} \U(\F_q)$ occupies most of $\tilde G$:
\begin{equation}\label{gbwb-2}
 |\tilde G| = (1+O(1/q)) |\B(\F_q) n_{\w_0} \B(\F_q)|.
\end{equation}
This of course generalises \eqref{gbwb}.  A similar argument shows that $\B n_{\w_0} \B$ is Zariski-dense in $\G$.  Thus, by Lemma \ref{prol}(iii), to show that $(\tilde G, \G)$ has the Schwartz-Zippel property, it suffices to establish the Schwartz-Zippel property for the pair
$$(\B(\F_q) n_{\w_0} \B(\F_q), \B n_{\w_0} \B) =  (\U(\F_q) \T(\F_q) n_{\w_0} \U(\F_q),  \U \T n_{\w_0} \U ).$$
Composing with the map $(u_1,h,u) \mapsto u_1 h n_{\w_0} u$, which is a polynomial map of degree $O(1)$, it thus suffices by Lemma \ref{prol}(ii) to show that the pair
$$
(\U(\F_q) \times \T(\F_q) \times \U(\F_q), \U \times \T \times \U )
$$
has the Schwartz-Zippel property.  By Lemma \ref{prol}(i), it thus suffices to show that $(\U(\F_q),\U)$ and $(\T(\F_q),\T)$ have the Schwartz-Zippel property.

By using the factorisation \eqref{factor} to parameterise $\U$ and $\U(\F_q)$ (using polynomial maps of degree $O(1)$) and Lemma \ref{prol}(ii), we see that the Schwartz-Zippel property for $(\U(\F_q),\U)$ will follow from the Schwartz-Zippel property for $(\F_q^{\Phi^+}, \k^{\Phi^+})$; but this latter property follows from Lemma \ref{basic-schwartz-zippel}.  Now we turn to the Schwartz-Zippel property for $(\T(\F_q),\T)$.  The abelian algebraic group $\T$ is generated by a family $(\HH_\alpha)_{\alpha \in \Pi}$ of commuting one-parameter subgroups
$$ \HH_\alpha(\k) \coloneqq  \{ h_\alpha(\lambda): \lambda \in \k^\times \}$$
where $h_\alpha\colon \k^\times \to \G(\k)$ is a homomorphism (viewing $\k^\times$ as a multiplicative group), with $h_\alpha(t)$ being a polynomial of degree $O(1)$ divided by a monomial in $t$, also of degree $O(1)$, and $\alpha$ ranges over a set $\Pi$ of cardinality $O(1)$ (the set of fundamental roots).    The exact construction of $h_\alpha$ and $\Pi$ will not be important to us, but see \cite[7.1]{carter} for details.  The factorisation can be localised to $\F_q$, thus
\begin{equation}\label{tfqh}
 \T(\F_q) = \prod_{\alpha \in \Pi} \HH_\alpha(\F_q)
\end{equation}
and
$$ \HH_\alpha(\F_q) \coloneqq  \{ h_\alpha(\lambda): \lambda \in \F_q^\times \}.$$
The product decomposition in \eqref{tfqh} is not unique, but as all groups here are abelian, every element in $\T(\F_q)$ has the same number of representations as a product of elements in $\HH_\alpha(\F_q)$.  Thus by Lemma \ref{prol}(ii) and Lemma \ref{prol}(i), to establish the Schwartz-Zippel property for $(\T(\F_q),\T)$, it suffices to do so for each $(\HH_\alpha(\F_q), \HH_\alpha)$ with $\alpha \in \Pi$.
By another application\footnote{Strictly speaking, Lemma \ref{prol}(ii) does not quite apply here because $h_\alpha(\lambda)$ is rational instead of polynomial, with a denominator that is a monomial in $\lambda$ of degree $O(1)$, but the proof of Lemma \ref{prol}(ii) still applies after clearing denominators, since $\lambda$ is non-vanishing on $\F_q^\times$.} of Lemma \ref{prol}(ii), it suffices to show that $(\F_q^\times, \k^\times)$ has the Schwartz-Zippel property; by Lemma \ref{prol}(iii), it suffices to establish this property for $(\F_q,\k)$.  But this again follows from Lemma \ref{basic-schwartz-zippel}.\vspace{11pt}

\textsc{The Steinberg case.} Now we adapt the previous argument to establish \eqref{Schwartz-1-gen} the case of Steinberg groups.  The arguments here will work to some extent for the Suzuki-Ree groups as well; we will indicate the point where the two cases diverge.

As in the untwisted case, the algebraic group $\G$ contains a Borel subgroup $\B$, a maximal torus $\T$, a unipotent group $\U$, and a group $\N$ containing $\T$ as a finite index subgroup.  In the twisted case we also have the automorphism $\sigma\colon \G(\F_q) \to \G(\F_q)$.  We then form the groups
\begin{align*}
U^1 &\coloneqq  \{ u \in \U(\F_q): \sigma(u) = u\} \\
T^1 &\coloneqq  \T(\F_q) \cap \tilde G \\
N^1 &\coloneqq  \N(\F_q) \cap \tilde G \\
W^1 &\coloneqq  N^1 / T^1.
\end{align*}
Then $W^1$ can be shown to be a finite group of size $O(1)$ (see \cite[13.3]{carter}).  As in the untwisted case, we have a decomposition
$$ \tilde G = \bigsqcup_{\w \in W^1} U^1 T^1 n_\w (U_\w^-)^1$$
where $(U_\w^-)^1$ is a certain subgroup of $U^1$, where $n_\w$ is an (arbitrarily chosen) representative of $\w$ in $N$, and with every $g \in G$ having a unique representation of the form
$$ g = u_1 h n_\w u$$
with $u_1 \in U^1$, $h \in T^1$, $\w \in W^1$, and $u \in (U_\w^-)^1$; see\footnote{Recall from Remark \ref{diff} that the groups defined in \cite{carter} agree with the ones defined here for all but finitely many exceptions, which we may ignore as our results are asymptotic in nature.} \cite[Proposition 13.5.3]{carter}.

As before, the unipotent group $U^1$ can be decomposed further. The set $\Phi_+$ can be partitioned in a certain way into a collection $\Sigma$ of disjoint subsets $S$ of $\Phi_+$ in such a way that
\begin{equation}\label{uxs}
U^1 = \prod_{S \in \Sigma} X^1_S
\end{equation}
for certain commuting abelian subgroups $X^1_S$ of $\prod_{\alpha \in S} \X_\alpha(\F_q)$, with each element of $U^1$ having a unique representation as such a product; furthermore, we have
$$ (U_\w^-)^1 = \prod_{S \in \Sigma_\w} X^1_S$$
for some subset of $\Sigma_\w$, with the long word $\w_0$ (which can be viewed as an element of $W_1$ as well as $W$) being the unique element of $W_1$ for which $\Sigma_{\w_0} = \Sigma$; see \cite[Proposition 13.6.1]{carter}.  The groups $X^1_S$ are in fact the fixed points of $\sigma$ in $\prod_{\alpha \in S} \X_\alpha(\F_q)$ (see \cite[Lemma 13.5.1]{carter}), and are described explicitly in \cite[Proposition 13.6.3]{carter}.  One consequence of this description is that
$$ |X^1_S| \gg q^{1/d}$$
for each $S \in \Sigma$.  As a consequence we see that
$$ |(U_\w^-)^1| \ll q^{-1/d} |U^1|$$
for all $\w \neq \w_0$, so as before the large Bruhat cell $U^1 T^1 n_{\w_0} U^1$ occupies most of $\tilde G$:
\begin{equation}\label{tgg}
 |\tilde G| = (1 + O(q^{-1/d})) |U^1 T^1 n_{\w_0} U^1|.
\end{equation}
To show the Schwartz-Zippel property for $(\tilde G, \G)$, it thus suffices by Lemma \ref{prol}(iii) (and the Zariski-density of $\U \T n_{\w_0} \U$ in $\G$, as noted in the previous section) to establish this property for the pair
$$ ( U^1 T^1 n_{\w_0} U^1, \U \T n_{\w_0} \U ).$$
Using Lemma \ref{prol}(i), (ii) as in the untwisted case, it thus suffices to establish the Schwartz-Zippel property for the pairs $(U^1, \U)$ and $(T^1,\T)$.

We begin with the Schwartz-Zippel property for $(U^1,\U)$. Splitting $U^1$ using \eqref{uxs}, and using the corresponding decomposition
$$\U = \prod_{S \in \Sigma} \prod_{\alpha \in S} \X_\alpha$$
of $\U$, we see from Lemma \ref{prol}(i), (ii) that it suffices to establish the Schwartz-Zippel property for the pairs
$$ (X_S^1, \prod_{\alpha \in S} \X_\alpha(\k))$$
for each $S \in \Sigma$.

Thus far our arguments have made no distinction between the Steinberg and Suzuki-Ree cases.  Now we specialise to the Steinberg case (so $q = \tilde q^d$ for some $d=2,3$ and $\tau(x) \coloneqq  x^{\tilde q}$) in order to describe the sets $X_S^1$ more explicitly.  By \cite[Proposition 13.6.3]{carter} (or \cite[Theorem 2.4.1]{gls3}), one is in one of the following four cases:
\begin{itemize}
\item[(i)] $S = \{\alpha\}$, and $X^1_S = \{ x_\alpha(t): t \in \F_{\tilde q} \}$.
\item[(ii)] $S = \{ \alpha, \overline{\alpha}\}$, $d=2$ and $X^1_S = \{ x_\alpha(t) x_{\overline{\alpha}}(t^{\tilde q}): t \in \F_q \}$.
\item[(iii)] $S = \{ \alpha, \overline{\alpha}, \overline{\overline{\alpha}}\}$, $d=3$, and $X^1_S = \{ x_\alpha(t) x_{\overline{\alpha}}(t^{\tilde q}) x_{\overline{\overline{\alpha}}}(t^{\tilde q^2})\}$.
\item[(iv)] $S = \{ \alpha, \rho(\alpha), \alpha+\rho(\alpha)\}$, $d=2$, and
$$ X^1_S = \{ x_\alpha(t) x_{\rho(\alpha)}(t^{\tilde q}) x_{\alpha+\rho(\alpha)}( u ): t,u \in \F_q; u + u^{\tilde q} = - N_{\alpha,\rho(\alpha)} t t^{\tilde q} \},$$
where $N_{\alpha,\rho(\alpha)}$ is an integer.
\end{itemize}

In case (i), we see from Lemma \ref{prol}(ii) that it suffices to establish the Schwartz-Zippel property for $(\F_{\tilde q}, \k)$; but this follows from Lemma \ref{basic-schwartz-zippel}.

In cases (ii) and (iii), we again apply Lemma \ref{prol}(ii) and reduce to the task of establishing the Schwartz-Zippel property for
\begin{equation}\label{sz-2}
 ( \{ (t, t^{\tilde q}): t \in \F_q \}, \k^2 )
\end{equation}
in the $d=2$ case and
\begin{equation}\label{sz-3}
( \{ (t, t^{\tilde q}, t^{\tilde q^2}): t \in \F_q \}, \k^3 )
\end{equation}
in the $d=3$ case.  We will just establish \eqref{sz-3}, as \eqref{sz-2} is similar.  Let $P$ is a non-vanishing polynomial of degree $D$ on $\k^3$.  Our task is to show that
\begin{equation}\label{quiet}
|\{ t \in \F_q: P( t, t^{\tilde q}, t^{\tilde q^2} ) = 0 \}| \ll D \tilde q^{-1} q.
\end{equation}
If $D \geq \tilde q$ then the bound is trivial, so suppose that $D < \tilde q$.  Then as $P$ is non-vanishing and of degree $D$, we see that $P( t, t^{\tilde q}, t^{\tilde q^2} )$, viewed as a polynomial in $t$, is also non-vanishing and of degree at most $D \tilde q^2 = D \tilde q^{-1} q$.  The claim then follows from Lemma \ref{basic-schwartz-zippel}.

Finally, in case (iv), we have $d=2$, so that $q = \tilde q^2$. We abbreviate $t^{\tilde q}$ as $\overline{t}$, and $N_{\alpha,\rho(\alpha)}$ as $N$.  We again apply Lemma \ref{prol}(ii) and reduce to establishing the Schwartz-Zippel property for
$$
( \{ (t, \overline{t}, u): t, u \in \F_q; u + \overline{u} = - N t \overline{t} \}, \k^3 ).$$
If the characteristic $p$ is not $2$, we can reparameterise the triple $(t,\overline{t},u)$ as
$$ (t,\overline{t},u) = (a+ib, a-ib, - N (a^2-ib^2)/2 + i c)$$
for $a,b,c \in \F_{\tilde q}$, where $i$ is any non-zero element of $\F_q$ with $\overline{i} = -i$.  By Lemma \ref{prol}(ii), we reduce to establishing the Schwartz-Zippel property for $( \F_{\tilde q}^3, \k^3 )$; but this follows from Lemma \ref{basic-schwartz-zippel}.  If instead the characteristic $p$ is $2$, we have the alternate parameterisation
$$ (t,\overline{t},u) = (a+\omega b, a+\omega b+b, N (a^2+ab+\omega\overline{\omega} b^2)+ c)$$
for $a,b,c \in \F_{\tilde q}$, where $\omega$ is an element of $\F_q$ with $\omega+\overline{\omega}=1$, and the claim follows from Lemma \ref{prol}(ii) and Lemma \ref{basic-schwartz-zippel} as before.  This concludes the demonstration of the Schwartz-Zippel property for $(U^1,\U)$ in the Steinberg case.

Finally, we establish the Schwartz-Zippel property for $(T^1,\T)$.  We recall that the abelian group $\T$ is generated by the commuting subgroups
$$ \HH_\alpha(\k) \coloneqq  \{ h_\alpha(\lambda): \lambda \in \k^\times \}$$
for $\alpha \in \Pi$, for some rational maps $h_\alpha\colon \k^\times \to \G(\k)$ with $h_\alpha(t)$ a polynomial of degree $O(1)$ divided by a monomial in $t$ of degree $O(1)$, where $\Pi$ is a finite set of cardinality $O(1)$.

It turns out that the Dynkin graph automorphism $\rho$ acts on $\Pi$ as a permutation of order $d$; see \cite[13.7]{carter}.  Let $\Gamma$ be the orbits of $\rho$ on $\Pi$, thus each element $J$ of $\Gamma$ is either a singleton $\{\alpha\}$, a pair $\{ \alpha,\overline{\alpha})\}$ (if $d=2$), or a triplet $\{ \alpha,\overline{\alpha}, \overline{\overline{\alpha}}\}$ if $d=3$.  One can then show that $T^1$ is generated by the groups $H_J^1$ for $J \in \Gamma$, where
\begin{align*}
H_{\{\alpha\}}^1 &\coloneqq  \{ h_\alpha(t): t \in\F_{\tilde q}^\times \} \\
H_{\{\alpha,\overline{\alpha}\}}^1 &\coloneqq  \{ h_\alpha(t) h_{\overline{\alpha}}(t^{\tilde q}): t \in \F_q^\times \} \\
H_{\{\alpha,\overline{\alpha},\overline{\alpha}\}}^1 &\coloneqq  \{ h_\alpha(t) h_{\overline{\alpha}}(t^{\tilde q}) h_{\overline{\overline{\alpha}}}(t^{\tilde q^2}): t \in \F_q^\times \};
\end{align*}
see (the proof of) \cite[Theorem 13.7.2]{carter} or \cite[Theorem 2.4.7]{gls3}.  Note that each $H_J^1$ is a subgroup of the group $\HH_J$ generated by the $\HH_\alpha$ for $\alpha \in J$.  By Lemma \ref{prol}(i) and Lemma \ref{prol}(ii), we thus see that to establish the Schwartz-Zippel property for $(T^1,\T)$, it suffices to do so for the pairs $(H^J_1, \HH_J)$ for each $J \in \Gamma$.  By Lemma \ref{prol}(ii) (and clearing denominators, as in the previous section), it suffices to establish the Schwartz-Zippel property for the pair
$$ (\F_{\tilde q}^\times, \k^\times),$$
as well as the pair
$$ ( \{ (t,t^{\tilde q}): t \in \F_q^\times \}, (\k^\times)^2 )$$
for $d=2$ and the pair
$$ ( \{ (t,t^{\tilde q}, t^{\tilde q^2}): t \in \F_q^\times \}, (\k^\times)^3 )$$
for $d=3$.  But these follow from Lemma \ref{basic-schwartz-zippel}, \eqref{sz-2}, \eqref{sz-3}, and Lemma \ref{prol}(iii).\vspace{11pt}

\textsc{The Suzuki-Ree case.} Now we establish \eqref{Schwartz-1-gen} for Suzuki-Ree groups, thus $G = {}^2 \D(p \theta^2)$ for some $p=2,3$ and some $\theta = p^k$ (so in particular $\theta$ is comparable to $q^{-1/d}$), with the field automorphism $\tau$ given by the Frobenius map $x \mapsto x^\theta$.
By the arguments given for Steinberg groups, it suffices to establish the Schwartz-Zippel property for $(U^1,\U)$ and $(T^1,\T)$.

We begin again with $(U^1,\U)$.  By the arguments given for Steinberg groups, it suffices to establish the Schwartz-Zippel property for
the pairs
$$ (X_S^1, \prod_{\alpha \in S} \X_\alpha(\k))$$
for each $S \in \Sigma$.  By \cite[Proposition 13.6.3]{carter} or \cite[Theorem 2.4.5]{gls3}, one is in one of the following three cases:
\begin{itemize}
\item[(i)] $S = \{\alpha, \overline{\alpha} \}$ has cardinality $2$, and $X^1_S = \{ x_\alpha(t^\theta) x_{\overline{\alpha}}(t): t \in \F_q \}$.
\item[(ii)] $S = \{a, b, a+b, 2a+b \}$ has cardinality $4$, and
$$X^1_S = \{ x_a(t^\theta) x_b(t) x_{a+b}(t^{\theta+1}+u) x_{2a+b}(u^{2\theta}): t,u \in \F_q \}.$$
\item[(iii)] $S = \{a,b,a+b,2a+b,3a+b,3a+2b\}$ has cardinality $6$, and
\begin{align*} X^1_S = \{ x_a(t^\theta) & x_b(t) x_{a+b}(t^{\theta+1}+u^\theta) \times \\ & \times x_{2a+b}(t^{2\theta+1}+v^\theta) x_{3a+b}(u) x_{3a+2b}(v): t,u,v \in \F_q \}.\end{align*}
\end{itemize}

By Lemma \ref{prol}(ii), it suffices to establish the Schwartz-Zippel property for the pairs
\begin{equation}\label{poc1}
( \{ (t^\theta, t): t \in \F_q \}, \k ),
\end{equation}
\begin{equation}\label{poc2}
( \{ (t^\theta, t, t^{\theta+1}+u, u^{2\theta}): t,u \in \F_q \}, \k^4 ),
\end{equation}
and
\begin{equation}\label{poc3}
( \{ (t^\theta, t, t^{\theta+1}+u^\theta, t^{2\theta+1}+v^\theta, u, v): t,u \in \F_q \}, \k^6 ).
\end{equation}
The Schwartz-Zippel property for \eqref{poc1} is proven analogously to \eqref{sz-2} (or \eqref{sz-3}) and is omitted.  To prove \eqref{poc2}, we first parameterise $\k^4$ as $(x,y,xy+z, w)$ and reduce (by Lemma \ref{prol}(ii)) to showing the Schwartz-Zippel property for
$$( \{ (t^\theta, t, u, u^{2\theta}): t,u \in \F_q \}, \k^4 ).$$
Let $P$ be a polynomial of degree at most $D$ that does not vanish on $\k^4$; our task is to show that
$$ |\{ (t,u) \in \F_q^2: P( t^\theta,t,u,u^{2\theta}) = 0 \}| \ll D q^{-1/2} q^2.$$
This claim is trivial if $D \geq \theta/10$, so suppose that $D < \theta/10$.  Then $P( t^\theta,t,u,u^{2\theta})$ is a polynomial function of $t,u$ of degree $O(D\theta) = O(Dq^{1/2})$ that does not vanish identically, and the claim follows from Lemma \ref{basic-schwartz-zippel}.

In a similar fashion, after parameterising $\k^6$ as $(x,y,xy+z,x^2y+w,u,v)$ and applying Lemma \ref{prol}(ii), we see that to establish the Schwartz-Zippel property for \eqref{poc3} it suffices to do so for
$$ ( \{ (t^\theta, t, u^\theta, v^\theta, u, v): t,u \in \F_q \}, \k^6 ),$$
which can be done by the same sort of argument used to establish \eqref{poc2}.

Finally, we need to establish the Schwartz-Zippel property for $(T^1,\T)$.  Recall that $\T$ is generated by $\HH_\alpha$ for $\alpha \in \Pi$.  In the Suzuki-Ree cases, it turns out that $\Pi$ splits into pairs $\{ \alpha, \overline{\alpha}\}$ (consisting of one long root and one short root), and $T^1$ is generated by the finite abelian groups
$$ \{ h_\alpha(t) h_{\overline{\alpha}}(t^{\lambda(\alpha) \theta}): t \in \F_q^\times \}$$
as $\{\alpha,\overline{\alpha}\}$ range over these pairs, where $\lambda(\alpha)$ is either $1$ or $p$ (depending on whether the root $\alpha$ is short or long); see \cite[Theorem 13.7.4]{carter} or \cite[Theorem 2.4.7]{gls3}.  By Lemma \ref{prol}(ii) (and clearing denominators), it thus suffices to establish the Schwartz-Zippel property for the pairs
$$ ( \{ (t, t^\theta): t \in \F_q^\times \}, (\k^\times)^2 \}$$
and
$$ ( \{ (t, t^{p\theta}): t \in \F_q^\times \}, (\k^\times)^2 \}.$$
But this can be proven by the same methods used to prove \eqref{poc1} (or \eqref{sz-2}, \eqref{sz-3}).

\section{The triality case}\label{d4-sec}

In this section we complete the proof of Theorem \ref{mainthm} in the case when $G$ is a Tits-Steinberg triality group $G = {}^3D_4(q)$, where $q = \tilde q^3$.

We will use the mother group $\G = \SO_8$, which of course has Dynkin diagram $D_4$.  Rather than use the adjoint representation, we will use the tautological representation
$$ \G \subset \GL_8 \subset \Mat_8$$
(see e.g. \cite[\S 11.3]{carter} for the details of this tautological representation).
In particular, $\tilde G$ is now viewed as a subgroup of $\GL_8(\F_q)$.  For sake of concreteness, one could take the quadratic form defining $\SO_8$ to be $x_1 x_5 + x_2 x_6 + x_3 x_7 + x_4 x_8$, so that the Lie algebra $\mathfrak{so}_8$ consists of those $8 \times 8$ matrices of the form
$$ \begin{pmatrix} A & S_1 \\ S_2 & -A^T \end{pmatrix}$$
where $A, S_1, S_2$ are $4 \times 4$ matrices with $S_1,S_2$ skew-symmetric.

An inspection of the arguments of Section \ref{sec1a} (using the $d=3$ case of the Schwartz-Zippel bounds in Proposition \ref{sufficiently-zdense-prop}) reveals that one only needs to establish the subfield non-concentration bound \eqref{subfield-case-2}, and furthermore that this bound is already established by those arguments in the event that $H$ comes from a subfield $\F_{q'}$ of index greater than three.  Thus it only remains to control the subfields of index two and three.

By repeating the proof of Lemma \ref{subfield-to-struct}, we have

\begin{lemma}\label{subfield-to-struct-triality} Let $G = {}^3 D_4(q)$ be a triality group.
If $\F_{q'}$ is a subfield of $\F_q$ of index $3$ \textup{(}i.e. $q'^3=q$\textup{)}, then $\tilde G \cap \G(\F_{q'})$ is contained in a proper subvariety of $\G$ of complexity $O(1)$.
\end{lemma}

Thus we can dispose of the contribution of subfields $\F_{q'}$ of index three by using the structural bound \eqref{structural-case-2} that has already been established.  This leaves only the subfields $\F_{q'}$ of index two, thus we have
$$
q = \tilde q^3 = (q')^2$$
and so we can find a power $q_0$ of the characteristic $p$ such that
\begin{equation}\label{q-form}
q = q_0^6; \quad \tilde q = q_0^2; \quad q' = q_0^3.
\end{equation}

For any $x \in \Mat_8(\k)$, let $\gamma_1(x),\ldots,\gamma_8(x) \in \k$ be the coefficients of the characteristic polynomial of $x$.
Let $X \subset \GL_8(\k)$ denote the set of matrices $x \in \GL_8(\k)$ such that $\gamma_1(x),\dots,\gamma_8(x) \in \F_{q'}$, where $\gamma_1(x) = \tr(x),\gamma_2(x),\dots,\gamma_7(x),\gamma_8(x) = \det(x)$ are the coefficients of the characteristic polynomial of $x$.  Then $X$ contains every subfield group of $\tilde G$ associated to $\F_{q'}$, and so it will suffice to show that
$$ \P_{a,b \in \tilde G; w \in W_{n,2}} (w(a,b) \in X ) \ll |\tilde G|^{-3\gamma}$$
where $n = 2 \lfloor c_0 \log |\tilde G|\rfloor$ for some sufficiently small $c_0>0$.

The first step is to pass from $\tilde G$ to the large Bruhat cell $U^1 T^1 n_{\w_0} U^1$ that was introduced in Section \ref{non-conc-sec-untwist}.
 In view of \eqref{tgg}, it will suffice to show that
$$ \P_{a,b \in U^1 T^1 n_{\w_0} U^1; w \in W_{n,2}} (w(a,b) \in X ) \ll |\tilde G|^{-3\gamma}$$

We will view $U^1 T^1 n_{\w_0} U^1$ as a ``sufficiently Zariski-dense'' finite subset of $\U \T n_{\w_0} \U$.  Now we divide the words $w \in W_{n,2}$ into two categories.  Let us say that $w$ is \emph{degenerate} if one has $w(a,b) \in X$ for all $a,b \in U^1 T^1 n_{\w_0} U^1$, and \emph{non-degenerate} otherwise.

We first dispose of the degenerate case.  We will need two key lemmas.  The first shows that $\tilde G$ contains a simpler subgroup $H$ which can be used as a proxy for $\tilde G$ for the purposes of excluding degeneracy:

\begin{lemma}[Good embedded subgroup]\label{subgp}  There exists a subgroup $H$ of $\tilde G$ with the following properties:
\begin{itemize}
\item[(i)] $H$ is isomorphic to the central product $\SL_2(\F_{q}) \circ \SL_2(\F_{\tilde q})$ of $\SL_2(\F_{q})$ and $\SL_2(\F_{\tilde q})$, i.e. the quotient of the direct product $\SL_2(\F_{q}) \times \SL_2(\F_{\tilde q})$ by the diagonally embedded common centre of $\SL_2(\F_q)$ and $\SL_2(\F_{\tilde q})$ \textup{(}which is trivial in even characteristic and is the two-element group $\{ (+1,+1), (-1,-1)\}$ on odd characteristic\textup{)}.
\item[(ii)]  $H$ has  large intersection with the Bruhat cell $U^1 T^1 n_{\w_0} U^1$ in the sense that
\begin{equation}\label{sl2-prop-1}
|H\setminus U^1 T^1 n_{\w_0} U^1| \leq |H|/2.
\end{equation}
\item[(iii)] $H$ mostly avoids $X$ in the sense that
\begin{equation}\label{sl2-prop-2}
|H \cap X| \ll |H|^{1 - c}
\end{equation}
for some absolute constant $c>0$.
\end{itemize}
\end{lemma}

\begin{proof}  We use an explicit description of the root system for $D_4$.  Namely, we can take the set $\Phi$ of roots to be the set
$$ \Phi \coloneqq  \{ \pm e_i \pm e_j: 1 \leq i < j \leq 4 \}$$
in $\R^4$, with $e_1,\ldots,e_4$ being the standard basis for $\R^4$, with fundamental roots
$$ \Pi \coloneqq  \{ e_1-e_2, e_2-e_3, e_3-e_4, e_3+e_4\};$$
see e.g. \cite[3.6]{carter}.  We can take the triality map $\rho$ to be a map that cyclically permutes the three fundamental roots $e_1-e_2, e_3-e_4, e_3+e_4$ in that order (e.g. $\rho(e_1-e_2) = e_3-e_4$) while leaving $e_2-e_3$ unchanged.  Among the orbits of this map on $\Phi$ include the singleton orbit $\{ \alpha\}$ with $\alpha \coloneqq  e_2-e_3$ and the tripleton orbit $\{ \beta, \overline{\beta}, \overline{\overline{\beta}}\}$ with $\beta \coloneqq  e_2 + e_3$, $\overline{\beta} \coloneqq  \rho(\beta) = e_1 + e_4$, $\overline{\overline{\beta}} \coloneqq  \rho^2(\beta) = e_1-e_4$.

To each root $a \in \Phi$ we have a one-parameter root group $\X_a(\k) = \{ x_a(t): t \in \k\}$, with two root groups $X_a, X_b$ commuting if $a+b$ is neither zero nor a root (see \cite[Theorem 4.2.1]{carter}).  Using the particular representation of $\SO_8$ described above, we have the explicit formulae
\begin{align*}
 x_{e_i-e_j}(t) &= 1 + t
\begin{pmatrix} e_{ij} & 0 \\ 0 & - e_{ji} \end{pmatrix} \\
 x_{e_j-e_i}(t) &= 1 + t
\begin{pmatrix} e_{ji} & 0 \\ 0 & - e_{ij} \end{pmatrix} \\
 x_{e_i+e_j}(t) &= 1 + t
\begin{pmatrix} 0 & e_{ij} - e_{ji} \\ 0 & 0 \end{pmatrix} \\
x_{-e_i-e_j}(t) &= 1 + t
\begin{pmatrix} 0 & 0 \\ e_{ji} - e_{ij} & 0  \end{pmatrix}
\end{align*}
for $1 \leq i < j \leq 4$, where $e_{ij}$ is the elementary $4 \times 4$ matrix with an entry of $1$ at the $(i,j)$ position and zero elsewhere; see \cite[11.2]{carter}.

If we let $\SS_a$ be the group generated by $\X_a$ and $\X_{-a}$, then $\SS_a$ is isomorphic to $\SL_2$ (see \cite[Chapter 6]{carter}).  Observe that $\pm a \pm b$ is neither zero nor a root when $a,b$ are distinct elements of $\{ \alpha, \beta, \overline{\beta}, \overline{\overline{\beta}}\}$.  Thus $\SS_\alpha$, $\SS_\beta$, $\SS_{\overline{\beta}}$, and $\SS_{\overline{\overline{\beta}}}$ all commute with each other.

The group $\tilde G$ contains the subgroups
$$ X_{\pm \alpha}^1 \coloneqq  \{ x_{\pm \alpha}(t): t \in \F_{\tilde q} \}$$
and
$$ X_{\pm \{\beta,\overline{\beta},\overline{\overline{\beta}}\}}^1 \coloneqq  \{ x_{\pm \beta}(t) x_{\pm \overline{\beta}}(t^{\tilde q})
 x_{\pm \overline{\overline{\beta}}}(t^{\tilde q^2}): t \in \F_q \} $$
for either fixed choice of sign $\pm$; see \cite[Proposition 13.6.3]{carter}.  Let $S_\alpha^1$ be the subgroup of $\tilde G$ generated by $X_{\pm \alpha}^1$, and similarly let $S_{\{\beta,\overline{\beta},\overline{\overline{\beta}}\}}^1$ be the subgroup of $\tilde G$ generated by $X_{\pm \{\beta,\overline{\beta},\overline{\overline{\beta}}\}}^1$.  Then from the preceding discussion, $S_\alpha^1$ is isomorphic to $\SL_2(\F_{\tilde q})$ and $S_{\{\beta,\overline{\beta},\overline{\overline{\beta}}\}}^1$ is isomorphic to $\SL_2(\F_q)$, and furthermore these two groups commute with each other, and so can only intersect in their common centre, whose order $C$ is equal to $2$ in odd characteristic and $1$ in even characteristic.
Set $H$ to be the group generated by both $S_\alpha^1$ and $S_{\{\beta,\overline{\beta},\overline{\overline{\beta}}\}}^1$, then $H$ is a central product of $\SL_2(\F_q)$ and $\SL_2(\F_{\tilde q})$.

If $q$ is even, then $H$ is a direct product of the two subgroups.  If $q$ is odd, then
$H$ is contained in the centraliser of an involution $z$, namely the non-trivial central element of $H$ (in the above concrete representation, one has $z = \operatorname{diag}(1,-1,-1,1,1,-1,-1,1)$).  Moreover, it follows
from \cite[4.5.1]{gls3} that $H$ has index $2$ in the centralizer.  In particular,
$[T^1:H \cap T^1]=C$ where $C$ is defined as above.

Now we verify (ii).  From (i) we have
$$ |H| = (\frac{1}{C} + O(1/\tilde q)) \tilde q^{12}$$
so it will suffice to show that
$$ |(H \cap U^1) (H \cap T^1 n_{\w_0})(H \cap  U^1)| = (\frac{1}{C} + O(1/\tilde q)) \tilde q^{12}.$$
The representative $n_{\w_0}$ of $\w_0$ can be chosen\footnote{For instance, we may take it to be the antidiagonal matrix with entries $1,-1,-1,1,1,-1,-1,1$.} to lie in $H$, so that
$$|H \cap T^1n_{\w_0}| = |H \cap T^1| = \frac{1}{C} |T^1|$$
and thus
$$ |H \cap T^1| = (\frac{1}{C} + O(1/\tilde q)) \tilde q^4.$$
Since $H \cap U^1$ is a maximal unipotent subgroup of $H$,
$|H \cap U^1| = q \tilde q = \tilde q^4$, giving the claim.

Now we verify (iii).  Let $\pi\colon \SL_2(\F_q) \times \SL_2(\F_{\tilde q}) \to H$ be the obvious surjective homomorphism.  It suffices to show that
$$ |\{ (x,y) \in \SL_2(\F_q) \times \SL_2(\F_{\tilde q}): \tr(\pi(x,y)) \in \F_{q'} \}| \ll q^{-1/6} |\SL_2(\F_q)| |\SL_2(\F_{\tilde q})|,$$
where we view $\pi(x,y)$ as an element of $\Mat_8(\F_q)$ in order to take an eight-dimensional trace.   We claim the formula
$$
\tr(\pi(x,y)) = a b + a^{\tilde q} a^{\tilde q^2}$$
where $a = \tr(x) \in \F_q$ is the two-dimensional trace of $x$, and similarly $b = \tr(y) \in \F_{\tilde q}$ is the two-dimensional trace of $y$.  This follows by noting the $8$-dimensional representation
restricted to $H$ is a direct sum of two irreducible $4$-dimensional representations (which are precisely
the eigenspaces of the central involution if $q$ is odd), arising from $\operatorname{span}(f_2,f_3,f_6,f_7)$ and $\operatorname{span}(f_1,f_4,f_5,f_8)$, where $f_1,\ldots,f_8$ is the standard basis for $\k^8$.   The first irreducible composition factor is just
the tensor product of the two natural two dimensional representations of the $\SL_2$ factors.  The other irreducible is the fixed space of $\SL_2(\tilde q)$ and is the tensor product of the two nontrivial Frobenius twists of the natural two dimensional module for $\SL_2(q)$.

Observe that each trace $a \in \F_q$ is attained by $O(q^2)$ values of $x$, and each trace $b \in \F_{\tilde q}$ is similarly attained by $O(\tilde q^2)$ values of $y$.  So it suffices to show that
$$ |\{ (a,b) \in \F_q \times \F_{\tilde q}: ab + a^{\tilde q} a^{\tilde q^2} \in \F_{q'}  \}| \ll q^{-1/6} q \tilde q.$$
We may delete the contribution of the case $a=0$ as it is certainly acceptable.  Now note that as $\tilde q = q^{1/3}$ and $q' = q^{1/2}$, any non-zero dilate of $\F_{\tilde q}$ and any translate of $\F_{q'}$ can meet in a set of size at most $|\F_{\tilde q} \cap \F_{q'}| = q^{1/6}$.  We thus see that for fixed non-zero $a$, there are at most $q^{1/6} = q^{-1/6} \tilde q$ choices of $b \in \F_{\tilde q}$ for which $ab+a^{\tilde q} a^{\tilde q^2} \in \F_{q'}$.  The claim follows.
\end{proof}

Next, we need to show a version of Theorem \ref{mainthm} for $H$.

\begin{lemma}[Expansion in good subgroup]\label{expand-good}  Let $H \coloneqq  \SL_2(\F_{q}) \circ \SL_2(\F_{\tilde q})$, and let $a,b$ be chosen uniformly at random from $H$.  Then with probability $1-O(|H|^{-\delta})$, $\{a,b\}$ is $\eps$-expanding in $H$ for some absolute constants $\eps,\delta>0$.
\end{lemma}

\begin{proof} This is a special case of Theorem \ref{mainthm2} proved in the next section. Note that there is no circular argument here, because none of the two simple factors of $H$ are of the $D_4$ type. 
\end{proof}

We can now deal with the degenerate case:

\begin{lemma}\label{deg}  If $n \geq C_1 \log |H|$ for a sufficiently large absolute constant $C_1$, then we have
$$ \P_{w \in W_{n,2}}(w \hbox{ degenerate}) \ll |\tilde G|^{-3\gamma}.$$
\end{lemma}

\begin{proof}  We may take $q$ to be large.  By Lemma \ref{expand-good} and \eqref{sl2-prop-1}, we may find $a_0,b_0 \in H$ such that $\{a_0,b_0\}$ is $\eps$-expanding in $H$ for some absolute constant $\eps>0$, and such that $a_0, b_0$ both lie in the Bruhat cell $U^1 T^1 n_{\w_0} U^1$.  If we let $\mu_{a_0,b_0} \coloneqq  \frac{1}{4} (\delta_{a_0} + \delta_{b_0} + \delta_{a_0^{-1}} + \delta_{b_0^{-1}})$, then from the rapid mixing formulation of expansion (as discussed in the introduction) we have the uniform distribution
\[ \Vert \mu_{a_0,b_0}^{(n)} - {\bf u}_H \Vert_{L^{\infty}(H)} \leq |H|^{-10}\]
for any $n \geq C_1 \log |H|$. and a sufficiently large absolute constant $C_1$, where ${\bf u}_H$ is the uniform distribution of $u$.
For such $n$, this bound and \eqref{sl2-prop-2} imply the estimate
\[ \P_{w \in W_{n,2}} (w(a_0,b_0) \in X) \ll |H|^{-c} \ll |\tilde G|^{-c'}\]
for some absolute constants $c,c'>0$, which clearly implies the required bound.
\end{proof}

In view of the above lemma, to conclude the proof of Theorem \ref{mainthm} in the triality case it suffices to show that
$$\P_{a,b \in U^1 T^1 n_{\w_0} U^1} (w(a,b) \in X ) \ll |\tilde G|^{-3\gamma}.$$
for any non-degenerate word $w \in W_{n,2}$, and for $n$ of the form $2 \lfloor C_2 \log |G| \rfloor$ for some sufficiently large $C_2$.

Fix $w$.  The condition that $w(a,b)$ lies in $X$ is equivalent to the eight equations
\[ \gamma_j(w(a,b))^{q'} - \gamma_j(w(a,b)) = 0, \qquad j = 1,\dots,8.\]
By assumption, at least one of these equations does not hold identically for $a,b \in U^1 T^1 n_{\w_0} U^1$. Let $j \in \{1,\dots,8\}$ be such that $\gamma_j(w(a,b))^{q'} - \gamma_j (w(a,b))$ is not identically zero on $U^1 T^1 n_{\w_0} U^1$.  It will then suffice to show that
$$\P_{a,b \in U^1 T^1 n_{\w_0} U^1} (\gamma_j(w(a,b))^{q'} - \gamma_j (w(a,b)) = 0) \ll |\tilde G|^{-3\gamma}.$$

As $w$ has length at most $n$, we know that $\gamma_j(w(a,b))$ is a polynomial combination (over $\F_q$) of $a,b$ of degree $O(n)$.  Unfortunately, the operation of raising to the power $q'$ increases this degree to the unacceptably high level of $O(nq')$.  To get around this difficulty we will reparameterise in terms of a smaller field $\F_{q_0}$ than $\F_q$ in order to linearise the Frobenius map $x \mapsto x^{q'}$.

We turn to the details.  Recall from the treatment of the Steinberg groups in Section \ref{non-conc-sec-untwist} that all the products in $U^1 T^1 n_{\w_0} U^1$ are distinct.  Also, we can express $U^1$ as the product of $O(1)$ groups $X_S$, each of which take the form
$$ \{ x_\alpha(t): t \in \F_{\tilde q} \}$$
or
$$ \{ x_\alpha(t) x_{\overline{\alpha}}(t^{\tilde q}) x_{\overline{\overline{\alpha}}}(t^{\tilde q^2}): t \in \F_q \},$$
and that all the products in this decomposition of $U^1$ are distinct, and the $x_\alpha$ are polynomial maps of degree $O(1)$.  Similarly, the finite abelian group $T^1$ is generated by groups $H_J^1$, each of which take the form
$$\{ h_\alpha(t): t \in\F_{\tilde q}^\times \} $$
or
$$ \{ h_\alpha(t) h_{\overline{\alpha}}(t^{\tilde q}) h_{\overline{\overline{\alpha}}}(t^{\tilde q^2}): t \in \F_q^\times \}$$
where the $h_\alpha$ are rational maps of degree $O(1)$ whose denominators are monomials.  The maps $t \mapsto t^{\tilde q}$, $t \mapsto t^{\tilde q^2}$ are of high degree when viewed as polynomials over $\F_q$, but if one instead views $\F_q$ as a three-dimensional vector space over $\F_{\tilde q}$ (thus $\F_q \equiv \F_{\tilde q}^3$ and $\F_q^\times \equiv \F_{\tilde q}^3 \backslash \{0\}$), then these maps become linear over $\F_{\tilde q}$.
Using the above factorisations, we can thus parameterise $U^1 T^1 n_{\w_0} U^1$ in terms of $O(1)$ coordinates, drawn from $\F_{\tilde q}$,  $\F_{\tilde q}^\times$, and $\F_{\tilde q}^3 \backslash\{0\}$, using rational maps whose numerators and denominators are polynomials of degree $O(1)$, with the denominators non-vanishing on the domain of the parameters.  This parameterisation is not quite unique (because the products of elements from $H_J^1$ can collide), but each element of $U^1 T^1 n_{\w_0} U^1$ has the same number of representations by such a parameterisation.  Taking products, we see that we can parameterise an element $(a,b)$ of $U^1 T^1 n_{\w_0} U^1 \times U^1 T^1 n_{\w_0} U^1$ by $O(1)$ coordinates from $\F_{\tilde q}$, $\F_{\tilde q}^\times$, and $\F_{\tilde q}^3 \backslash \{0\}$ using the above rational maps, and with each $(a,b)$ having the same number of representations.

From \eqref{q-form} we can view $\F_{\tilde q}$ as a two-dimensional vector space over $\F_{q_0}$, so that we can view the above parameterisation of $(a,b) \in U^1 T^1 n_{\w_0} U^1 \times U^1 T^1 n_{\w_0} U^1$ as being in terms of a bounded number $s_1,\ldots,s_l$ of variables in $\F_{q_0}$, omitting some coordinate hyperplanes associated to the missing origin in $\F_{\tilde q}^\times$ or $\F_{\tilde q}^3 \backslash \{0\}$.  Let $E \subset \F_{q_0}^l$ be the domain of the parameters $(s_1,\ldots,s_l)$ (thus $E$ is $\F_{q_0}^l$ with some coordinate hyperplanes removed).  Observe that the Frobenius map $x \mapsto x^{q'}$ on $\F_q$, while having large degree when viewed as a polynomial over $\F_q$, becomes linear over $\F_{q_0}$ when $\F_q$ is viewed as a (six-dimensional) vector space over $\F_{q_0}$.  From this, we see that the quantity
$$ \gamma_j(w(a,b))^{q'} - \gamma_j (w(a,b))$$
is a rational function $P(s_1,\ldots,s_l)/Q(s_1,\ldots,s_l)$ of the parameters $s_1,\ldots,s_l$, with numerator $P$ and denominator $Q$ being polynomials of degree $O(n)$, and the denominator $Q$ non-vanishing on $E$.  It will thus suffice to show that
\begin{equation}\label{doo}
\P_{(s_1,\ldots,s_l) \in E}(P(s_1,\ldots,s_l) = 0) \ll |\tilde G|^{-3\gamma}.
\end{equation}
(Here we use the fact that each $(a,b)$ is parameterised by the same number of tuples $(s_1,\ldots,s_l)$.)  On the other hand, since $\gamma_j(w(a,b))^{q'} - \gamma_j (w(a,b))$ is non-vanishing for at least one $(a,b) \in \gamma_j(w(a,b))^{q'} - \gamma_j (w(a,b))$, $P$ is non-vanishing on $E$.  Applying Lemma \ref{basic-schwartz-zippel} (and noting that $|E|$ is comparable to $q_0^l$) we conclude that
$$\P_{(s_1,\ldots,s_l) \in E}(P(s_1,\ldots,s_l) = 0) \ll n q_0^{-1},$$
and the claim \eqref{doo} follows from \eqref{lang-weil} and the choice of $n$.  This concludes the proof of Theorem \ref{mainthm} in the triality group case.

\section{Products of finite simple groups}\label{semisimple}

In this last section, we extend the main result of the paper to products of a bounded number of finite simple (or quasisimple) groups of Lie type of bounded rank.  Here, by \emph{product} we mean either a direct product, or an almost direct product, that is to say a quotient of a direct product by a central subgroup.  In particular, this includes central product $\SL_2(\F_q) \circ \SL_2(\F_{\tilde q})$ of the quasisimple groups $\SL_2(\F_q), \SL_2(\F_{\tilde q})$ that appears in Section \ref{d4-sec}.

The Bourgain-Gamburd method applies here again, but needs to be appropriately modified. For instance, the product theorem (Proposition \ref{prod}) above is no longer true in a direct product of groups $G\coloneqq G_1 \times G_2$, because approximate subgroups of the form $A=A_1 \times A_2$ with for instance $A_1$ small in $G_1$ and $A_2$ large in $G_2$ may generate $G$ but be neither small nor large in $G$.   The ``correct'' product theorem for such groups is as follows.

\begin{theorem}[Product theorem for semisimple groups]\label{semisimpleproduct} Let $G$ be a product of finite simple (or quasisimple) groups of Lie type and suppose that $A$ a $K$-approximate subgroup of $G$. Then either $|A| \geq |G|/K^C$, or there is a proper subgroup $H$ of $G$ and $x \in G$ such that $|A \cap xH| \geq |A|/K^C$, where $C>0$ is a constant depending only on the rank of $G$.
\end{theorem}

\begin{proof}
All constants will depend only on the rank of $G$. Clearly (up to passing to a finite quotient) we may assume that $G$ is a direct product of finite quasisimple groups of Lie type. We will prove by induction on the rank of $G$ the slightly stronger statement that either $A^m=G$ for some constant $m$, or there is a proper subgroup $H$ such that $|A \cap xH| > |A|/K^C$ for some coset $xH$ of $H$. So suppose $G \simeq G_1 \times G_2$, where $G_2$ is a simple group. By induction, we may assume that, for some constant $m$, $A^m$ projects onto $G_1$ and onto $G_2$ (otherwise a large part of $A$ will be contained in a coset of the pull-back of a proper subgroup of either $G_1$ or $G_2$).

We view $G_1,G_2$ as commuting subgroups of $G$.  Suppose $A^{4m} \cap G_2$ is non trivial, so that it contains some element $x \neq 1$. Then for some larger but still bounded $m'$ we must have $G_2 \subset A^{m'}$. Indeed, for every $g_2 \in G_2$ there is $g_1 \in G_1$ such that $g_1g_2 \in A^m$ and thus
$$ g_2 x g_2^{-1} = (g_1 g_2) x (g_1 g_2)^{-1} \in A^m A^{4m} A^m = A^{6m},$$
so that $A^{6m}$ contains a full conjugacy class of $G_2$. But since the rank of $G_2$ is bounded, it is known (and follows as well from the product theorem, Proposition \ref{prod}) that for some constant $M$, $\mathcal{C}^M=G_2$ for every non trivial conjugacy class $\mathcal{C}$ of $G_2$. It follows that $A^{6mM}$ contains $G_2$ and hence $A^{12mM}=G$.

Hence we may assume that $A^{4m} \cap G_2$ is trivial. But then (arguing as in Goursat's lemma), this means that for every $g_1 \in G_1$, there is one and only one $g_2 \in G_2$ such that $g_1g_2 \in A^m$. Let $\alpha \colon G_1 \to G_2$ be the map thus defined. The hypothesis $A^{4m} \cap G_2$ is trivial similarly implies that $\alpha$ is a group homomorphism. This means that $A^m$ is contained in the proper subgroup of $G$ defined by $\{g_1\alpha(g_1): g_1 \in G_1\}$ and we are done.
\end{proof}

\begin{remark}  In many important special cases, one can embed $G$ in a linear group $\GL_n(\F)$ for some finite field $\F$ and $n$ bounded in terms of the rank of $G$.  In this case, we may also establish Theorem \ref{semisimpleproduct} using the results of Pyber and Szab\'o.  Indeed, \cite[Theorem 10]{pyber-szabo} asserts that if $A$ is a $K$-approximate subgroup of $\GL_n(\F)$, then there are normal subgroups $L\leq \Gamma$ of $\langle A \rangle$ such that $L$ is solvable, $\Gamma \subset A^6L$ and $A$ is covered by $K^m$ cosets of $\Gamma$, where $m$ depends only on $n$ (hence on the rank of $G$). Now in Theorem \ref{semisimpleproduct} we may assume that $\langle A \rangle = G$, since otherwise we could take $H \coloneqq  \langle A \rangle$.  Hence $L$, being normal and solvable in $G$ must be central and bounded in cardinality. Now if $\Gamma=G$, then $\Gamma \subset A^6L$ implies that $|A| \geq |G|/O(K^5)$, while if $\Gamma$ is a proper subgroup, then the fact that $A$ is covered by $K^m$ cosets of $\Gamma$ implies that $|A \cap x\Gamma| \geq |A|/K^m$ for at least one cosets $x\Gamma$ and thus Theorem \ref{semisimpleproduct} follows in this case.
\end{remark}

Using Theorem \ref{semisimpleproduct}, we can now give the following extension of Theorem \ref{mainthm} to product groups.

\begin{theorem}[Random pairs expand in semisimple groups]\label{mainthm2} Suppose that $G$ is a product of finite simple \textup{(}or quasisimple\textup{)} groups of Lie type and that $a, b \in G$ are selected uniformly at random. Let $S$ be the smallest simple factor of $G$. Then with probability at least $1-C|S|^{-\delta}$, the pair $\{a,b\}$ is  $\eps$-expanding
for some $C,\eps,\delta> 0$ depending only on the rank of $G$.
\end{theorem}

We will deduce this result from our main theorem, Theorem \ref{mainthm}. To prove Theorem \ref{mainthm2} for a particular $G$, one only needs to know Theorem \ref{mainthm} for those simple factors $S$ contained in $G$. Thus our proof of Theorem \ref{mainthm} for the Triality groups ${}^3 D_4(q)$ in Section \ref{d4-sec} is not circular, because we need Theorem \ref{mainthm2} only in the case that $G$ is a central product of $\SL_2(\F_q)$ and $\SL_2(\F_{\overline{q}})$.

The proof of Theorem \ref{mainthm} will split into two cases. We will first establish Proposition \ref{nocommonfactor} below, which asserts that a generating set of $G$ is expanding if and only if its projection to its primary components are expanding. Then we will treat separately the case of a power of a single finite simple group of Lie type. Theorem \ref{mainthm2} will then follow immediately.

\begin{proposition}\label{nocommonfactor} Let $r \in \N$ and $\eps>0$. Suppose $G=G_1G_2$, where $G_1$ and $G_2$ are products of at most $r$ finite simple \textup{(}or quasisimple\textup{)} groups of Lie type of rank at most $r$. Suppose that no simple factor of $G_1$ is isomorphic to a simple factor of $G_2$. If $x_1=x_1^{(1)}x_1^{(2)},\ldots, x_k=x_k^{(1)}x_k^{(2)}$ are chosen so that $\{x_1^{(1)},\ldots,x_k^{(1)}\}$ and $\{x_1^{(2)},\ldots,x_k^{(2)}\}$ are both $\eps$-expanding generating subsets in $G_1$ and $G_2$ respectively, then $\{x_1,\ldots,x_k\}$ is $\delta$-expanding in $G$ for some $\delta=\delta(\eps,r)>0$.
\end{proposition}

\begin{remark} The assumption that $G_1$ and $G_2$ have no common simple factor is necessary in this proposition. Indeed consider for example the case when $G_1 \simeq G_2$ are finite simple groups both isomorphic to (say) $\PSL_2(\F_q)$, and choose $\eps$-expanding pairs $(a_1,b_1)$ in $G_1$ and $(a_2,b_2)$ in $G_2$ such that $a_2=\phi(a_1)$ and $b_2=\phi(b_1)$ for some group isomorphism $\phi \colon G_1 \to G_2$. Then $(a_1a_2,b_1b_2)$ does not generate $G=G_1 \times G_2$, hence cannot $\delta$-expand for any fixed $\delta>0$ as $q$ tends to $\infty$. The key point here is that under the assumption of the proposition, Goursat's lemma forces $x_1,\ldots,x_k$ to generate $G$.
\end{remark}

\begin{proof}  We follow the Bourgain-Gamburd method (see Appendix \ref{bg-app}) and proceed by induction on the rank of $G$. Recall (see Definition \ref{specex}) that the $x_i$'s give rise to the averaging operator $T$. Let $\rho$ be an irreducible component of the representation of $G$ associated to the eigenvalue of $T$ with largest modulus (different from $1$). Let $H=\ker \rho$. It is a normal subgroup of $G$, and hence is either central and bounded, or contains a quasisimple non-abelian subgroup of $G$. In the latter case, the quotient $\rho(G)\simeq G/H$ will have strictly smaller rank than $G$ and will satisfy the same assumptions as $G$. Namely $\rho(G)=\rho(G_1) \cdot \rho(G_2)$ and both $\{\rho(x_1)^{(1)},\ldots,\rho(x_k)^{(1)}\}$ and $\{\rho(x_1)^{(2)},\ldots,\rho(x_k)^{(2)}\}$ are $\eps$-expanding in $\rho(G_1)$ and $\rho(G_2)$ respectively (indeed all irreducible components of $\ell^2(\rho(G_i))$ occur already in $\ell^2(G_i)$). By induction hypothesis, we are then done. In the former case, $\rho$ is of the form $\sigma_1 \otimes \ldots \otimes \sigma_m$, where the $\sigma_i$'s are non trivial irreducible representations of the $m$ quasisimple normal subgroups $S_i$ of $G\simeq S_1\ldots S_m$. In particular in view of the quasirandomness of quasisimple groups of Lie type (Proposition \ref{quasi}) there is $\beta=\beta(r)>0$ such that $\dim \sigma_i \gg |S_i|^\beta$ for all $i=1,\ldots, m$, and thus 

\[ \dim \rho \geq \prod_1^m \dim \sigma_i \gg \prod_1^m |S_i|^\beta \geq |G|^\beta.\]

We are thus in a position to apply the Bourgain-Gamburd method, as adapted to the product setting in \cite{bv}. Without loss of generality, we may assume that $|G_1| \geq |G_2|$. Let $\overline{G_i}$ be the quotient of $G_i$ modulo its center. Then we have epimorphisms $\pi_i \colon G \to \overline{G_i}$. Let $\mu$ be the probability measure associated to $\{x_1,\ldots,x_k\}$ as in Definition \ref{specex}. It follows from our assumption that $\{\pi_1(x_1),\ldots,\pi_1(x_k)\}$ is $\eps$-expanding in $\overline{G_1}$. This implies that
$$\|\pi_1(\mu)^n - 1\|_2 \leq (1-\eps)^n \|\pi_1(\mu)\|_2 \leq (1-\eps)^n |G_1|^{1/2}$$ for all $n \in \N$. However since $\overline{G_1}$ is a quotient of $G$, we have
 $$\frac{1}{|G|}\|\mu^n\|_2^2 \leq \frac{1}{|\overline{G_1}|} \|\pi_1(\mu)^n\|_2^2$$
 since the probability of return to the identity in $G$ is at most the probability of return to the identity in $\overline{G_1}$.  Therefore, since $|\overline{G_1}| \gg |G|^{1/2}$, for any $\kappa \in (0,\frac{1}{4})$ we have
\begin{equation}\label{decay}
\|\mu^n\|_2 \leq \frac{|G|^{1/2}}{|\overline{G_1}|^{1/2}}(1+e^{-\eps n}|\overline{G_1}|^{1/2}) \ll |G|^{1/2 - \kappa}
\end{equation}
 if $n \geq \frac{\kappa}{\eps}\log |G|$, securing Phase I of the Bourgain-Gamburd method (see Appendix \ref{bg-app}).

Using the fact, proved in the first paragraph above, that $\dim \rho \gg |G|^\beta$, and in view of Phase III of the Bourgain-Gamburd method, we are left to show that for some $n \leq C_0 \log |G|$ (where $C_0$ depends only on $\eps$ and $r$) $\|\mu^n\|_2  \ll |G|^{\beta/10}$.

Suppose first that $|G_2| \leq |G_1|^{\beta/5}$. Then $\frac{|G|}{|\overline{G_1}|} \ll |G|^{\beta/5}$ and estimate $(\ref{decay})$ above already shows that
$$\|\mu^n\|_2  \ll |G|^{\beta/10}$$
if $n \geq \frac{1/2-\beta/5}{\eps}\log |G|$. We are therefore done in this case.

We may then assume that $|G_2| \geq |G_1|^{\beta/5}$ and proceed with Phase II of the Bourgain-Gamburd method. This requires the modified product theorem, Theorem \ref{semisimpleproduct}. Going through Phase II as described in Appendix \ref{bg-app} and using the weighted Balog-Szemer\'edi-Gowers lemma proved in Appendix \ref{bsg-app-sec}, we obtain (keeping the same notation) a $|G|^{\eta}$-approximate group $A$, where $\eta=c\beta$ for some small constant $c>0$, and an integer $m=O_c(\log |G|)$ such that
\begin{equation}\label{largemass}
\mu^m(xA) \geq |G|^{-\eta}
\end{equation}
for some $x \in G$ and
 $$|A| \leq |G|^{1+\eta}/\|\mu^m\|_2^2.$$
As in Appendix \ref{bg-app} we examine the various possibilities for $A$ given by the product theorem (Theorem \ref{semisimpleproduct}) applied to $A$ with $K=|G|^{\eta}$. If $|A|\geq |G|/K^C$, then the last displayed equation yields
$$\|\mu^m\|_2^2 \leq |G|^{\beta/10}$$
if $c$ is chosen small enough. So we are done in this case.

We are going to rule out the other cases. According to Theorem \ref{semisimpleproduct}, there must exist a proper subgroup $H$ of $G$ and $y \in G$ such that $|A\cap yH| \geq |A|/K^C$. We seek a contradiction. By Ruzsa's covering lemma (see e.g. \cite[Lemma 3.6]{tao-noncommutative}), $A$ is contained in at most $K^C$ left cosets of $H$. In view of $(\ref{largemass})$ this implies that
$$\mu^m(x_0H) \geq |G|^{-(1+C)\eta}$$
for at least one coset $x_0H$ of $H$. Since $\mu$ is symmetric, this implies
\begin{equation}\label{largemass2}
\mu^{2m}(H) \geq |G|^{-2(1+C)\eta}
\end{equation}
However $\mu^{2m}(H)$ is non-increasing as a function of $m$, hence this bound holds for every smaller $m$.

We only now make use of the assumption that $G_1$ and $G_2$ have no common simple factor: this assumption implies that every proper subgroup of $G$ projects to a proper subgroup of either $\overline{G_1}$ or of $\overline{G_2}$. This is an instance of Goursat's lemma. So for some $i_0$ equal to $1$ or $2$, we have that $\overline{H}\coloneqq \pi_{i_0}(H)$ is proper in $\overline{G_{i_0}}$.

By assumption $|G_1|\geq |G_2| \geq |G_1|^{\beta/5}$. This means that $|G| \leq |G_1|^2$ and $|G| \leq |G_2|^{1+5/\beta}$, so $|G|\leq |G_i|^{D}$ for each $i=1,2$ if we take $D=1+5/\beta$. Hence $(\ref{largemass2})$ implies
\begin{equation}\label{cont1}
\pi_{i_0}(\mu)^{2n}(\overline{H}) \geq |\overline{G_{i_0}}|^{2D(1+C)\eta}
\end{equation}
for all $n \leq m=O_c(\log |G|)$ and in particular for $n$ of size $O(\frac{1}{\eps}\log |\overline{G_{i_0}}|)$. However as we will see, this is in conflict with the assumption that $\pi_{i_0}(\mu)$ is $\eps$-expanding in $\overline{G_{i_0}}$.  Indeed this assumption implies that for $n = O(\frac{1}{\eps}\log |\overline{G_{i_0}}|)$,
$$\|\pi_{i_0}(\mu)^n - 1\|_{L^2(\overline{G_{i_0}})} \leq e^{-\eps n} | \overline{G_{i_0}}|^{1/2} \leq 1$$
hence
\begin{equation}\label{cont2}
\pi_{i_0}(\mu)^{2n}(\overline{H}) \leq |\overline{H}|\pi_{i_0}(\mu)^{2n}(1) = \frac{|\overline{H}|}{|\overline{G_{i_0}}|} \|\pi_{i_0}(\mu)^{n}\|_{L^2(\overline{G_{i_0}})}^2 \leq 2\frac{|\overline{H}|}{|\overline{G_{i_0}}|}.
\end{equation}
However, it follows\footnote{One can also establish this claim from quasirandomness, Proposition \ref{quasi}, by considering the quasiregular representation associated to a proper subgroup.} from the classification of maximal subgroups of simple groups of Lie type (see Proposition \ref{max} above) that every proper subgroup of a product $L$ of finite simple or quasisimple groups of Lie type has index at least $|L|^{\gamma}$, where $\gamma>0$ is a constant depending only on the rank of $L$. Now (\ref{cont1}) and (\ref{cont2}) are incompatible if $\eta$ is small enough. This gives the desired contradiction and ends the proof of Proposition \ref{nocommonfactor}.
\end{proof}

\bigskip

We are now ready for the proof of Theorem \ref{mainthm2}.

We first prove the result in the special case when all simple factors of $G$ are isomorphic to a single simple group. If $G$ is simple itself, then this is our main theorem, Theorem \ref{mainthm}. In fact we proved the result for the quasisimple bounded cover group $\tilde G$ sitting above $G$ and remarked at the beginning of Section \ref{outlinesection} that it implies the expansion result for $G$. So we have in fact proved the result for all quasisimple finite groups of Lie type. For the same reason it is enough to prove the result in the case when $G$ is a direct (as opposed to almost direct) product of say $r$ copies of a single quasisimple group $S$.

We now pass to the case when $G=S^r$ and $r \geq 2$ and $S$ is quasisimple. The proof follows again the Bourgain-Gamburd method (see Proposition \ref{machine}). Quasirandomness (item (i) in Proposition \ref{machine}) holds in our case, because any non trivial representation of $G$ must be non trivial on some factor $S$ of $G$, hence must have dimension at least $|S|^\beta = |G|^{\beta/r}$, where $\beta>0$ depends only on the rank of $S$. As mentioned at the beginning of this section, the product theorem (item (ii) in Proposition \ref{machine}) does not hold as such, but as we will see Theorem \ref{semisimpleproduct} will serve as a replacement. Item (iii), that is the non-concentration estimate, was the main part of our proof of Theorem \ref{mainthm}. See Proposition \ref{nonc} above. Let us now show that it holds as well in our case, when $G=S^r$.

We recall a standard group-theoretic lemma.

\begin{lemma} Let $S_1,\ldots,S_r$ be perfect groups for some $r \geq 2$, and let $H$ be a proper subgroup of $G \coloneqq  S_1 \times \ldots \times S_r$.  Then there exists $1 \leq i < j \leq r$ such that the projection of $H$ to $S_i \times S_j$ is also proper.
\end{lemma}

\begin{proof} Suppose for contradiction that all projections from $H$ to $S_i \times S_j$ are surjective.  In particular, the projections of $H$ to $S_1 \times S_2$ and $S_1 \times S_3$ contain $S_1 \times \{1\}$ and $S_1 \times \{1\}$ respectively.  Taking commutators and using the hypothesis that $S_1$ is perfect, we conclude that the projection of $H$ to $S_1 \times S_2 \times S_3$ contains $S_1 \times \{1\} \times \{1\}$.  Iterating this, we conclude that $H$ itself contains $S_1 \times \{1\} \times \ldots \times \{1\}$, and similarly for permutations.  But then $H$ must be all of $S_1 \times \ldots \times S_n = G$, a contradiction.
\end{proof}

In view of this lemma, we see that every proper subgroup of $S^r$ has at least one of its $\binom{r}{2}$ projections to $S^2$ proper.  From this and the union bound, we see that to prove the non-concentration bound for all $r \geq 2$, it suffices to do so in the $r=2$ case.

It thus remains to verify the $r=2$ case.  In this case, Goursat's lemma tells us that $H$ projects onto a proper subgroup of either $S_1$ or $S_2$, or it is contained in the pull-back modulo the center of a ``diagonal'' subgroup of the form $\{(s,\alpha(s)), s \in \overline{S}\}$, where $\overline{S}$ is the simple quotient of $S$. In the former case, we can apply the non-concentration estimate proved earlier in this paper in the quasisimple case. The latter case reduces to proving the non-concentration estimate for random pairs in $\overline{S} \times \overline{S}$ for the subgroups of the form $H_\alpha\coloneqq \{(s,\alpha(s)), s \in \overline{S}\}$. This is dealt with in Proposition \ref{productconcentration} below.

So we have now established item (iii) of Proposition \ref{machine} for almost all (i.e. $\gg |G|^2(1-Cr^2/|S|^\delta)$) pairs $(a,b)$. We also checked that item (i) holds. As already mentioned item (ii), the product theorem, does not hold as such in $G=S^r$. Let us show how Theorem \ref{semisimpleproduct} can serve as a replacement. The product theorem is only used in Phase II of the Bourgain-Gamburd method; see Appendix \ref{bg-app}. Keeping the same notation as in this appendix, and setting $\mu=\mu_{a,b}$ (as in Definition \ref{specex}) and picking $\eta>0$ a small constant to be specified later, we obtain a $|G|^{\eta}$-approximate subgroup $A$ of $G$ such that
\begin{equation}\label{largema}
\mu^m(xA) \geq |G|^{-\eta}
\end{equation}
for some $m=O_\eta(\log|G|)$ and some $x \in G$ and
 $$|A| \leq |G|^{1+\eta}/\|\mu^m\|_2^2.$$
The modified product theorem (Theorem \ref{semisimpleproduct}) tells us that either $|A| \geq |G|^{1-C\eta}$, in which case we are done by the same argument as in Appendix \ref{bg-app}, or there is a proper subgroup $H$ of $G$ which intersects a translate of $A$ in a large subset. By the Ruzsa covering lemma (see e.g. \cite[Lemma 3.6]{tao-noncommutative}), this implies that $xA$ is contained in at most $|G|^{C\eta}$ cosets of $H$, and hence (by $(\ref{largema})$) that at least one of these cosets, say $x_0H$ satisfies
$$\mu^m(x_0H) \geq |G|^{-(1+C)\eta}.$$
and thus
$$\mu^{2m}(H) \geq |G|^{-2(1+C)\eta}.$$
Since $m \mapsto \mu^{2m}(H)$ is non-increasing, this holds for all smaller $m$'s and thus provides a contradiction to the non-concentration estimate once $\eta$ is chosen small enough. This ends the proof of Theorem \ref{mainthm2} in the case when $G$ has all its simple factors in the same isomorphism class.

The general case of the theorem (i.e. when $G$ has possibly several non-isomorphic simple factors) is now an easy consequence of Proposition \ref{nocommonfactor} above. Indeed $G$ can be written as $G=S_1^{r_1} \cdot \ldots S_k^{r_k}$, where the $S_i$'s are pairwise non-isomorphic quasisimple groups of bounded rank, and the $r_i$'s are bounded. We know that there is $\eps,\delta,C>0$ depending only on the rank on $G$ such that one may choose at least $|S_i^{r_i}|^2(1-C|S_i|^{-\delta})$ pairs $(a_i,b_i)$ in $S_i^{r_i}$ which generate $S_i^{r_i}$ and are $\eps$-expanding. In view of Proposition \ref{nocommonfactor} applied repeatedly, we get that all pairs $(a_1\cdot \ldots \cdot a_k, b_1\cdot \ldots \cdot b_k)$ are $\eps'$-expanding generating pairs in $G$, where $\eps'$ depends only on the rank of $G$. This makes at least $|G|^2 \prod_1^k (1-C|S_i|^{-\delta})$ pairs, hence at least $|G|^2 (1 - Ck|S|^{-\delta})$ pairs. This ends the proof of Theorem \ref{mainthm2}, conditionally to the proof of Proposition \ref{productconcentration} which we are now ready to give.\vspace{11pt}

As promised we now turn to the proof of the non-concentration estimate for random pairs in a group of the form $S \times S$, where $S$ is a finite simple group of Lie type.

\begin{proposition}\label{productconcentration} Let $S$ be a finite simple group of Lie type. There are constants $\Lambda,\kappa, C,\delta>0$ depending only on the rank of $S$ such that for at least $|G|^2(1-C/|G|^\delta)$ pairs $(a,b)$ in $G\coloneqq S \times S$, the probability measure $\mu_{a,b}=\frac{1}{4}(\delta_a + \delta_{a^{-1}} +  \delta_b + \delta_{b^{-1}})$ satisfies the non-concentration estimate 
\[ \sup_{H < G}\mu_{a,b}^{(n)}(H) < |G|^{-\kappa}\] for some $n \leq \Lambda \log |G|$, where the supremum is over all proper subgroups $H < G$.
\end{proposition}

\begin{proof} Write $a=(a_1,a_2)$ and $b=(b_1,b_2)$ in $G=S\times S$. According to the main result of this paper, there are $\kappa,\delta,C>0$ depending only on the rank of $S$ and at least $|S|^2(1-C/|S|^\delta)$ pairs $(a_1,b_1)$ such that
\[ \sup_{H < S}\mu_{a_1,b_1}^{(n)}(H) < |S|^{-\kappa},\]
 where the supremum is over all proper subgroups $H < S$. The same holds for the second factor. By Goursat's lemma, given that $S$ is simple, every proper subgroup of $G$ must either project to a proper subgroup of one of the two $S$ factors, or be of the form $H_\alpha\coloneqq \{(s,\alpha(s)), s \in S\}$, where $\alpha \in \Aut(S)$ is an automorphism of $S$. This yields at least $|G|^2(1-2C/|S|^\delta)$ pairs $(a,b)$ in $G$ such that
\[ \sup_{H < G, H\neq H_\alpha}\mu_{a,b}^{(n)}(H) < |S|^{-\kappa}= |G|^{-\kappa/2},\]
 where the supremum is over all proper subgroups $H < G$ not of the form $H_\alpha$ for some $\alpha \in \Aut(S)$.

To handle the subgroups $H_\alpha$, we will proceed as before, using the Schwartz-Zippel estimates from Proposition \ref{sufficiently-zdense-prop} and Borel's dominance theorem \cite{borel-dom, larsen-words}. Note that if $a,b$ are such that for some $\alpha \in \Aut(S)$
$$\mu_{a,b}^{(n)}(H_\alpha) \geq |S|^{-\kappa},$$
then, setting $A\coloneqq \Aut(S)$,
$$\P_{w\in W_{n,2}}( w(b_1,b_2) \in A\cdot w(a_1,a_2)) \geq  |S|^{-\kappa},$$
where $W_{n,2}$ is the set of (non-reduced) words in two letters of length exactly $n$.

Following a similar line of argument as in the treatment of the structural case in Section \ref{class}, using Markov's inequality and Fubini we have

\begin{align*}\P_{a,b \in G}(\sup_{H_\alpha < G, \alpha \in A} & \mu_{a,b}^{(n)}(H_\alpha) \geq |S|^{-\kappa}) \\ & \leq
\P_{a,b \in G}(\P_{w\in W_{n,2}}( w(a_2,b_2) \in A\cdot w(a_1,b_1)) > |S|^{-\kappa}) \\ &
\leq |S|^\kappa \E_{a,b \in G}(\P_{w\in W_{n,2}}( w(a_2,b_2) \in A\cdot w(a_1,b_1)))\\ &
\leq |S|^\kappa \E_{w\in W_{n,2}}( \P_{a,b\in G}( w(a_2,b_2) \in A\cdot w(a_1,b_1))).
\end{align*}

Some of the words in $W_{n,2}$ become trivial in the free group. However by Kesten's bound (see Lemma \ref{noncom}) at most an exponentially small fraction of them do so, hence at most $e^{-cn}=|S|^{-c\Lambda}$, where $c>0$ is an absolute constant. Taking $\kappa>0$ very small, we can therefore ignore the contribution of these words and focus on those that do not vanish when reduced in the free group.

It is known (see \cite[Theorem 30]{steinberg-yale} or \cite[p. 211]{carter}) that every automorphism of $S$ is a product of an inner automorphism, a field automorphism, a graph automorphism and a so-called diagonal automorphism. Since the rank of $S$ is bounded, there are only boundedly many graph and diagonal automorphisms. The number of field automorphisms is $O(\log q)=O(\log|S|)$. It follows that every $A$-orbit in $S$ is contained in at most $O(\log|S|)$ conjugacy classes. Therefore for all $n \leq \Lambda \log |G|$ we have

\begin{align*}\P_{a,b \in G}(\sup_{H_\alpha < G, \alpha \in A} & \mu_{a,b}^{(n)}(H_\alpha) \geq |S|^{-\kappa}) \\ & \ll
 |S|^{\kappa + 2\Lambda \log 4} \log|S| \sup_{w\in W^*_{n,2}} \sup_{C \subset S} \P_{a,b\in G}( w(a_2,b_2) \in C),
\end{align*}
where the second supremum is taken over all conjugacy classes in $S$, and $W^*_{n,2}$ is the set of non-trivial reduced words of length at most $n$ in the free group. Thus it suffices to show that there is a positive $\delta_0>0$ depending only on the rank of $S$ such that for every conjugacy class $C$ and every $w \in W^*_{n,2}$ we have
$$\P_{a_2,b_2\in S}( w(a_2,b_2) \in C) \leq \frac{1}{|S|^{\delta_0}}.$$
The preimage in the bounded cover $\tilde S$ of a conjugacy class in $S$ consists of boundedly many conjugacy classes of $\tilde S$. Hence it is enough to prove the above for $\tilde S$ in place of $S$. Now this follows directly from the Schwarz-Zippel estimate of Proposition \ref{sufficiently-zdense-prop} after we note that the subvarieties of $\mathbf{S} \times \mathbf{S}$ defined by $\{(s_1,s_2) ; w(s_1,s_2) \in C\}$ have degree at most $O(n)=O(\log|S|)$ uniformly in the choice of $C$ and $w$. 
\end{proof}

\appendix

\section{A weighted Balog-Szemer\'edi-Gowers theorem}\label{bsg-app-sec}

The aim of this section is to establish  a weighted version of the Balog-Szemer\'edi-Gowers theorem \cite{balog,gowers-4aps}, which will be needed for the proof of Proposition \ref{machine} in Appendix \ref{bg-app}. The precise statement is as follows.  We use the usual $L^p(G)$ norms
$$ \|f\|_{L^p(G)} \coloneqq  (\E_{x \in G} |f(x)|^p)^{1/p}$$
for $1 \leq p < \infty$, with the usual convention
$$ \|f\|_{L^\infty(G)} \coloneqq  \sup_{x \in G} |f(x)|.$$

\begin{proposition}[Weighted BSG]\label{weighted-bsg}
Let $K \geq 1$. Suppose that $\nu : G \rightarrow \R_{\geq 0}$ is a probability measure on some finite group $G$ which is symmetric in the sense that $\nu(x) = \nu(x^{-1})$ and which satisfies $\Vert \nu \ast \nu \Vert_{L^2(G)} \geq \frac{1}{K}\Vert \nu \Vert_{L^2(G)}$, where convolution is defined in \eqref{convdef}.
Then there is a $O(K^{O(1)})$-approximate subgroup $H$ of $G$ with $|H| \ll K^{O(1)} |G|/\Vert \nu \Vert_{L^2(G)}^2$ and an $x \in G$ such that $\nu(Hx) \gg K^{-O(1)}$.
\end{proposition}

\emph{Remark. } It would be possible to formulate a version of this proposition in which $\nu$ is not symmetric, or even with two different measures $\nu,\nu'$ having comparable $L^2(G)$-norms. We do not do this here, since Proposition \ref{weighted-bsg} is all that is required for our applications here. Propositions of this type are not new, appearing for example in the work of Bourgain-Gamburd \cite{bourgain-gamburd} and the paper of Varj\'u. It is possible to extend this result to infinite groups $G$ if one removes the normalisation on the counting measure, but we will not need this extension here.

\begin{proof} It will be convenient to adopt some notational conventions for this proof.  If $A \subseteq G$ is a set, write $\mu_A$ for the uniform measure on $A$, defined by $\mu_A(x) = |G|/|A|$ if $x \in A$ and $\mu_A(x) = 0$ otherwise.  We write $X \lessapprox Y$ or $Y \gtrapprox X$ as an abbreviation for $X \ll K^{O(1)} Y$, and $X \approx Y$ as an abbreviation for $X \lessapprox Y \lessapprox X$.

Set $\delta \coloneqq  \frac{1}{100K^2}$ and $M \coloneqq  10K$ (say), and define
\[ \nu_1 \coloneqq  \nu 1_{\nu \geq M\Vert \nu\Vert_{L^2(G)}^2}\] and
\[ \nu_2 \coloneqq  \nu 1_{\nu \leq \delta\Vert \nu\Vert_{L^2(G)}^2}\]and
\[ \tilde \nu \coloneqq  \nu 1_{\delta \Vert \nu \Vert_{L^2(G)}^2 < \nu < M \Vert \nu \Vert_{L^2(G)}^2} = \nu - \nu_{1} - \nu_{2}.\]
We note that $\nu_{1}$ is small in $L^1(G)$: indeed
\begin{equation}\label{v-large}
\begin{split}
\Vert \nu_{1} \Vert_{L^1(G)}  &\coloneqq  \E_{x \in G} \nu_{1}(x) \\
&\leq \E_{x \in G} \nu(x) \frac{\nu(x)}{M \Vert \nu \Vert_{L^2(G)}^2} \\
&\leq \frac{1}{10K}.
\end{split}
\end{equation}
In contrast, $\nu_{2}$ is small in $L^2(G)$, since
\begin{align*}
\E_{x \in G} \nu_{2}(x)^2 &= \E_{x \in G} \nu(x)^2 1_{\nu(x) \leq \delta\Vert \nu\Vert_{L^2(G)}^2} \\
&\leq \delta \Vert \nu \Vert_{L^2(G)}^2 \E_{x \in G} \nu(x) \\
&= \delta \Vert \nu \Vert_{L^2(G)}^2.
\end{align*}
Therefore
\begin{equation}\label{v-small}
\Vert \nu_{2} \Vert_{L^2(G)} \leq \frac{1}{10K} \Vert \nu \Vert_{L^2(G)}.
\end{equation}
Recall \emph{Young's inequality}, an instance of which is the bound $\Vert f \ast g \Vert_{L^2(G)} \leq \Vert f \Vert_{L^1(G)} \Vert g \Vert_{L^2(G)}$ (this is also easily established from Minkowski's inequality). Starting from the assumption that $\Vert \nu \ast \nu \Vert_{L^2(G)} \geq \frac{1}{K} \Vert \nu \Vert_{L^2(G)}$, we may apply Young's inequality and \eqref{v-large}, \eqref{v-small}, noting that $\nu, \nu_{1},\nu_{2}$ are all symmetric, to obtain
\[ \Vert \nu_* \ast \nu_{1} \Vert_{L^2(G)}, \Vert \nu_* \ast \nu_2\Vert_{L^2(G)} \leq \frac{1}{10K} \Vert \nu \Vert_{L^2(G)}\]
for $\nu_* = \nu,\nu_1,\nu_2$, and hence by the triangle inequality
\begin{equation}\label{nu-tilde-est}
\Vert \tilde \nu \ast \tilde \nu \Vert_{L^2(G)}  \geq \frac{1}{2K}\Vert \nu \Vert_{L^2(G)} \gtrapprox \Vert \nu \Vert_{L^2(G)}.\end{equation}
Using another application of Young's inequality, together with the bound $\Vert \tilde \nu \Vert_{L^1(G)} \leq \Vert \nu \Vert_{L^1(G)} = 1$, we thus have
\begin{equation}\label{l2-lower}\Vert \tilde \nu \Vert_{L^2(G)} \gtrapprox \Vert \nu\Vert_{L^2(G)}.\end{equation}
Setting $A \coloneqq  \Supp(\tilde \nu)$ and noting that \[ \nu(x) \approx \Vert \nu \Vert_{L^2(G)}^2\] uniformly in $x \in A$, it follows easily that $\mu_A(x) \approx \nu(x)$ uniformly in $x$. From \eqref{nu-tilde-est} it hence follows that
\[ \Vert \mu_A \ast \mu_A \Vert_{L^2(G)} \gtrapprox \Vert \mu_A \Vert_{L^2(G)},\] which means that the \emph{multiplicative energy} (that is, the number of solutions to $a_1 a_2^{-1} = a_3 a_4^{-1}$) of $A$ is $\gtrapprox |A|^3$. Applying the (noncommutative) Balog-Szemer\'edi-Gowers theorem for sets \cite{tao-noncommutative}, we find that there is a $O(K^{O(1)})$-approximate group $H$ together with some $x \in G$ such that
\[ |A \cap Hx| \gtrapprox \max(|A|, |H|).\] This implies that
\[ \nu(Hx) \geq \tilde\nu(Hx) \gtrapprox \mu_A(Hx) \gtrapprox 1.\] From \eqref{l2-lower} we also have
\[ |H| \lessapprox |A| = \frac{|G|}{\Vert \mu_A \Vert_{L^2(G)}^2} \lessapprox \frac{|G|}{\Vert \tilde\nu \Vert_{L^2(G)}^2} \lessapprox \frac{|G|}{\Vert \nu \Vert_{L^2(G)}^2},\] thereby confirming the proposition.\end{proof}

\section{Proof of the Bourgain-Gamburd reduction}\label{bg-app}

The purpose of this section is to prove Proposition \ref{machine}.  Thus, we fix a finite group $G$ and a symmetric set $S \subset G$ of $k$ generators that obey the quasirandomness, product theorem, and non-concentration estimate hypotheses of Proposition \ref{machine} for some choices of parameters $\kappa,\Lambda,\delta'()$.  Henceforth we allow all constants in the asymptotic notation to depend on these parameters. Our goal is then to obtain a uniform lower bound on the expansion of $S$.

Consider the convolution operator $T_S : L^2(G) \rightarrow L^2(G)$ on the Hilbert space $L^2(G)$ given by $Tf \coloneqq  f \ast \mu_S$. As stated in the introduction, it will suffice to establish the rapid mixing property
\[ \Vert \mu^{(n)} - {\bf u}_G \Vert_{L^\infty(G)} \leq |G|^{-10}\]
for some $n = O(\log |G|)$.

In \cite{bourgain-gamburd}, Bourgain and Gamburd verify this mixing property, using the three ingredients detailed in the statement of Proposition \ref{machine}. This in turn is undertaken in three phases, described below, corresponding to the ``evolution'' of the convolution power $\mu^{(n)}$ as the time $n$ increases. The aim is to show that $\mu^{(n)}$ becomes more and more spread out.  This ``spread'' will be measured by the smallness of the $L^2$ norm
\[\Vert \mu^{(n)} \Vert_{L^2(G)} \coloneqq  (\E_{g \in G} \mu^{(n)}(g)^2 )^{1/2}.\]
Note from Young's inequality that we have the monotonicity property
\[ |G|^{1/2} \geq \Vert \mu^{(1)} \Vert_{L^2(G)} \geq \Vert \mu^{(2)} \Vert_{L^2(G)} \geq \dots \geq 1.\]

\emph{Phase I.} The aim here is to show that, by time $n \sim C_0 \log |G|$, the measure $\mu^{(n)}$ has become at least reasonably spread out in the sense that \[ \Vert \mu^{(n)} \Vert_{L^2(G)} \leq |G|^{1/2 - \kappa/2}.\]
The main tool here will be the non-concentration hypothesis.

\emph{Phase II.}  The aim here is to show that, by some later time $n \sim C_1 \log |G|$, the measure $\mu^{(n)}$ is extremely spread out in the sense that \[ \Vert \mu^{(n)} \Vert_{L^2(G)} \leq |G|^{\kappa/10}.\]
The main tool here will be the product theorem hypothesis.

\emph{Phase III.} Finally, we show that at some still later time $n \sim C_2 \log |G|$, we have the desired bound
\[ \Vert \mu^{(n)} - 1 \Vert_{L^\infty(G)} \leq |G|^{-10}.\]
The main tool here will be the quasirandomness hypothesis.

The constants $0 < C_0 < C_1 < C_2$ will depend only on $k,\kappa,\Lambda$, and the function $\delta'(\cdot)$.\vspace{11pt}

\emph{Remark.} Bourgain and Gamburd organise things slightly differently, deducing a spectral gap directly from the roughly uniform measure resulting from Phase II rather than the highly uniform one resulting from Phase III. To achieve this one must use the fact that all eigenvalues of $T_S$ occur with high multiplicity, a fact which follows easily from the quasirandomness of $G$. This form of the argument stems from an idea of Sarnak and Xue \cite{sarnak-xue}. Which argument one prefers is definitely a matter of taste, and on some not-especially-deep level they are equivalent.

We turn now to the discussion of the three phases in turn. Throughout the discussion that follows, the parameters $k, \kappa,\Lambda$ are as in the statement of Proposition \ref{machine}.\vspace{11pt}

\emph{Phase I.} Recall that our task is to show that there is some $C_0$ such that, if $n \geq C_0 \log |G|$, we have $\Vert \mu^{(n)} \Vert_{L^2(G)} \leq |G|^{1/2 - \kappa/2}$. This is actually a very simple task given the non-concentration estimate (iii), applied to the special case $H = \{\id\}$. This implies that $\Vert \mu^{(n)} \Vert_{L^\infty(G)} \leq |G|^{1 - \kappa}$, and so by a trivial instance of H\"older's inequality
\begin{equation}\label{conv-smooth}
 \Vert \mu^{(n)} \Vert_{L^2(G)} \leq \Vert \mu^{(n)} \Vert_{L^\infty(G)}^{1/2} \Vert \mu^{(n)} \Vert_{L^1(G)}^{1/2} = \Vert \mu^{(n)} \Vert_{L^\infty(G)}^{1/2} \leq |G|^{1/2 - \kappa/2},
 \end{equation}
as required.\vspace{11pt}

\emph{Phase II.} This is the heart of the argument in some sense. Suppose that $\nu$ is a symmetric probability measure on $G$ such that \begin{equation}\label{non-conc}\nu(xH) \leq |G|^{-\kappa}\end{equation} for all cosets $xH$ of proper subgroups $H <  G$. We begin by noting that this automatically implies the same estimate for any convolution $\nu \ast \mu$, and in particular we have
\begin{equation}\label{nu-m-est} \nu^{(m)}(xH) \leq |G|^{-\kappa}\end{equation} for all $m$. Indeed, note that for any $x_0$ we have
\begin{align*} \sup_x \nu(x H) & \geq \E_x \mu(x) \E_y \nu(y) 1_H (x_0^{-1}x y) \\ & = \E_x \mu(x) \E_z \nu(x^{-1} x_0 z) 1_H(z) \\ & = \E_z (\E_x \mu(x) \nu(x^{-1} z))1_{x_0H}(z) \\ & = (\mu \ast \nu)(x_0 H),\end{align*}
and so
\[ \sup_x \nu(x H) \geq \sup_x (\mu \ast \nu)(xH).\]
The non-concentration estimate, item (iii) of Proposition \ref{machine}, states that some measure $\nu = \mu^{(n)}$ satisfies \eqref{non-conc}, for some $n \leq \Lambda \log |G|/2$.  Indeed, if $\mu^{(n)}(xH) > |G|^{-\kappa}$ for some coset $xH$, then by symmetry we also have
$\mu^{(n)}(Hx^{-1}) > |G|^{-\kappa}$, and hence by convolution $\mu^{(2n)}(H) > |G|^{-2\kappa}$, which will contradict (iii) for a suitable choice of $n$. By the conclusion \eqref{conv-smooth} of Phase I and the monotonicity of the $L^2(G)$ norms $\Vert \mu^{(n)} \Vert_{L^2(G)}$ we may, by increasing $\Lambda$ to $O_{\Lambda,\kappa}(1)$ if necessary, assume that $\nu$ additionally satisfies
\begin{equation}\label{ell2-upper} \Vert \nu \Vert_{L^2(G)} \leq |G|^{1/2 - \kappa/2}.\end{equation}

To establish Phase II we need only show that there is $m = O_{\kappa, k,\Lambda}(\log |G|)$ such that $\Vert \nu^{(m)} \Vert_{L^2(G)} \leq |G|^{\kappa/10}$. Let $\eta > 0$ be a quantity to be chosen later, depending on $\kappa, k,\Lambda$. Consider the sequence of measures $\nu$, $\nu^{(2)}$, $\nu^{(4)},\dots, \nu^{(2^{m_0})}$. If $m_0 = O_{\eta}(1)$ is large enough then by the pigeonhole principle there is some $j$ such that
\[ \Vert \nu^{(2^{j+1})} \Vert_{L^2(G)} \geq |G|^{-\eta} \Vert \nu^{(2^j)} \Vert_{L^2(G)}.\]
Writing $\mu \coloneqq  \nu^{(2^j)}$, this of course implies that
\[ \Vert \mu \ast \mu \Vert_{L^2(G)} \geq |G|^{-\eta} \Vert \mu \Vert_{L^2(G)}.\]
We now apply a weighted version of the noncommutative Balog-Szemer\'edi-Gowers theorem from additive combinatorics, stated and proved in Appendix \ref{bsg-app-sec}. This tells us that there is a $|G|^{C\eta}$-approximate group $A \subseteq G$ with
\begin{equation}\label{a-bound}|A| \leq |G|^{1 + C\eta}/\Vert \mu \Vert_{L^2(G)}^2\end{equation} and some $x \in G$ such that
\begin{equation}\label{a-bd-2}\mu(xA) \geq |G|^{-C\eta}.\end{equation}
As a consequence of the Product Theorem (item (ii) in the list of hypotheses of Proposition \ref{machine}), one of the following three alternatives must hold provided that $\eta$ is chosen sufficiently small:
\begin{enumerate}
\item $|A| < |G|^{\kappa/2}$;
\item $|A| > |G|^{1 - \kappa/10}$;
\item $A$ generates a proper subgroup of $G$.
\end{enumerate}
We shall also assume, since this will be required below, that $\eta < c\kappa$ for some small constant $c > 0$. We now examine the above three possibilities in turn.

If (i) holds then, by \eqref{a-bd-2} and the pigeonhole principle, there is some element $x \in G$ with \[ \mu(x) \geq |G|^{-C\eta - \kappa/2} \geq |G|^{-\kappa}.\] This contradicts the non-concentration hypothesis  (with $H = \{\id\}$ being the trivial subgroup).

If (ii) holds then, by \eqref{a-bound}, we obtain \[ \Vert \mu \Vert_{L^2(G)} \leq |G|^{C \eta + \kappa/20} \leq |G|^{\kappa/10},\] which is the desired outcome of Phase II.

Finally, suppose that (iii) holds. Then, by \eqref{a-bd-2}, we have \[ \mu(xH) \geq \mu(xA) \geq |G|^{-C\eta} \geq |G|^{-\kappa},\] where $H = \langle A\rangle$ is the group generated by $A$, and by assumption $H$ is proper, and so once again the non-concentration estimate is violated. This concludes the analysis of Phase II. \vspace{11pt}

\emph{Phase III.} Our task here is to prove the following lemma.

\begin{lemma}\label{phase3} Suppose that $\nu$ is a probability measure on $G$ with $\Vert \nu \Vert_{L^2(G)} \leq |G|^{\kappa/100}$, such as the one output by Phase II above. Then for sufficiently large $m = O_{\kappa}(1)$, we have
\begin{equation}\label{to-prove} \Vert \nu^{(m)} - 1 \Vert_{L^\infty(G)} \leq |G|^{-10}.\end{equation}
\end{lemma}

\begin{proof}
We use the quasirandomness of $G$, and in particular an inequality of Babai, Nikolov and Pyber \cite{babai-nikolov-pyber}, related to earlier work of Gowers \cite{gowers}. The inequality states\footnote{Note that Babai, Nikolov and Pyber use the counting measure on $G$ rather than the normalised counting measure as we do. This is why we omit a factor of $\sqrt{n}$ from their estimate.} that for any probability measures $\nu_1, \nu_2$ on $G$ we have
\begin{equation}\label{bnp}
 \Vert \nu_1 \ast \nu_2 - 1 \Vert_{L^2(G)} \leq \sqrt{\frac{1}{d_{\min}(G)}} \Vert \nu_1 - 1 \Vert_{L^2(G)} \Vert \nu_2 - 1 \Vert_{L^2(G)},
 \end{equation}
  where $d_{\min}(G)$ is the smallest dimension of a nontrivial representation (over $\C$) of $G$. By assumption we have $d_{\min}(G) \geq |G|^{\kappa}$, and so for any probability measure $\nu$ we have
\[ \Vert \nu^{(2)} - 1 \Vert_{L^2(G)} \leq |G|^{-\kappa/2} \Vert \nu - 1 \Vert_{L^2(G)}.\]
Applying this repeatedly, after $O_{\kappa}(1)$ convolutions we will end up with some $m_0 = O_{\kappa}(1)$ such that
\[ \Vert \nu^{(m_0)} - 1 \Vert_{L^2(G)}  \leq |G|^{-5}.\]
Finally, a single application of the Hausdorff-Young inequality allows us to conclude that
\begin{align*} \Vert \nu^{(2m_0)} - 1 \Vert_{L^\infty(G)} & = \Vert (\nu^{(m_0)} - 1) \ast (\nu^{(m_0)} - 1) \Vert_{L^\infty(G)} \\ & \leq \Vert \nu^{(m_0)} - 1 \Vert_{L^2(G)} \leq |G|^{-10},\end{align*}
as required.
\end{proof}

\section{A locally commutative subgroup of the affine group of the plane}\label{affine-sec}

In this section, we consider the problem of finding a strongly dense free subgroup in the group of (special) affine transformations of the plane, i.e. the group $\G(k)=k^2\rtimes \SL_2(k) $. This is a non-semisimple but perfect algebraic group.  The purpose of this appendix is to establish the existence of strongly dense free subgroups in this group, as this will be needed to obtain strongly dense free subgroups of $\Sp_4$ in the next section.

More precisely, we establish the following.

\begin{theorem}[Strongly dense affine groups]\label{affdense} Let $k$ be a field which is not locally finite\footnote{A field $k$ is \emph{locally finite} if every finite subset of $k$ generates a finite subfield.  In particular, locally finite fields are necessarily positive characteristic and countable.}.  Then $\G(k)\coloneqq k^2\rtimes \SL_2(k)$ contains a strongly dense free subgroup on two generators, i.e. a free subgroup such that every non-abelian proper subgroup is Zariski-dense in $\G$.
\end{theorem}

For the rest of this section, $k$ and $\G$ are as in Theorem \ref{affdense}.  We begin with the following classification of Zariski closures of free subgroups of $\G$.

\begin{lemma}\label{freeg}  Let $F$ be a free group on two generators in $\G$, and let $\HH$ be the Zariski closure of $F$.  Then either $\HH$ is equal to all of $\G$, or there exists an element $x\in k^2$ such that $\HH$ is equal to the stabiliser $\Stab(x) \coloneqq  \{ g \in \G(k): gx = x \}$ of $x$ \textup{(}using the obvious action of the affine group $\G$ on the plane $k^2$\textup{)}.
\end{lemma}

\begin{proof}  We may assume without loss of generality that $\HH$ is proper.  We project $\HH$ to $\SL_2$ under the quotient map from $\G$ to $\SL_2$.  This is a closed subgroup of $\SL_2$, and is thus either virtually solvable or all of $\SL_2$.  In the former case, $\HH$ is virtually solvable and thus cannot contain a free group; and so $\HH$ projects to all of $\SL_2$.

Now consider the intersection of $\HH$ with $k^2$. Since $\HH$ projected surjectively onto $\SL_2$, $\HH \cap k^2$ must be a closed subgroup of $k^2$ that is normalised by the $\SL_2$ action, and is thus either trivial or all of $k^2$.  In the latter case, we have $\HH = \G$, so we may assume that $\HH \cap k^2 = \{0\}$, and so the projection from $\HH$ to $\SL_2$ is an isomorphism of algebraic groups.  In particular, $\HH$ induces a class in the first cohomology group of $\SL_2$.  However we have the following well-known lemma. 

\begin{lemma}[Vanishing of cohomology] The first cohomology group of $\SL_2$ acting on its natural module $\overline{k}^2$ is trivial.
\end{lemma}

\begin{proof} For the convenience of the reader, we provide a short proof when $\ch(k) \neq 2$ and we refer the reader to
 \cite[II.4.13]{RAG} for the general case (which is not needed for our application to $\Sp_4$ in characteristic $3$). We need to show that any closed subgroup $\HH$ of $\G$ which maps isomorphically to $\SL_2$ under the quotient homomorphism stabilizes a point in $\overline{k}^2$, and is thus conjugate to the stabilizer of $0$ in $\overline{k}^2$. Since $\SL_2$ has a non trivial centre $\{\pm 1\}$, $\HH$ must contain an element of order $2$. The eigenvalues of the linear part of this element must be equal to $-1$, hence this element must fix a (unique) point in $\overline{k}^2$. This fixed point must therefore be preserved by the centralizer of this involution, hence by all of $\HH$. We are done.
\end{proof}

It follows immediately that $\HH$ and our free subgroup must fix a point in $\overline{k}^2$. Since fixed points of elements of $\G(k)$ are defined over $k$, they must lie in $k^2$, and the claim follows.
\end{proof}

As a consequence of Lemma \ref{freeg}, Theorem \ref{affdense} may now be reformulated as follows.

\begin{proposition}\label{loc-commut} Let $k$ be a field which is not locally finite.  Then there exists a non-abelian free subgroup on two generators in $\G(k)$ whose action on $k^2$ is \emph{locally commutative} in the sense that the stabilizer of every point in $k^2$ is commutative.
\end{proposition}

When $k=\R$, this proposition  was first proven by K. Sato in \cite{sato}, and answered a long-standing open problem (see \cite[Problem 19.c, p.233]{wagon} and \cite{sato}) regarding paradoxical decompositions of the affine plane. Indeed, as is well-known (see \cite[Ch. 4]{wagon}), Proposition \ref{loc-commut} implies that one can duplicate the plane $k^2$ by affine maps using only four pieces. In other words, one can write $\k^2=A_1 \cup B_1 \cup A_2 \cup B_2$ a partition of $\k^2$ into disjoint (nonmeasurable) sets and find affine maps $g_1,g_2 \in \k^2\rtimes \SL_2(\k) $ such that $\R^2=A_1 \ \cup g_1B_1=A_2 \cup g_2B_2$.

Sato's proof used a direct computation in the spirit of Hausdorff's original 1914 argument \cite{hausdorff} for the existence of a free subgroup of $\SO_3(\R)$. We will present here a different argument, based on a ping-pong argument, which is valid in arbitrary characteristic.

We now begin the proof of Proposition \ref{loc-commut}. Since fixed points of elements in $\G(k)$ are defined over $k$, we may assume that $k$ is a finitely-generated infinite field. It is well-known that every infinite finitely-generated field embeds in some local field $K$ (i.e. $\R$, $\C$, a finite extension of the field of $p$-adic numbers $\Q_p$, or a field of Laurent series over a finite field $\F_q((t))$) in such a way that $k$ is dense in $K$. Let $|\cdot|$ be an absolute value on $K$ defining the topology of $K$.

We parameterize any given element $g \in \G(k) = k^2 \rtimes \SL_2(k)$ as $g = (c(g), \ell(g))$, thus $c(g) \in k^2$ is the translation part of $g$, $\ell(g) \in \SL_2(k)$ is the linear part, and the action on $k^2$ is given by
$$
g\cdot x=\ell(g) x +c(g).
$$
If $\ell(g)$ is not unipotent (i.e. does not have $1$ as eigenvalue), then $g$ fixes a unique point $x(g)$ on $k^2$ given by the formula $x(g)\coloneqq (1-\ell(g))^{-1}c(g)$.

We perform the following ping-pong type construction.  First, we choose a generic element $h \in \G(k)$, and then choose an element $L \in k$ with $|L|$ sufficiently large depending on $h$ (in particular, we will require $|L| > 1$).  Let $a \in \G(k)$ be the affine map
$$ a \cdot (x,y) \coloneqq  (L^{10} x, L^{-10} y).$$
Thus, $x(a)=0$, and the two $a$-invariant lines of $K^2$ are the $x$-axis and the $y$-axis.  If we introduce the norm $\| (x,y) \| \coloneqq  \max(|x|,|y|)$ on $K^2$ and define the regions
\begin{align*}
U_a^- &\coloneqq \{(x,y) \in K^2 : \|(x,y)\| <L^{-1} \textnormal{ or }  |x| > |L| |y| \}\\
U_{a^{-1}}^- &\coloneqq \{(x,y) \in K^2 : \|(x,y)\| <L^{-1} \textnormal{ or }  |y| > |L| |x| \}\\
U_a^+ &\coloneqq \{(x,y) \in K^2 :  |y| > \max( |L| |x|, |L| ) \}\\
U_{a^{-1}}^+ &\coloneqq \{(x,y) \in K^2 : |x| > \max( |L| |y|, |L| ) \}
\end{align*}
then we observe that we have the inclusions
\begin{align*}
U_a^+ &\subset U_{a^{-1}}^- \\
U_{a^{-1}}^+ &\subset U_{a}^- \\
a(K^2 \setminus U_a^-) &\subset U_a^+\\
a^{-1}(K^2 \setminus U_{a^{-1}}^-) &\subset U_{a^{-1}}^+.
\end{align*}
Furthermore, for $L$ large enough, we have the norm dilation property
$$ \| a p \| > \|p\|$$
for all $p \in K^2 \backslash U_a^-$
and similarly
$$ \| a^{-1} p \| > \|p\|$$
for all $p \in K^2 \backslash U_{a^{-1}}^-$.

Now we define the conjugate $b \coloneqq  h a h^{-1} \in \G(k)$ of $a$ by $h$.  As we chose $h$ to be generic, the fixed point $x(b) = h(x(a)) = h(0)$ of $b$ wil be distinct from that of $a$; furthermore, the two $b$-invariant lines (i.e. the images of the $x$ and $y$ axes under $h$) will not contain $x(a)$ or be parallel to either of the two $a$-invariant lines, and vice versa.  If we then set
\begin{align*}
U_b^- &\coloneqq  h(U_a^-),\\
U_{b^{-1}}^-&\coloneqq  h(U_{a^{-1}}^-),\\
U_b^+&\coloneqq  h(U_a^+), \\
U_{b^{-1}}^+&\coloneqq  h(U_{a^{-1}}^+)
\end{align*}
then by conjugation we have the inclusions
\begin{align*}
U_b^+ &\subset U_{b^{-1}}^- \\
U_{b^{-1}}^+ &\subset U_{b}^- \\
b(K^2 \setminus U_b^-) &\subset U_b^+\\
b^{-1}(K^2 \setminus U_{b^{-1}}^-) &\subset U_{b^{-1}}^+
\end{align*}
and the norm dilation property
$$ \| b p \| > \|p\|$$
for all $p \in K^2 \backslash U_b^-$
and similarly
$$ \| b^{-1} p \| > \|p\|$$
for all $p \in K^2 \backslash U_{b^{-1}}^-$.

Finally, we define the regions
\begin{align*}
\Omega^- &\coloneqq U_a^- \cup U_{a^{-1}}^- \cup U_b^- \cup U_{b^{-1}}^-.\\
\Omega^+ &\coloneqq U_a^+ \cup U_{a^{-1}}^+ \cup U_b^+ \cup U_{b^{-1}}^+ \subset \Omega^-.\\
\end{align*}
Note that for $L$ large enough, the four regions $U_a^+, U_{a^{-1}}^+, U_b^+, U_{b^{-1}}^+$ in $\Omega^+$ are disjoint.  Also,
since by construction $x(a)$ (resp. $x(b)$) is not on the $b$-invariant lines (resp. $a$-invariant lines), taking $|L|$ large enough, we may also arrange so that the six intersections $U_a^- \cap U_{a^{-1}}^-$, $U_a^- \cap U_b^-$, $U_a^- \cap U_{b^{-1}}^-$, $U_{a^{-1}}^- \cap U_{b^{-1}}^-$, $U_{a^{-1}}^- \cap U_b^-$, and $U_{b^{-1}}^- \cap U_b^-$ are disjoint, and thus any given point of the plane belongs to at most $2$ regions of the form $U_u^-$, where $u$ is one of the four letters $a,a^{-1},b,b^{-1}$; see Figure \ref{fig1}.

\begin{figure}\label{fig1}
\begin{center}
\includegraphics[scale=0.3]{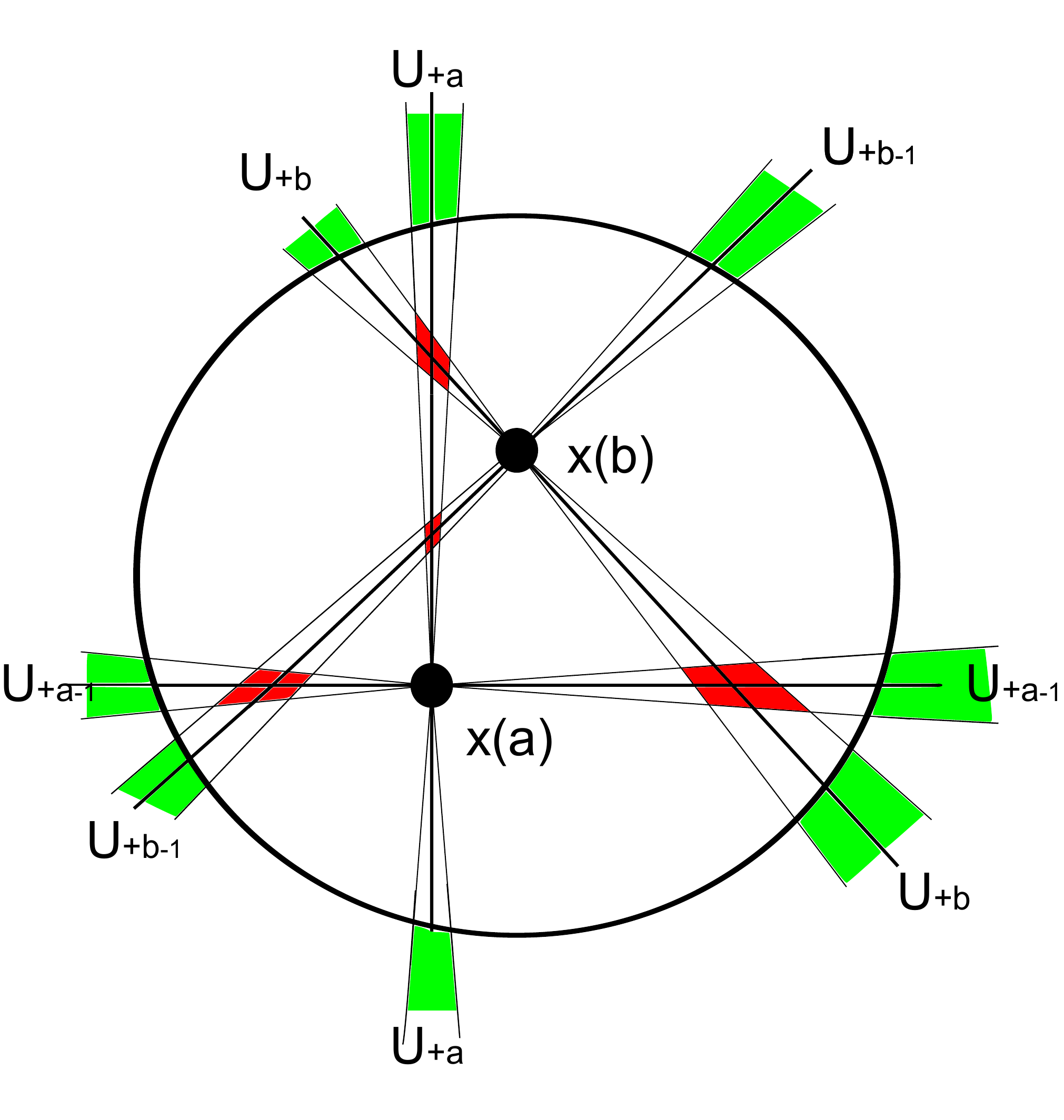}
\caption{The ping-pong table}
\end{center}
\end{figure}

This is a classical ping-pong situation:

\begin{lemma}[Ping-pong lemma]\label{free} The transformations $a$ and $b$ generate a free subgroup.
\end{lemma}

\begin{proof} Let $w$ be a non-trivial reduced word in the free group $F_2$ on two generators.  It suffices to show that the action of $w(a,b)$ on $K^2$ is non-trivial.  However, from the inclusions given above and an easy induction on the length $|w|$ of $w$, we see that $w(a,b)$ maps $K^2 \setminus \Omega^-$ to $U_c^+ \subset \Omega^+\subset \Omega^-$, where $c \in \{a,b,a^{-1},b^{-1}\}$ is the first symbol of $w$, and the claim follows.
\end{proof}

Next, we use the norm dilation property to locate the fixed point of various words $w(a,b)$.  For any non-trivial reduced word $w \in F_2$, let $E(w) \in \{a,b,a^{-1},b^{-1}\}$ denote the last letter of $w$ (i.e. $E(w)$ is the unique $u \in \{a,b,a^{-1},b^{-1}\}$ such that $|wu^{-1}| < |w|$).

\begin{lemma} If $w$ is a reduced word in the free group $F_2$ and if $u = E(w)$, then the fixed point of $w(a,b)$ must belong to $U_u^-$.
\end{lemma}

\begin{proof} If $p \in K^2\setminus U_u^-$, then $\|w(a,b)p\|>\|p\|$, because this norm keeps increasing every time we add a letter. Hence $w(a,b)p \neq p$.  Taking contrapositives, we obtain the claim.
\end{proof}

Recall that each point in the plane belongs to at most $2$ regions of the form $U_u^-$, where $u$ is one of the four letters $a,a^{-1},b,b^{-1}$.  From this and the preceding lemma we conclude:

\begin{lemma}\label{fixed-point} Let $w_1$, $w_2$ and $w_3$ be three reduced words in $F_2$ and assume that their last letters $E(w_1), E(w_2), E(w_3)$ are all distinct. Then $w_1(a,b)$, $w_2(a,b)$ and $w_3(a,b)$ do not have a common fixed point in $K^2$.
\end{lemma}

In order to exploit this lemma, we use following elementary fact from combinatorial group theory.

\begin{lemma}[Nonabelian subgroups of the free group]\label{three-letters} A subgroup $H$ of the free group $F_2$ is nonabelian if and only if there is $g \in F_2$ such that $E(gHg^{-1}) \coloneqq  \{ E(gkg^{-1}): k \in H \} $ has at least $3$ elements.
\end{lemma}

Before giving the proof of this lemma, let us conclude the proof of Theorem \ref{affdense}.

\noindent \emph{Proof of Theorem \ref{affdense}}. We in fact give a proof of Proposition \ref{loc-commut}, which was seen to be equivalent to Theorem \ref{affdense} in the discussion following the statement of that theorem. Let $a,b,K$ be as in the preceding discussion. By Lemma \ref{free} we know that the subgroup $\langle a,b \rangle$ generated by $a,b$ is a free subgroup. Let $p \in K^2$ and let $H_p$ be the stabilizer of $p$ in the free subgroup $\langle a,b \rangle$. If $H_p$ were non abelian, then by Lemma \ref{three-letters} above, we could conjugate $H_p$ by some element $g \in \langle a ,b \rangle$ in such a way that $|E(gH_pg^{-1})| \geq 3$. However $gH_pg^{-1}=H_{g\cdot p}$, so $gH_pg^{-1}$ is a subgroup of $\langle a,b \rangle$ which has a fixed point in $K^2$ and contains three elements $w_1(a,b)$, $w_2(a,b)$ and $w_3(a,b)$ whose last letters are all distinct. This contradicts Lemma \ref{fixed-point} and concludes the proof.
\endproof

It remains to give a proof of the combinatorial group theory lemma. This can be done combinatorially by choosing two non-commuting elements $w_1$ and $w_2$ in $H$ and cyclically conjugating $w_1$ and $w_2$ until three different letters arise. Instead, we present an elegant geometric argument, which was suggested to us by Thomas Delzant.\vspace{11pt}

\begin{proof}[Proof of Lemma \ref{three-letters}] In one word: a graph whose fundamental group is not $\Z$ must have a vertex of degree at least $3$. Let us now explain this sentence. The free group $F_2$ is the fundamental group of the wedge $X$ of two circles. We can label the first circle by $a$ and the second by $b$ and draw an arrow on each circle. The universal cover $\widetilde{X}$ of $X$ is a tree with edges labeled by arrows and letters $a$ or $b$ and $F_2$ acts freely on it by graph automorphisms. Every vertex has two incoming edges and two outgoing edges. The quotient graph $\widetilde{X}/H$ is also a graph labeled in the same manner and its fundamental group is isomorphic to $H$. More precisely, if we fix the base point $x_0$ of $X$ (to be the intersection of the two circles), a lift $\widetilde{\widetilde{x_0}}$ in $\widetilde{X}$ and its projection $\widetilde{x_0}$ in $\widetilde{X}/H$, then homotopy classes of loops based at $\widetilde{x_0}$ are in one-to-one correspondence with elements of $H$. Non-backtracking paths (i.e. paths not passing through the same edge consecutively) from $\widetilde{x_0}$ to $\widetilde{x_0}$ correspond to elements of $H$ as one can read off the letters appearing in the words in $H$ as the sequence of letters read as one follows the path.

If $\widetilde{x_1}$ is another vertex of $\widetilde{X}/H$, then  the non-backtracking loops starting and ending at $\widetilde{x_1}$ are in one-to-one correspondence with the reduced words in $gHg^{-1}$, where $g$ is any element of the free group $F_2$ such that $g\cdot\widetilde{x_0}=\widetilde{x_1}$. Now the graph $\widetilde{X}/H$ may be pruned by deleting all branches with no loop: the resulting graph will have isomorphic fundamental group and in particular, if $H$ is non abelian, then the graph will not be homotopic to a circle and hence will admit a vertex, say $\widetilde{x_1}$, of degree at least $3$. Thus there will exist three non-backtracking paths around $\widetilde{x_1}$ with distinct last edge.
This means that when considering all reduced words appearing in $gHg^{-1}$, at least three different letters will arise as the last letter of a word. We are done. \end{proof}

\emph{Remark.} Our proof in fact shows slightly more than the claim of Theorem \ref{affdense}, namely that every finitely generated Zariski-dense subgroup $\Gamma$ of $\G(k)$ contains a strongly dense free subgroup. This answers positively Problem 1 in our previous paper \cite{bggt2} for the affine group of the plane.  This is because one can use the Zariski-density hypothesis to  locate elements $a,b$ in the group $\Gamma$ that obey the properties used in the above construction (namely, that they are conjugate with one large eigenvalue, and have distinct fixed points and distinct invariant lines).  We omit the details.

\section{Strongly dense subgroups of $\Sp_4$}\label{sp4-app}



In this section,  we complete the proof of the main theorem of \cite{bggt2} by proving the existence of strongly dense subgroups
of $\Sp_4(\k)$ when $\k$ is an uncountable algebraically closed field of characteristic $3$. This case had been left aside in \cite{bggt2}, because our proof for all other groups broke down in that particular case (and only in that case).  Actually, our arguments here will work in all characteristics other than $2$.  Our main tool will be the results of Appendix \ref{affine-sec}.  As in \cite{bggt2}, one can reduce to the case of finite transcendence degree.

Henceforth we fix an uncountable algebraically closed field $\k$ with characteristic not $2$.  Let $\G(\k)=\Sp_4(\k)$. Recall that a subset of $\G(\k)$ is called \emph{generic} if its complement is contained in at most countably many proper algebraic subvarieties.  We will need the following general fact:

\begin{lemma}\label{step3-lem}
Suppose that $\G(k)$ is a semisimple algebraic group and that $w,w' \in F_2$ are noncommuting words.
Then for generic $(a,b) \in \G(k) \times \G(k)$ the elements $w(a,b), w'(a,b)$ generate a group whose
closure $\overline{\langle w(a,b), w'(a,b)\rangle}$ is infinite and is either \textup{(i)}  $\G(k)$ or \textup{(ii)} a proper semisimple subgroup
$\HH < \G(k)$ with $\rk(\HH) = \rk(\G)$.
\end{lemma}

\begin{proof}  See \cite[Lemma 2.7]{bggt2}.  The key is to show that generically, $w(a,b)$, $w'(a,b)$, and $[w(a,b),w'(a,b)]$ each generate a Zariski-dense subgroup of a maximal torus (see \cite[Lemma 2.6]{bggt2}).
\end{proof}

\begin{lemma}  \label{words}  Let $w_1, w_2$ be noncommuting words in $F_2$, the free group on two generators.
    Then the set of
$(a,b) \in \G(\k) \times \G(\k)$ with $\langle w_1(a,b), w_2(a,b)
\rangle$ Zariski-dense in $\G(\k)$  is generic.
\end{lemma}

\begin{proof}   By Lemma \ref{step3-lem}, the set of pairs $(a,b)$
such that the Zariski closure of $\langle w_1(a,b), w_2(a,b) \rangle$ is semisimple of
rank $2$ is generic.  Since the only proper semisimple rank $2$ subgroups of $\G(\k)$
are the stabilizers of a nondegenerate $2$-space in $\k^4$, we see that either the result
holds or for {\it every} $(a,b)$,  $ \langle w_1(a,b), w_2(a,b) \rangle$
stabilizes some
$2$-dimensional subpace (in the natural $4$-dimensional representation).
This follows as in \cite[Lemma 3.4(i)]{bggt2}.

Let $J$ be the derived subgroup of the stabilizer of $1$-dimensional subspace.
So $J$ is the centralizer of a long root subgroup $Z$.    Let $Q$ be
the unipotent
radical of $J$.  Then $Z = [Q,Q]$ and $J/Z \cong \k^2\rtimes \SL_2(\k)$.
By Theorem \ref{affdense}, there exist $(a,b) \in J \times J$ so that
$Z \overline{\langle w_1(a,b), w_2(a,b) \rangle} = J$.  Since
$Z$ is contained in the Frattini subgroup of $J$, this implies that
$J=\overline{\langle w_1(a,b), w_2(a,b) \rangle}$.   However, $J$ leaves no invariant
no $2$-dimensional space, whence the result.
\end{proof}

Since there are only countably many pairs of words in $F_2$, we immediately obtain the following.

\begin{theorem} \label{sp4}  The set of
$(a,b) \in \G(\k) \times \G(\k)$ with  $\langle a, b \rangle$ free and strongly dense in $\Sp_4(\k)$
is generic.  In particular, at least one such pair $(a,b)$ exists.
\end{theorem}

\section{Equivalence of one and two sided expansion}\label{equiv-sec}

The purpose of this appendix is to establish that ``combinatorial expansion implies spectral expansion''. Though this is not needed elsewhere in the paper, it may be of interest to readers. 

\begin{proposition}\label{spok}  Let $G$ be a finite group, let $k \geq 1$, and let $x_1,\ldots,x_k \in G$ be combinatorially expanding in the sense that
\begin{equation}\label{combo}
 |(A x_1 \cup A x_1^{-1} \cup A x_2 \cup A x_2^{-1} \cup \dots \cup A x_k \cup Ax^{-1}_k) \setminus A| \geq \eps' |A|
 \end{equation}
for every set $A \subseteq G$ with $|A| \leq |G|/2$, and some $\eps'>0$.  Suppose also that there does not exist an index two subgroup $H$ of $G$ which is disjoint from the $x_1,\ldots,x_k$.  Then $\{x_1,\ldots,x_k\}$ is $\eps$-expanding \textup{(}in the sense of Def. \ref{specex}\textup{)} for some $\eps>0$ depending only on $\eps',k$.
\end{proposition}

In other words: if a Cayley graph is an expander graph, and is not bi-partite, then the averaging operator $T$ defined in Definition \ref{specex} has spectrum not only bounded away from $1$ but also from $-1$. This feature is quite special to Cayley graphs and does not hold for arbitrary regular graphs.

Note that the condition that an index two subgroup $H$ disjoint from $x_1,\ldots,x_k$ does not exist is necessary, since otherwise the convolution operator $T$ defined in Definition \ref{specex} has an eigenvalue at $-1$ with eigenfunction $1_H - 1_{G \backslash H}$, and the Cayley graph is bi-partite. It is not hard to show that this is in fact the only way that $T$ can attain an eigenvalue at $-1$ if $x_1,\ldots,x_k$ generates $G$.  When $G$ is a nonabelian simple group, there are no index two subgroups $H$, and so we see that spectral expansion and combinatorial expansion are essentially equivalent in this setting.  From the discrete Cheeger inequality (see \cite{alon,alon-milman,dodziuk}, or \cite[Prop. 4.2.4]{lubotzky-book}) the combinatorial expansion hypothesis already gives one side of $\eps$-expansion, in that the operator $T$ has spectrum in $[-1,1-\eps]$ for some $\eps>0$ depending only on $\eps',k$; the novelty in Proposition \ref{spok} is that spectrum can also be excluded in the interval $(-1,-1+\eps]$.  On the other hand, our argument relies heavily on the Cayley graph structure, whereas the discrete Cheeger inequality is valid for arbitrary regular graphs.

We now prove this proposition.  Let $G,k,x_1,\ldots,x_k,\eps'$ be as in the proposition, and let $\eps>0$ be a sufficiently small quantity (depending on $\eps',k$) to be chosen later.  Write $S \coloneqq  \{ x_1,\ldots,x_k,x_1^{-1},\ldots,x_k^{-1} \}$, thus \eqref{combo} tells us that
\begin{equation}\label{expando}
|SA \backslash A| \geq \eps' |A|
\end{equation}
whenever $A \subset G$ with $|A| \leq |G|/2$.  Since $|SA \backslash A|$ is, up to a multiplicative constant depending on $k$, the number of edges in the Cayley graph connecting $A$ to its complement, we also have the variant estimate
\begin{equation}\label{expando-2}
|SA \backslash A| \geq \eps'' |G \backslash A|
\end{equation}
when $|A| \geq |G|/2$, where $\eps''>0$ depends on $\eps$ and $k$.

We use $o(X)$ to denote a quantity bounded in magnitude by $c_{\eps',k}(\eps) X$, where $c_{\eps',k}(\eps)$ is a quantity depending only on $\eps',k,\eps$ that goes to zero as $\eps \to 0$ for all fixed $\eps',k$.  Suppose for contradiction that $T$ has a non-trivial eigenvalue outside of $[-1+\eps,1-\eps]$, so that $T^2$ has a non-trivial eigenvalue in $[(1-\eps)^2, 1]$.  Applying the discrete Cheeger inequality \cite{alon,alon-milman,dodziuk} to the (weighted) Cayley graph associated to $S^2$, we then can find a non-empty set $A \subset G$ with $|A| \leq |G|/2$ such that

$$ |S^2 A \backslash A| = o( |A| ).$$

In particular $|S(A \cup SA) \backslash (A \cup SA)| = o(|A|)$, which by \eqref{expando}, \eqref{expando-2} forces $|A \cup SA| = |G| - o(|A|)$ or $|A \cup SA| = o(|A|)$. The latter is not possible for $\eps$ small enough, hence $|A \cup SA| = (1-o(1)) |G|$; since $|SA| \leq |S^2 A| \leq |A| +o(|A|)$ and $|A| \leq |G|/2$, we conclude that $|A| = (1/2 - o(1))|G|$.  Also we have
$$ |S B \Delta B| = o(|G|)$$
for $B = A \cap SA$, which by another application of \eqref{expando} forces $|A \cap SA| = o(|G|)$.  In particular, for any $s \in S$, since $|sA| = |A| = (\frac{1}{2}-o(1))|G|$, we conclude that $sA$ is nearly the complement of $A$ in the sense that
\begin{equation}\label{soo}
 |sA \Delta (G \backslash A)| = o(|G|).
\end{equation}

Let $g \in G$ be arbitrary, then we also have
$$ |sAg \Delta (G \backslash Ag)| = o(|G|).$$
Thus if we write $A_g \coloneqq  A \cap Ag$ and $A'_g \coloneqq  G \backslash (A \cup Ag)$ then
$$ |sA_g \Delta A'_g| = o(|G|)$$
for all $s \in S$, and thus also by symmetry
$$ |sA'_g \Delta A_g| = o(|G|).$$
Thus we have
$$ |SB \Delta B| = o(|G|)$$
for $B$ equal to $A_g \cup A'_g$ and its complement, which by \eqref{expando} forces $|A_g \cup A'_g|$ to equal either $o(|G|)$ or $(1-o(1))|G|$.  Since $|A_g \cup A'_g| = 2 |A \cap Ag|$ and $|A| = (\frac{1}{2}-o(1)) |G|$, we arrive at the following dichotomy: for any $g \in G$, we either have
\begin{equation}\label{a1}
|A \cap Ag| \geq (1-o(1)) |A|
\end{equation}
or
\begin{equation}\label{a2}
|A \cap Ag| \leq o(1) |A|.
\end{equation}
We now use a ``pivoting'' argument similar to that in an old paper of Freiman \cite{freiman-small} (though the use of the ``pivot'' terminology originates in \cite{helfgott-sl3}).  Let $H$ denote the set of all $g \in G$ for which $|A \cap Ag| \geq 0.9 |A|$ (say) holds, which by the above dichotomy implies \eqref{a1} (if $\eps$ is small enough).  Clearly $H$ is symmetric and contains the identity.  Also, if $g, h \in H$, then by the triangle inequality we see that
$$ |A \cap Agh| \geq (1-o(1)) |A|$$
and so $gh \in H$.  Thus $H$ is a subgroup of $G$.  On the other hand, from the estimate
$$ |A|^2 = \sum_{g \in G} |A \cap Ag| = |A| |H| + o(|G|^2)$$
we see that
$$ |H| = (1/2 + o(1)) |G|$$
and hence the index of $H$ in $G$ is exactly $2$ (if $\eps$ is small enough).  Now note that
$$ \sum_{g \in H} |A \cap Ag| = |A \cap H|^2 + |A \backslash (A \cap H)|^2$$
and thus by \eqref{a1}
$$ |A \cap H|^2 + |A \backslash (A \cap H)|^2 = |A|^2 + o(|A|^2)$$
which we can rearrange as
$$ |A \cap H| |A \backslash (A \cap H)| = o( |G|^2 )$$
thus one either $|A \Delta Hg| = o(|G|)$ for one of the two cosets $Hg$ of $G$.    From \eqref{soo} one concludes that
$$ |sH \cap H| = o(|G|)$$
for all $s \in S$, which (for $\eps$ small enough) forces $S \subset G \backslash H$ since any two cosets of $H$ are either equal or disjoint.  But this contradicts the hypotheses of Proposition \ref{spok}, and the claim follows.\vspace{11pt}

\emph{Remark.} There is a slightly different way to prove Proposition \ref{spok} relying more on spectral theory than on combinatorial methods, which we sketch here.  The key observation is that once one has the expansion property \eqref{combo}, then there cannot be two orthogonal eigenfunctions $\phi, \psi$ in the region $[-1,-1+\eps]$ of the spectrum for $\eps$ sufficiently small, basically because the function $\phi \overline{\psi}$ would then have mean zero and be almost $T$-invariant, contradicting the Cheeger inequality\footnote{The argument is actually a bit more complicated than this, because $\phi, \psi$ are \emph{a priori} only bounded in $L^2$ and so $\phi \overline{\psi}$ is controlled only in $L^1$ instead of $L^2$, and one needs additional arguments related to the proof of the Cheeger inequality to address this, but we ignore this issue for the sake of the sketch.}.  Thus, if there is an eigenfunction $\phi$ in this region, then it must be real-valued (up to multiplication by scalars), and every right shift $\phi(\cdot g)$ of this eigenfunction is equal to either $\phi$ or $-\phi$.  This gives a homomorphism from $G$ to $\{-1,+1\}$ whose kernel $H$ is an index two subgroup, with $\phi$ being a multiple of $1_H - 1_{G \backslash H}$, and as $\phi$ is an eigenfunction with eigenvalue close to $-1$ one can then easily deduce that $x_1,\ldots,x_k \in G\backslash H$.

\textsc{Acknowledgments.} EB is supported in part by the ERC starting grant 208091-GADA.
BG was, for some of the period during which this work was being carried out, a Fellow at the Radcliffe Institute at Harvard. He is very happy to thank the Institute for providing excellent working conditions. More recently he has been supported by ERC starting grant 274938 \emph{Approximate Algebraic Structure and Applications}. RG is supported by NSF grant DMS-1001962 and Simons Foundation Fellowship  224965.
   He also thanks the Institute for Advanced Study, Princeton.
TT was supported by a grant from the MacArthur Foundation, by NSF grant DMS-0649473, the NSF Waterman award, and a Simons Investigator award during portions of this research.

The authors thank the anonymous referee for useful suggestions and corrections.

\end{document}